\documentclass[11pt,a4paper,reqno]{amsart}

\usepackage{color}
\usepackage{amsmath,enumitem}
\usepackage{amsfonts}
\usepackage{amssymb}
\usepackage{amsbsy}
\usepackage{amscd}
\usepackage{amsthm}
\usepackage{array,calc,youngtab}
\usepackage{mathrsfs}
\usepackage[english]{babel}
\usepackage{graphicx}
\usepackage[all]{xy}
\usepackage{mathtools}
\setlist{nosep}

\usepackage[active]{srcltx}
\usepackage[parfill]{parskip}

\newtheorem{theorem}{Theorem}[section]
\newtheorem{lemma}[theorem]{Lemma}
\newtheorem{definition}[theorem]{Definition}
\newtheorem{corollary}[theorem]{Corollary}
\newtheorem{proposition}[theorem]{Proposition}
\newtheorem{remark}[theorem]{Remark}

\newcommand{\Hom}{{\mathrm{Hom}}}

\newcommand{\TKK}{{\mathrm{TKK}}}
\newcommand{\Ad}{{\mathrm{Ad}}}
\newcommand{\diff}{{\mathrm{d}}}
\newcommand{\Lie}{{\mathrm{Lie}}}
\newcommand{\im}{{\mathrm{im}}}

\newcommand{\eins}{\leavevmode\hbox{\small1\kern-3.8pt\normalsize1}}

\newcommand{\Ann}{{\rm Ann} }

\newcommand{\ev}{{\mathrm{ev}}}

\newcommand{\mR}{\mathbb{R}}
\newcommand{\mC}{\mathbb{C}}
\newcommand{\mF}{\mathbb{F}}

\newcommand{\mN}{\mathbb{N}}
\newcommand{\mE}{\mathbb{E}}
\newcommand{\mZ}{\mathbb{Z}}

\newcommand{\mg}{\mathfrak{g}}

\newcommand{\mA}{\mathbb{A}}

\newcommand{\fg}{{\mathfrak g}}

\newcommand{\istr}{\mathfrak{istr}}
\newcommand{\Inn}{{\rm Inn}}

\newcommand{\cD}{\mathcal{D}}

\newcommand{\cH}{\mathcal{H}}

\newcommand{\cP}{\mathcal{P}}

\newcommand{\cC}{\mathcal{C}}

\newcommand{\cS}{\mathcal{S}}

\newcommand{\cO}{\mathcal{O}}

\newcommand{\cJ}{\mathcal{J}}

\newcommand{\cB}{\mathcal{B}}

\newcommand{\oa}{\bar{0}}
\newcommand{\ob}{\bar{1}}

\newcommand{\End}{{\rm End}}

\newcommand{\id}{{\rm id}}

\newcommand{\kerorbit}{C}

\newcommand{\mk}{\mathfrak{k}'}
\newcommand{\ck}{\mathfrak{k}}

\newcommand{\x}{z}

\DeclarePairedDelimiter\abs{\lvert}{\rvert}%
\DeclarePairedDelimiter\norm{\lVert}{\rVert}%
\makeatletter
\let\oldabs\abs
\def\abs{\@ifstar{\oldabs}{\oldabs*}}
\let\oldnorm\norm
\def\norm{\@ifstar{\oldnorm}{\oldnorm*}}
\makeatother

\newcommand{\minus}{\scalebox{0.9}{{\rm -}}}
\newcommand{\plus}{\scalebox{0.6}{{\rm+}}}

\begin{document}

\title{A minimal representation of the orthosymplectic Lie supergroup}

\author{Sigiswald Barbier and Jan Frahm}
\date{}

\begin{abstract}
We construct a minimal representation of the orthosymplectic Lie supergroup $OSp(p,q|2n)$, generalising the Schr\"{o}dinger model of the minimal representation of $O(p,q)$ to the super case. The underlying Lie algebra representation is realized on functions on the minimal orbit inside the Jordan superalgebra associated with $\mathfrak{osp}(p,q|2n)$, so that our construction is in line with the orbit philosophy. Its annihilator is given by a Joseph-like ideal for $\mathfrak{osp}(p,q|2n)$, and therefore the representation is a natural generalization of a minimal representation to the context of Lie superalgebras. We also calculate its Gelfand--Kirillov dimension and construct a non-degenerate sesquilinear form for which the representation is skew-symmetric and which is the analogue of an $L^2$-inner product in the supercase.
\end{abstract}

\maketitle
\section{Introduction}

\subsection*{Minimal representations}

The orbit philosophy is a guiding principle in the representation theory of Lie groups and suggests a relation between coadjoint orbits and irreducible unitary representations. For nilpotent groups, or more generally solvable groups, it can be used to establish a bijective correspondence between coadjoint	 orbits and irreducible unitary representations, but already for the semisimple group $SL(2,\mR)$ this correspondence does not cover the whole unitary dual. One of the main problems is the quantisation of nilpotent coadjoint orbits of semisimple groups, which are expected to correspond to rather small unitary representations. Minimal representations are the irreducible unitary representations of semisimple Lie groups which are supposed to correspond to a minimal nilpotent coadjoint orbit. Prominent examples are the Segal--Shale--Weil representation of the metaplectic group $Mp(n,\mR$), which is a double cover of the symplectic group, or the more recently studied minimal representation of $O(p,q)$.

More technically, a unitary representation of a simple real Lie group $G$ is called minimal if the annihilator ideal of the derived representation of the universal enveloping algebra of the Lie algebra of $G$ is the Joseph ideal. The Joseph ideal is the unique completely prime, two-sided ideal in the universal enveloping algebra such that the associated variety is the closure of the minimal nilpotent coadjoint orbit (see~\cite{GS}). Minimal representations have been constructed in various different ways, algebraically, analytically, or through Howe's theta correspondence.

\subsection*{$L^2$-models}

For a minimal representation, the Gelfand--Kirillov dimension which measures the size of an infinite-dimensional representation attains its minimum among all infinite-dimensional unitary representations. Therefore, explicit geometric realisations of minimal representations are expected to have large symmetries and allow interactions with other mathematical areas such as conformal geometry, integral operators or special functions (see~\cite{Ko,KM2,HKMM}). In every known realisation, some aspects of the representations are rather clear to describe and some are more subtle. One realisation in which for instance the invariant inner product is particularly easy to see is the $L^2$-model (also called Schr\"{o}dinger model) which is due to Vergne--Rossi~\cite{VR}, Dvorsky--Sahi~\cite{DS} and Kobayashi--{\O}rsted~\cite{KO}. Here the representation is realized on the Hilbert space $L^2(C)$ where $C$ is a homogeneous space for a subgroup of $G$. The three constructions in \cite{VR,DS,KO} are different in nature, and only more recently a unified construction of $L^2$-models of minimal representations was developed in \cite{HKM}, using the framework of Jordan algebras. The approach consists of the following steps:
\begin{itemize}
\item Start from a simple real Jordan algebra.
\item Consider the Tits--Kantor--Koecher (TKK) Lie algebra of the Jordan algebra.
\item Construct a representation of the TKK algebra on functions on the Jordan algebra.
\item Identify a minimal orbit of the structure group of the Jordan algebra and restrict the representation to functions on this orbit.
\item Find an admissible subrepresentation which integrates to the conformal group.
\item Show that this representation is unitary with respect to an $L^2$-inner product on the minimal orbit.
\item Show that the constructed representation is indeed a minimal representation.
\end{itemize}

We remark that the indefinite orthogonal groups $G=O(p,q)$ are special among the cases discussed above, since their corresponding minimal representations are in general neither spherical nor highest/lowest weight representations. This makes them harder to construct than in the remaining cases, but as a consequence their $L^2$-models allow a richer analysis.

\subsection*{Minimal representations of Lie supergroups}

Supersymmetry is a framework introduced in the seventies to consider bosons and fermions at the same level \cite{SS,WZ}. Lie supergroups and Lie superalgebras are the mathematical concepts  underlying supersymmetry. Since the ingredients of the orbit method also exist in the super case, it is expected that the orbit method is a useful tool also in the study of irreducible representations of Lie supergroups~\cite[Chapter 6.3]{Kirillov}. For example, the orbit method provides a classification of irreducible unitary representations of nilpotent Lie supergroups (see~\cite{Salmasian, NS}). With this perspective in mind, it is natural to ask for a super version of minimal representations.

The goal of this paper is to construct a minimal representation of the orthosymplectic Lie supergroup $OSp(p,q|2n)$ following the same approach as in \cite{HKM}.
Therefore we need the following concepts in the super case:
\begin{itemize}
\item Jordan superalgebras,
\item TKK algebra of a Jordan superalgebra,
\item a representation of the TKK-algebra on functions on the Jordan superalgebra,
\item a minimal orbit to which the representation restricts,
\item an admissible subrepresentation which integrates to the group level,
\item an invariant inner product,
\item minimality of the constructed representation, i.e. show that the annihilator ideal is a Joseph-like ideal. 
\end{itemize}

The notion of Jordan superalgebras is already well-developed (see e.g.~\cite{Kac, Cantarini, MZ, Shtern}). For the TKK-algebra different definitions exist in the literature \cite{Tits, Kantor, Koecher, Kac, Krutelevich}, but it is shown in \cite{BC2} that for the Jordan superalgebra $J$ corresponding to the orthosymplectic Lie superalgebra $\mathfrak{osp}(p,q|2n)$ all different definitions are equivalent.

The next step is to construct a representation of $\mathfrak{osp}(p,q|2n)$ on functions on the Jordan superalgebra $J$. This has been done in \cite{BC}, where for general three-graded Lie superalgebras a family of representations $\pi_\lambda$ depending on a complex parameter $\lambda\in\mC$ was obtained.

The first new feature of this work is the construction of the minimal orbit $C$ in the super setting. Orbits under the action of a Lie supergroup on some supermanifold are delicate objects to handle. Supermanifolds, in contrast to ordinary manifolds, are not completely determined by their points, so in the super case we cannot define an orbit through a point $x$ as all points given by $g\cdot x$ for $g$ in $G$. Instead we will define an orbit as the quotient supermanifold of $G$ and the stabilizer subgroup $G_x$. Using this definition, we can construct a minimal orbit $C$ which in our situation can be characterized by $R^2=0$, where $R^2=\sum_{ij} x_i \beta^{ij} x_j$, with $\beta$ the defining orthosymplectic metric of our Jordan superalgebra. We then show that for precisely one value of the parameter $\lambda$ the representation $\pi_\lambda$ considered in \cite{BC} restricts to a representation $\pi_C$ on this minimal orbit.

The main part of this paper is the integration of the obtained Lie superalgebra representation $\pi_C$ to a Lie supergroup representation. For this we make use of the theory of Harish-Chandra supermodules developed in \cite{Alldridge}. The key point here is to construct an admissible submodule of $\pi_C$, which then, by the general theory in \cite{Alldridge}, integrates to a representation of the Lie supergroup $OSp(p,q|2n)$. This submodule is generated by a superversion of a K-Bessel function, resembling the K-Bessel function in \cite{HKM} (see also \cite{DS,KO}). In contrast to the classical case, we have to work harder to show that this submodule is indeed admissible. In \cite{HKM} some unitarity properties were implicitly used to do this, but these tools are not available in the supersetting. Instead we find an explicit decomposition of the submodule generated by the K-Bessel function in terms of radial functions and spherical harmonics. This decomposition is even new in the classical case and gives explicit formulas for \emph{all} K-finite vectors in the minimal representation of $O(p,q)$. To derive the decomposition of the Harish-Chandra module we have to go through several technical computations which are responsible for the length of this paper.

A priori, we know that the representation of $OSp(p,q|2n)$ we construct cannot be unitary, since it was shown in \cite[Theorem 6.2.1]{NS} that there exists no unitary representations of $OSp(p,q|2n)$ if $p, q$ and $n$ are all different from zero. This shows that minimal representations of Lie supergroups cannot be expected to be unitary in the usual sense. It is our hope that the representation constructed in this paper will be useful to find an appropriate replacement for the notion of unitarity for Lie supergroups.
We remark that in \cite{Michel}, a new definition of Hilbert superspaces and unitary representations using the super version of Krein spaces is introduced, which allows for a more general notion of unitary representations of Lie superalgebras than the one considered in \cite[Theorem 6.2.1]{NS}. However, it seems that our representation is not unitary even with respect to this broadened notion of unitarity. Nevertheless, we are able to define a non-degenerate superhermitian, sesquilinear form for which the representation is skew-symmetric. This sesquilinear form is the analogue of the $L^2$-inner product on the minimal orbit in the classical case.

Finally, we compute the annihilator ideal of our representation and show that it agrees with one of the two Joseph-like ideals constructed in \cite{CSS}. In this sense, our representation is the natural generalization of a minimal representation to the context of Lie superalgebras.

We remark that this construction only works for $p+q$ even, which is the same condition as for the existence of a minimal unitary representation of the Lie group $O(p,q)$.

\subsection*{Relation to other work}

The representation we construct is a natural analogue of the minimal representation of the group $O(p,q)$ (see e.g. \cite{KO}). This highlights the first factor of the even part $O(p,q)\times Sp(2n,\mR)$ of the supergroup $OSp(p,q|2n)$. The second factor $Sp(2n,\mR)$ does not admit a minimal representation, but its double cover $Mp(2n,\mR)$ does, the Segal--Shale--Weil representation. An analogue of this representation in the super context was constructed in \cite{Michel} (see also \cite{Ni} for the corresponding Lie superalgebra representation of $\mathfrak{osp}(p,q|2n)$). Its annihilator ideal is equal to the second Joseph-like ideal constructed in \cite{CSS}.

We further remark that in \cite[Section 5.2]{AS} highest weight representations of the Lie algebra $\mathfrak{su}(p,p|2p)$ are considered. It seems likely that, for a specific parameter, their representation has a subrepresentation which is the analogue of the minimal representation of $\mathfrak{su}(p,p)$.

\subsection*{Structure of the paper}

This paper is organized as follows. In Section~\ref{Section Preliminaries}, we introduce the notation and collect some results needed in the rest of the paper. In Section~\ref{Section associate structures} we present the spin factor Jordan superalgebra $\mathcal{JS}pin_{p-1,q-2|2n}$, and define and compute some Lie superalgebras and groups associated with it. In particular  the Tits--Kantor--Koecher algebra of $\mathcal{JS}pin_{p-1,q-2|2n}$ is given by $\mathfrak{osp}(p,q|2n)$. We also consider the family $\pi_\lambda$ of representations of $\mathfrak{osp}(p,q|2n)$ on functions on the Jordan superalgebra constructed in \cite{BC} depending on a complex parameter $\lambda\in\mC$.

The first new results are contained in Section~\ref{Section minimal orbit} about the minimal orbit. We define a superspace characterized by $R^2=0$. We then show that this gives a well-defined supermanifold and that this supermanifold is the orbit through a primitive idempotent under the action of the structure group on our Jordan superalgebra (see Theorem~\ref{Theorem minimal orbit}).  For a specific value of $\lambda$ the representation $\pi_\lambda$ can be restricted to the minimal orbit (see Proposition~\ref{Prop:BesselOperatorsTangential}). Section~\ref{Integration to group level} contains the main body of this paper. We introduce a submodule  of the representation restricted to the minimal orbit and give a very explicit description of this submodule in Theorem~\ref{Theorem structure W}. We also show that for $p+q$ even and $p-2n-3 \not \in -2\mN$, this module can be integrated to the group level (see Corollary~\ref{corollary integration to group level}). 

In the last three sections we derive some further properties of our representation.
 We compute the annihilator ideal of our representation in Section~\ref{Section Joseph ideal} and show that it is equal to a Joseph-like ideal of $\mathfrak{osp}(p,q|2n)$ constructed in \cite{CSS} (see Theorem~\ref{Theorem Joseph ideal}). This links our representation to the definition of minimal representations in the classical case. The Gelfand--Kirillov dimension is computed in Section~\ref{Section Gelfand-Kirillov dimension} and equals $p+q-3$ (see Proposition~\ref{Prop Gelfand-Kirillov}), which is the same as the Gelfand--Kirillov dimension of the minimal representation of $\mathfrak{so}(p,q)$ and thus independent of the `super part'.
In Section~\ref{Section integration}, we introduce a linear functional which defines a non-degenerate sesquilinear form on $W$ resembling the $L^2$-inner product on the minimal orbit in the classical case. Our representation is shown to be skew-symmetric with respect to this form if $p+q-2n-6\geq 0$ (see Theorem~\ref{Theorem skew-symmetric}). 
 
Finally in the Appendix, we give a short introduction to supermanifolds and gather some results on Gegenbauer polynomials, the Bessel functions and the generalized Laguerre functions introduced in \cite{HKMM}. In particular, we also prove some new recursion relations for the generalized Laguerre functions, which are needed in Section~\ref{Integration to group level}.

\subsection*{Acknowledgments}

SB is a PhD Fellow of the Research Foundation - Flanders (FWO).
SB would like to thank Kevin Coulembier and Hendrik De Bie for many fruitful discussions and useful comments and the Mathematical Department of the FAU Erlangen-N\"{u}rnberg for its hospitality. JF also thanks Kevin Coulembier and Hendrik De Bie for an invitation to Gent, where this project was initiated.

\section{Preliminaries and notations} \label{Section Preliminaries}
In this paper, manifolds, affine spaces, Jordan and Lie algebras will be defined over the field of real numbers $\mR$, while functions spaces will be over the complex field $\mC$, unless otherwise stated. We use the notation $\mR^+$ for $\{ x \in \mR \mid x>0\}$ and $\mR^{m}_{(0)}$ for $\mR^{m} \backslash \{0\}.$ 
We use the convention $\mN = \{ 0, 1, 2, \ldots, \}$.
We always assume $p, q \in \mN$ with $p\geq 2$ and $q\geq 2$. 

A super-vector space is a $\mZ_2$-graded vector space $V=V_{\oa} \oplus V_{\ob}$. Elements in $V_{\oa}$ are called even, elements in $V_{\ob}$ odd and elements in $V_{\oa}\cup V_{\ob}$ homogeneous.  We use the notation $\abs{x}$ for the parity of a homogeneous element. So $\abs{x}=0$ for $x$ even and $\abs{x}=1$ for $x$ odd. We use the convention that the appearance of $\abs{x}$ in a formula implies that we are considering homogeneous elements and the formula has to be extended linearly for arbitrary elements. Write $\mR^{m|n}$ for the super-vector space $V$ with $V_{\oa}=\mR^m$ and $V_{\ob}=\mR^n$.

\subsection{The orthosymplectic Lie superalgebra and superpolynomials} \label{Section: conventions orthosymplectic metric}
 \label{Section another realisation}

We will start with a realisation of the orthosymplectic Lie superalgebra as differential operators on superpolynomials. 
Denote by $\cP(\mR^m)$ the space of complex-valued polynomials in $m$ variables and by $\Lambda^{2n}=\Lambda(\mR^{2n})$ the Grassmann algebra in $2n$ variables. Then we define the space of superpolynomials as  
\[
 \cP(\mR^{m|2n}) := \cP(\mR^m)\otimes_\mC \Lambda^{2n},
\] the space of complex-valued polynomials in $m$ even and $2n$ odd variables. These variables satisfy the commutation relations \[ \x_i \x_j = (-1)^{\abs{i}\abs{j}} \x_j \x_i.\]
We define the differential operator $\partial^i$ as the unique derivation in $\End(\cP(\mR^{m|2n}))$ such that  $\partial^i (\x_j) =\delta_{ij}$.

Consider a supersymmetric,  non-degenerate, even bilinear form $\langle \cdot ,\cdot \rangle_\beta$ on $\mR^{m|2n}$ with components $\beta_{ij}$ and let $\beta^{ij}$ be the components of the inverse matrix. So $\sum_j \beta_{ij} \beta^{jk}= \delta_{ik}.$ Set $\x^j = \sum_{i} \x_i \beta^{ij}$.
We also set $\partial_{j} = \sum_{i} \partial^{i} \beta_{ji}$. It satisfies $\partial_{i}(\x^j) =\delta_{ij}$.

The orthosymplectic Lie superalgebra $\mathfrak{osp}(m|2n,\beta)$ is the subalgebra of $\mathfrak{gl}(m|2n)$ spanned by homogeneous operators $X$ that satisfy
\[
\langle X(u), v \rangle_\beta + (-1)^{\abs{u}\abs{X}}\langle u, X(v) \rangle_\beta =0 \qquad \text{ for all } u,v \in V.
\]

We can realise the orthosymplectic Lie superalgebra using differential operators acting on $\cP(\mR^{m|2n})$. A basis of the orthosymplectic Lie superalgebra in this realisation is given by
\begin{align*}
L_{i,j}&:=\x_i \partial_{j} -(-1)^{\abs{i}\abs{j}} \x_j \partial_{i}, \quad   \text{ for } i< j \quad \text{ and } \quad L_{i,i}:=2 \x_i \partial_{i}, \quad  \text{ for } \abs{i}=1.
\end{align*}

Define also operators by
\begin{align}\label{definition laplace, euler and R squared}
R^2 &:= \sum_{i,j} \beta^{ij} \x_i \x_j, \quad \mE:=\sum_{i}\x^i \partial_{i} \quad \text{ and }\quad  \Delta:= \sum_{i,j} \beta^{ij} \partial_{i}\partial_{j}.
\end{align}
The operator $R^2$ acts through multiplication, $\mE$ is called the Euler operator and $\Delta$ the Laplacian. 
We have the following.
\begin{lemma}\label{Lemma relations sl(2)}
The operators $R^2$, $\mE$ and $\Delta$  commute with the orthosymplectic Lie superalgebra in $\End(\mathcal{P}(\mR^{m|2n}))$. Furthermore, they satisfy
\begin{align*}
[\Delta, R^2] &= 4\mE + 2M, \qquad [\Delta, \mE]= 2 \Delta, \qquad 
[R^2, \mE ]= -2R^2,
\end{align*}
where $M=m-2n$ is the superdimension.
\end{lemma}
\begin{proof}
A straightforward calculation or see, for example, \cite{DeBie}.
\end{proof}
In particular, Lemma~\ref{Lemma relations sl(2)} implies that $(R^2/2, \mE+M/2, -\Delta /2)$ forms an $\mathfrak{sl}(2)$-triple.

Later on we will need these operators not only as operators acting on superpolynomials but as global differential operators acting on an affine superspace. (We refer to Appendix \ref{supermanifolds} for a definition of the affine superspace and for an explanation of the notations we use.) We can extend their definition as follows. 
Consider a finite-dimensional super-vector space $V$ equipped with a supersymmetric,  non-degenerate, even bilinear form $\langle \cdot ,\cdot \rangle_\beta.$
Denote by $\x_i$ the coordinate function on the affine superspace $\mA(V^\ast)$ given by $\x_i (v) = v_i$ for $v= \sum_i v_i e^i,$ where $(e^i)_i$ is a homogeneous basis of $V^\ast$. Define $\partial^i$ as the unique element of $\Gamma(\cD_{\mA(V^\ast)})$ which satisfies $\partial^i(\x_j)= \delta_{ij}$.  
We define  $R^2, \Delta$, $\mE$ and $L_{ij}$ similarly as for the $\mR^{m|2n}$ case.  The operators $L_{ij}$ will give a realisation of $\mathfrak{osp}(V)$ and Lemma \ref{Lemma relations sl(2)} still holds.

\subsection{Spherical harmonics} We will collect here also some results on spherical harmonics, which we will use later on.   \label{Section spherical harmonics}
We write $\cP_k(\mR^{m|2n})$ for the space of homogeneous polynomials of degree~$k$. These polynomials satisfy
\[
\mE f= k f \quad \text{ for all } f \in \cP_k(\mR^{m|2n}).
\] 
The space of spherical harmonics $\cH_k(\mR^{m|2n})$ of degree $k$ are the homogeneous polynomials of degree $k$ which are also in the kernel of the Laplace operator:
\[
\cH_k(\mR^{m|2n}) = \{ f \in \cP_k(\mR^{m|2n}) \mid \Delta f= 0 \}.
\]
We have the following decomposition of $\cP(\mR^{m|2n})$, \cite[Theorem 3]{DeBie}:
\begin{proposition}[Fischer decomposition]
If $m-2n\not\in -2\mN$, then $\cP(\mR^{m|2n})$  decomposes as
\[
\cP(\mR^{m|2n}) = \bigoplus_{k=0}^\infty \cP_k(\mR^{m|2n}) = \bigoplus_{k=0}^\infty \bigoplus_{j=0}^\infty R^{2j} \cH_k(\mR^{m|2n}).
\]
\end{proposition}

\begin{proposition}
If $m-2n \not\in -2\mN$ and $n\not=0$, then $\cH_k(\mR^{m|2n})$ is an irreducible $\mathfrak{osp}(m|2n)$-module. If $n=0$, then $\cH_k(\mR^{m})$ is an irreducible $\mathfrak{so}(m)$-module if $m>2$, while $\cH_k(\mR^{2})$ decomposes as $ \mC z^k  \oplus \mC \bar{z}^k,$ where $z=x+\imath y,$ $\bar{z} = x- \imath y$, $(x,y)\in \mR^2.$
\end{proposition}
\begin{proof}
This is \cite[Theorem 5.2]{Coulembier} for the case $n\not= 0$ and \cite[Introduction, Theorem 3.1]{Helgason} for the classical case.
\end{proof}
The dimension of the spherical harmonics of degree $k$ is given in \cite[Corollary 1]{DeBie}.
\begin{proposition} \label{Prop: dimension of spherical harmonics}
The dimension of $\cH_k(\mR^{m|2n})$, for $m\not =0$, is given by
\begin{align*}
\dim \cH_k(\mR^{m|2n}) &= \sum_{i=0}^{\min (k,2n)} \begin{pmatrix}
 2n \\ i \end{pmatrix} \begin{pmatrix}
 k-i+m-1 \\ m-1 
 \end{pmatrix}
  -\sum_{i=0}^{\min (k-2,2n)} \begin{pmatrix}
 2n \\ i \end{pmatrix} \begin{pmatrix}
 k-i+m-3 \\ m-1 
 \end{pmatrix}.
\end{align*}

\end{proposition}

\subsection{Jordan superalgebras}
A Jordan superalgebra is a supercommutative superalgebra $J$ satisfying the  Jordan identity. This means $J=J_{\oa}\oplus J_{\ob}$ and
\begin{itemize}
\item $J_iJ_j \subset J_{i+j} \text{ for } i,j \in \mZ_2 $
\item $xy = (-1)^{\abs{x}\abs{y}} yx $
\item $(-1)^{\abs{x}\abs{z}}[L_x,L_{yz}]+(-1)^{\abs{y}\abs{x}}[L_y,L_{zx}]+(-1)^{\abs{z}\abs{y}}[L_z,L_{xy}]=0$, where the operator $L_x$ is (left) multiplication with $x$. \end{itemize}
Note that $[\cdot, \cdot]$ is the supercommutator, i.e. $[L_x,L_y]:=L_x L_y- (-1)^{\abs{x}\abs{y}} L_yL_x$.

The Tits--Kantor--Koecher construction associates with each
 Jordan superalgebra  a $3$-graded Lie superalgebra. Different TKK-constructions exist in the literature, which, in general, can lead to different Lie superalgebras. We will quickly review the Koecher construction for a unital Jordan superalgebra.  See~\cite{BC2} for an overview of the different TKK constructions appearing in the literature.
 
Denote by $\Inn(J)$ the subalgebra of $\mathfrak{gl}(J)$ of inner derivations, i.e. the algebra generated by the operators $[L_u,L_v]$, $u,v \in J$. 
We then define the inner structure algebra as the  following subalgebra of  $\mathfrak{gl}(J)$
\begin{align*}
\istr(J)&:=\{ L_u   \mid u \in J \} +\Inn(J) =\text{span}_{\mR}\{ L_u, [L_u,L_v] \mid u,v \in J \}.
\end{align*}
The last equality follows from the following property,  \cite[Section 1.2]{Kac},
\begin{align*}
[[L_x,L_y],L_z]= L_{x(yz)} -(-1)^{\abs{x}\abs{y}} L_{y(xz)}.
\end{align*}
Let $J^{\plus}$ and $J^{\minus}$ be two copies of $J$. 
Set \[\TKK(J):= J^{\plus} \oplus \istr(J) \oplus J^{\minus}.\]
The inner structure algebra is a subalgebra of $\TKK(J)$ and the other brackets on $\TKK(J)$ are defined as 
\begin{align*}
[x,u]&=2L_{xu}+2[L_x,L_u], &\qquad [x,y]&=[u,v]=0, \\
[L_{a},x]&=ax,  &\qquad  [L_{a},u] &= -au, \\
 [[L_{a},L_{b}], x]&=[L_{a},L_{b}] x, &\qquad [[L_{a},L_{b}], u]&=[L_{a},L_{b}]u,
\end{align*}
for homogeneous $x,y$ in $J^{\plus}$, $u,v$ in $J^{\minus}$, $a,b\in J$ and extended linearly and anti-commutatively.
\subsection{Lie supergroups and their actions}\label{Section Lie supergroups}
A Lie supergroup $G$ is a group object in the category of smooth supermanifolds, i.e.\ there exist morphisms $\mu\colon G \times G \to G$, $i\colon G \to G$, $e\colon \mR^{0\mid 0} \to G$, called the multiplication, inverse and unit which satisfy the standard group properties. We again refer to Appendix \ref{supermanifolds} for definitions of supermanifolds and morphisms between supermanifolds.

 Alternatively, we can also characterise Lie supergroups in the following manner:
\begin{definition}\label{Def Lie supergroup}
A Lie supergroup $G$ is a pair $(G_0, \mg)$ together with a morphism $\sigma \colon G_0 \to \End (\mg)$, where $G_0$ is a Lie group and $\mg$ is a Lie superalgebra for which
\begin{itemize}
\item $\Lie(G_0)$ is isomorphic to $\mg_{\oa}$, the even part of the Lie superalgebra $\mg$.
\item The morphism $\sigma$ satisfies $\sigma(g)_{\mid \mg_{\oa}} = \Ad(g) $ and $\diff \sigma (X) Y = [X,Y]$ for all $g \in G_0$, $X \in\mg_{\oa}$ and $Y \in \mg$. Here $\Ad$ is the adjoint representation of $G_0$ on $\Lie (G_0) \cong \mg_{\oa}$.
\end{itemize}
\end{definition}
See~\cite[Chapter 7]{CCF} for more details and the connection between those two approaches. 

By a closed Lie subgroup $H$ of a Lie supergroup $G$ we mean a closed embedded submanifold of $G$ that is also a subgroup. In the previous sentence we used submanifold instead of subsupermanifold and subgroup instead of subsupergroup. From now we will often omit the prefix super if it is clear from the context. 

A (left) action of a Lie supergroup on a supermanifold is a morphism $ a \colon G \times M \to M$ such that
\begin{itemize}
 \item  $a \circ (\mu \times \id_M) = a \circ (\id_G \times a) $
 \item $a \circ (e \times \id_M) \cong \id_M$, using $\mR^{0\mid0} \times M\cong M$.
\end{itemize}
For every even point $p$ of a supermanifold $M$ we have a morphism $p_{\mR^{0|0}}\colon \mR^{0|0} \to M$ where $\abs{p_{\mR^{0|0}}}$ maps to $p$ and $p_{\mR^{0|0}}^\sharp$ is evaluation at $p$. Then we define $a_p\colon G \to M$ for $p \in \abs{M}$ and $a_g \colon M \to M$ for $g \in \abs{G}$ by \[
a_p :=a \circ (\id_G \times p_{\mR^{0|0}}), \qquad a_g := a \circ (g_{\mR^{0|0}} \times \id_M ).
\]

Also for actions, we can use the equivalent approach with pairs.
\begin{definition}
An action $a$ of a Lie supergroup $G=(G_0,\mg)$ on a supermanifold $M$ is a pair $(\underline{a}, \rho_a)$ where
\begin{itemize}
\item $\underline{a}  \colon G_0 \times M \to M$ is an action of $G_0$ on $M$. 
\item ${\rho}_a \colon \mg \to {\rm Vec}_M$ is a Lie superalgebra anti morphism such that  \begin{align*}
{\rho_a}_{\mid \mg_{\oa}} (X) &= (X \otimes \id_{\mathcal{O}_M} ) \underline{a}^\sharp \qquad\;\, \text{ for all } X \in \mg_{\oa}, 
\\
\rho_a (\sigma(g)Y) &= \underline{a}_{g^{-1}}^\sharp \rho_a (Y)\underline{a}_{g}^\sharp \quad\qquad \text{ for all } Y \in \mg, g \in G_0.
\end{align*}
\end{itemize}
Here ${\rm Vec}_M$ is the Lie superalgebra of vector fields on $M$, and we silently use the isomorphism $\mg_{\oa}\cong T_e G_0$.
\end{definition}
See~\cite[Chapter 8]{CCF} for more details.

By the reduced action $\abs{a}$, we will mean the (ordinary) Lie group action $\abs{\underline{a}}$ of $\abs{G}= G_0$ on $\abs{M}$.
We have the two following propositions.
\begin{proposition}[{\cite[Proposition 8.4.7]{CCF}}]\label{Prop: stabiliser subgroup}
Let $G$ be a supergroup with an action $a$ on $M$ and let $p \in \abs{M}$.
Set \[ \widetilde{G_p}= \{ g \in G_0 \mid \abs{a}(g,p)=p \} \qquad \text{ and }\qquad \mg_p := \ker \diff a_p.\]  
Then $G_p=(\widetilde{G_p}, \mg_p)$ is a closed subgroup of $G=(G_0, \mg)$.
\end{proposition}
\begin{proposition}[{\cite[Proposition 9.3.7]{CCF}}]\label{Prop: definition orbit}
Let $G$ be a Lie supergroup and $H$ a closed subgroup. There exists a  supermanifold $X= (\abs{G} / \abs{H} , \cO_X)$ and a morphism $\pi\colon G \to X$ such that 
\begin{itemize}
\item The reduction $\abs{\pi} \colon \abs{G} \to \abs{G} / \abs{H}$ is the natural map. 
\item The morphism $\pi$ is a submersion, i.e.\ for all $g \in \abs{G}$ the map $\diff \pi_g\colon T_gG \to T_{\pi(g)}X$ is surjective.
\item There is  an action $\beta\colon  G \times X \to X$, which reduces to the action of $\abs{G}$ on $\abs{X}$ such that $\pi \circ \mu= \beta \circ (\id_G \times \pi)$, where $\mu$ is the multiplication on $G$.
\end{itemize}
Moreover the pair $(X, \pi)$ satisfying these properties is unique up to isomorphism.
\end{proposition}
These two propositions allow us to define the orbit through an even point $p$.
\begin{definition} \label{Def: orbit}
Let $G$ be a Lie supergroup with an action on a supermanifold $M$. Let $p \in \abs{M}$. Let $G_p$ the closed subgroup defined in Proposition~\ref{Prop: stabiliser subgroup}. Then we define the orbit $C_p$ through the point $p$ as the manifold $X=(\abs{G}/\abs{G_p}, \cO_X)$ defined in Proposition~\ref{Prop: definition orbit}.
\end{definition} 

\section{The orthosymplectic Lie superalgebra and associate structures}\label{Section associate structures}
In this section, we introduce the algebras and groups we will use in the rest of the article. We start with the spin factor Jordan superalgebra. This is the Jordan superalgebra which is associated to the orthosymplectic Lie superalgebra via the TKK-construction. Then we calculate its structure and TKK algebra. Finally, we also define the structure group and conformal group as Lie supergroups who have the structure algebra and TKK algebra as underlying Lie superalgebras. 

\subsection{The spin factor} \label{Section definition spin factor}

We now define the real spin factor Jordan superalgebra associated with an orthosymplectic metric. 
Let $V$ be a real super-vector space with $\dim (V)= (p+q-3|2n)$ and a supersymmetric, non-degenerate, even, bilinear form $\langle \cdot, \cdot \rangle_{\tilde{\beta}}$ where the even part has signature $(p-1,q-2)$.  We will always assume that $p \geq 2$ and $q\geq 2$. We choose a homogeneous basis $(e_i)_i$ of $V$.
For $u= \sum_i u^ie_i$ and $v= \sum_i v^i e_i$ we then have
\[
\langle u, v\rangle_{\tilde{\beta}}  =\sum_{i,j} u^i {\tilde{\beta}}_{ij} v^j \quad\text{ with }  \quad {\tilde{\beta}}_{ij} :=\langle e_i, e_j \rangle_{\tilde{\beta}}.
\]
We have $\tilde{\beta}_{ij}= 0$ if $\abs{i}\not= \abs{j}$ since the form is even, while $\tilde{\beta}_{ij}=(-1)^{\abs{i}\abs{j}}\tilde{\beta}_{ji}$ because it is supersymmetric and $\det((\tilde{\beta}_{ij})_{ij}) \not=0$ since the form is non-degenerate.

We define the  spin factor Jordan superalgebra ${\mathcal{JS}pin}_{p-1,q-2|2n}$ as\[
J:= \mR e \oplus V \text{ with } \abs{e}=0. \]
The Jordan product is given by  \[ (\lambda e+ u ) (\mu e + v) = (\lambda\mu + \langle u, v\rangle_{\tilde{\beta}} ) e + \lambda v + \mu u \quad \text { for } u,v \in V, \;\lambda,\mu \in \mR.
\]
Thus $e$ is the unit of $J$.

We extend the homogeneous basis $(e_i)_{i=1}^{p+q-3+2n}$ of $V$ to a homogeneous basis $(e_i)_{i=0}^{p+q-3+2n}$ of $J$ by setting $e_0$ equal to the unit $e$.

Define $(\tilde{\beta}^{ij})_{ij}$ as the inverse of $(\tilde{\beta}_{ij})_{ij}$.
Let $(e^i)_i$ be the right dual basis of $(e_i)_i$ with respect to the form $\langle\cdot,\cdot \rangle$, i.e. \[\langle e_i, e^j \rangle = \delta_{ij} \qquad \text{with } \delta_{ij} \text{ the Kronecker delta}.\]
Then
\[
e^j =\sum_{i} e_i \tilde{\beta}^{ij}. 
\]
 
Consider $J^\ast=\mR e^\ast \oplus V^\ast$ the dual super-vector space of $J$ with right dual basis $(e^i)_i$. Define a bilinear form on $V^\ast$ by $\langle e^i, e^j \rangle :=\langle e^i, e^j \rangle_{{\tilde{\beta}}} = {{\tilde{\beta}}}^{ji}$. Then we can make also $J^\ast$ into a spin factor Jordan superalgebra with respect to this bilinear form.

\subsection{The structure and TKK algebra}
Consider the orthosymplectic metric $\tilde{\beta}$ used in Section \ref{Section definition spin factor}. We extend this form as follows.
Set $\beta_{00}=-1$, $\beta_{i0}= 0=\beta_{0i}$, $\beta_{ij}={\tilde{\beta}}_{ij}$ for $i,j \in \{ 1, .. , p+q-3+2n\}$. Then the corresponding form $\langle \cdot, \cdot \rangle_\beta$ is a supersymmetric, non-degenerate, even  bilinear form on the super-vector space $J$ where the even part has signature $(p-1,q-1)$. Consider the orthosymplectic Lie superalgebra $\mathfrak{osp}(J)$, i.e.\ the subalgebra of $\mathfrak{gl}(J)$ which leaves the form $\langle \cdot, \cdot \rangle_\beta$ invariant.
\begin{proposition} \label{Prop: inner structure algebra}
We have 
\[ \istr(J) = \mathfrak{osp}(J) \oplus \mR L_e, \]
where the direct sum decomposition is as algebras.
Furthermore \[ \TKK (J)= \mathfrak{osp}(p,q|2n). \] 
\end{proposition}
\begin{proof} From \cite[Section 6.1]{BC2}, it follows that for the complexified Jordan superalgebra $J_\mC$ we have
\[
\istr(J_{\mC}) = \mathfrak{osp}_{\mC}(J) \oplus \mC  \text{ and } \TKK (J_{\mC})= \mathfrak{osp}_{\mC}(p+q|2n).
\]  
For $n=0$ we find
\[
 \TKK (J)= \mathfrak{so}(p,q),
\]
see for example \cite[Section 2.5]{KM2}.
One can check that the even part of $\TKK(J)$ still contains a component $\mathfrak{so}(p,q)$ if $n>0$. For $p+q-2>0$, there is a unique real form of $\mathfrak{osp}_\mC (p+q|2n)$ with contains the component $\mathfrak{so}(p,q)$, \cite[Theorem 2.5]{Parker}. So we can conclude
 \[ \TKK (J)= \mathfrak{osp}(p,q|2n). \] 
The inner structure algebra is spanned by the operators $L_{e_i}$, $[L_{e_i},L_{e_j}]$ for $i,j >0$ and $L_e$. Observe that $L_e$ is in the centre of $\istr(J)$ since $e$ is the unit. 
We defined $\langle\cdot,\cdot \rangle_\beta$ such that \[
\langle e_i e_j, e_k \rangle_{\beta}=0, \quad \langle e_i e_0, e_j \rangle_\beta ={\tilde{\beta}}_{ij},\quad \text{ and } \langle e_0, e_0 \rangle_\beta=-1,
\]
for $i,j,k >0$. Using this, one can show that the operators $X=L_{e_i}$ or  $X=[L_{e_i},L_{e_j}]$ for $i>0$ satisfy
\[
\langle X (u), v \rangle_\beta + (-1)^{\abs{X}\abs{u}} \langle u, X(v) \rangle_\beta=0.
\]
Hence they form  a subspace of $\mathfrak{osp}(J)$ and we obtain \[ \istr(J) \subset \mathfrak{osp}(J) \oplus \mR L_e.\]  Since $\istr(J_{\mC}) = \mathfrak{osp}_{\mC}(J)\oplus \mC$ we conclude that this inclusion is actually an equality.  
\end{proof}

Using the bilinear form
\[
\overline{\beta} = \begin{pmatrix}
1 & &  &\\ & \beta_{sym} & & \\ & & -1 & \\ & &  & \beta_{asym}
\end{pmatrix},
\]
we saw in Section \ref{Section another realisation} that we have a realisation of  $\mathfrak{osp}(p,q|2n)$ using differential operators 
\[
L_{i,j} = \x_i \partial_j -(-1)^{\abs{i}\abs{j}} \x_j \partial_i.
\]
An explicit isomorphism of $\TKK(J)$ with this realisation of $\mathfrak{osp}(p,q|2n)$ is given by
\begin{align*}
e_i^+ &\mapsto L_{\tilde{i},(p+q-1)} -L_{\tilde{i},0} & e_0^+ &\mapsto -L_{(p+q-2),(p+q-1)} -L_{(p+q-2),0} \\
L_{e_i} &\mapsto L_{\tilde{i},(p+q-2)} &
L_{e_0} &\mapsto L_{0,(p+q-1)} \\
[L_{e_i} ,L_{e_j}] &\mapsto L_{\tilde{i},\tilde{j}} \\
e_i^- &\mapsto L_{\tilde{i},(p+q-1)} +L_{\tilde{i},0} &
e_0^+ &\mapsto L_{(p+q-2),(p+q-1)} +L_{(p+q-2),0}.
\end{align*}
Here $\tilde{i}=i $ if $\abs{i}=0$ and $\tilde{i}=i+1$ if $\abs{i}=1$. 
This yields another approach to show that $\TKK(J)\cong \mathfrak{osp}(p,q|2n)$.
\subsection{The structure group} \label{section structure group}

Define \begin{align*}
 O(p-1,q-1) &= \{ X \in  \mR^{(p+q-2)\times (p+q-2)}\mid X^t \beta^s X =\beta^s \} \\
Sp(2n, \mR)& = \{ X \in \mR^{(2n)\times (2n)}\mid X^t \beta^a X =\beta^a\},
 \end{align*}
 where $\beta^s$, $\beta^a$ are the matrices formed by the symmetric part and the anti-symmetric part of the bilinear form of $J$. 

Set \[ Str(J)_0:= \mR^+\times O(p-1,q-1) \times Sp(2n, \mR) \text{   } \]
and recall by Proposition~\ref{Prop: inner structure algebra}
\[  \istr(J)=\mathfrak{osp}(J) \oplus \mR L_e. \]
We embed $Str(J)_0$ in $\mR^{(p+q-2+2n) \times (p+q-2+2n)}$ by associating to the triple $(\nu, k, h) \in Str(J)_0$ the matrix $\nu\begin{pmatrix}
 k &  \\ 
 &  h
\end{pmatrix}$ with $\nu \in \mR^+$, $k \in \mR^{(p+q-2)\times (p+q-2)}$ and $h \in  \mR^{(2n)\times (2n)}$. We will also interpret $X \in\mathfrak{osp}(J)$ as an $(p+q-2+2n) \times (p+q-2+2n)$ matrix.

For $\nu \in \mR^+,$ $k \in O(p-1,q-1)$ and $h\in  Sp(2n, \mR)$, define $\sigma(\nu,k,h) \in \End(\istr(J))$  
\begin{align*}
 \sigma(\nu ,k,h) L_e &= L_e \quad \text{ and }\quad 
  \sigma(\nu ,k,h) X := \begin{pmatrix}
 k &  \\ 
 &   h
\end{pmatrix}X \begin{pmatrix}
  k^{-1} &  \\ 
 &  h^{-1}
\end{pmatrix}
\text{ for }X \in \mathfrak{osp}(J).
\end{align*}

Then $Str(J)=(Str(J)_0, \istr(J), \sigma)$ defines a Lie supergroup, in the sense of Definition~\ref{Def Lie supergroup}. We call $Str(J)$ \emph{the structure group}. 

Next, we define an action of $Str(J)$ on $\mA(J^\ast)$, the affine superspace associated to the dual super-vector space of $J$. Let $\x_i$ be the coordinate functions
 on $J^\ast$. For $x=\sum_i x_ie^i \in J^\ast$, we then have $\x_j(x) = x_j$. By the global chart theorem, \cite[Theorem 4.2.5]{CCF}, a morphism $\phi$ from a supermanifold $M$ to an affine superspace is determined by the pullbacks of the coordinate functions. So we can define $\underline{a}\colon Str(J)_0 \times \mA(J^\ast) \to \mA(J^\ast)$ by
\[
\underline{a}^\sharp (\x_i) =  g^{-1} \x_i = \begin{pmatrix}
(\nu  k)^{-1} &  \\ 
 & (\nu  h)^{-1}
\end{pmatrix} \begin{pmatrix}
\x_i 
\end{pmatrix} \in  \cO_{Str(J)_0} \hat{\otimes}  \cO_{\mA(J^\ast)},
\]
where $(\nu,k,h)=g=(g_{ij})_{1\leq i,j \leq p+q-2+2n} \in Str(J)_0 \subset \mR^{(p+q-2)\times (p+q-2)}$. We interpret the $g_{ij}$  as coordinate functions on $\mR^{(p+q-2)\times (p+q-2)}$ and then restrict them to functions in $\cO_{Str(J)_0}$.

 Set 
\begin{align*} 
\rho_a \colon \istr(J) \to {\rm Vec}_{\mA(J^\ast)} ;\quad  \rho_a (L_{ij})&= - (\x_i \partial_j - (-1)^{\abs{i}\abs{j}} \x_j \partial_i) \text{ for }L_{ij} \in \mathfrak{osp}(J), \\
\rho_a (L_e)  &= -\mE.
\end{align*}

\begin{proposition}
The pair $(\underline{a}, \rho_a)$ defines an action of the Lie supergroup $(Str(J)_0, \istr(J))$ on $\mA(J^\ast)$.
\end{proposition} 
\begin{proof}
The map $\rho_a$ is clearly a Lie superalgebra anti-morphism from $\istr(J)$ to ${\rm Vec}_{\mA(J^\ast)}$. One can also check that $\underline{a}$ indeed defines an action of $Str(J)_0$ on $M$. So we only need to prove \[ {\rho_a}_{\mid \istr(J)_{\oa}} (X) = (X \otimes \id_{\mathcal{O}_M} ) \underline{a}^\sharp, \qquad  
\rho_a (\sigma(g)Y) = \underline{a}_{g^{-1}}^\sharp \rho_a (Y)\underline{a}_{g}^\sharp  \]
for $X \in \istr(J)_{\oa}$, $Y \in \istr(J)$, $g \in Str(J)_0.$
If we interpret $X \in \istr(J)_{\oa}$ as an element of $T_e Str(J)_0$, then it acts on the coordinate functions $g_{ij}$ as $X(g_{ij})= X_{ij}$. The map $g \mapsto g^{-1}$ corresponds on algebra level to $X \mapsto -X$, hence we also have $X({(g^{-1})}_{ij})=-X_{ij}.$ Thus
\begin{align*}
(X \otimes \id_{\mathcal{O}_M} ) \underline{a}^\sharp(\x_i) &= X \otimes \id_{\mathcal{O}_M} \begin{pmatrix}
(\nu  k)^{-1} &  \\ 
 & (\nu  h)^{-1}
\end{pmatrix} \begin{pmatrix}
\x_i 
\end{pmatrix}  = -X (\x_i)= \rho_a(X). 
\end{align*}
Furthermore $\underline{a}^\sharp_g(\x_i)= g^{-1} \x_i$. We find
\begin{align*}
\underline{a}_{g^{-1}}^\sharp \rho_a (Y)\underline{a}_{g}^\sharp (\x_i)= -g (Y(g^{-1}\x_i))= -(\sigma(g)Y) \x_i = \rho_a (\sigma(g)Y) \x_i.
\end{align*}
We conclude that the pair $(\underline{a}, \rho_a)$ is an action.
\end{proof}
\subsection{The conformal group}\label{Section conformal group}

We define the conformal group as follows.  Define for $k \in O(p,q)$ and $ h \in Sp(2n, \mR)$, $\sigma(k,h) \in \End(\mathfrak{osp} (p,q|2n))$ by
\[
  \sigma(k,h) X := \begin{pmatrix}
k &  \\ 
 &   h
\end{pmatrix}X \begin{pmatrix}
  k^{-1} &  \\ 
 &  h^{-1}
\end{pmatrix}
\text{ for }X \in \mathfrak{osp} (p,q|2n). \]
Then also $(O(p,q) \times Sp(2n, \mR), \mathfrak{osp} (p,q|2n),\sigma)$ is a Lie supergroup, which we call the conformal group and denote by $OSp(p,q|2n)$. 

 \subsection{A representation of $\mathfrak{osp}(p,q|2n)$} \label{Section: definition representation}
We will consider a realisation of $\mathfrak{osp}(p,q|2n)=\TKK(J)$ in the space of differential operators of the affine superspace $\mA (J^{\ast})$ as constructed in \cite{BC}. This representation depends on a character of $\istr(J)$. For a simple Jordan algebra the representation constructed in \cite{BC} corresponds to the representation of the conformal algebra considered in \cite{HKM}.

Recall $\istr(J) = \mathfrak{osp}(J) \oplus \mR L_e$ by Proposition~\ref{Prop: inner structure algebra}. 
A character $\lambda\colon \istr(J) \to \mR$ is thus uniquely determined by its value on $L_e$, because $[\mathfrak{osp}(J),\mathfrak{osp}(J)]=\mathfrak{osp}(J)$. We will denote the value of $\lambda(2L_e)$ also by $\lambda$.

 Up to an automorphism of $\Gamma(\cD_{\mA(J^\ast)})$ induced by $e_k \mapsto -\imath e_k$, the representation in \cite[Section 4.1]{BC} is given as follows  
\[ \pi_\lambda: {\rm TKK}(J)=J^{\plus}\oplus \istr(J) \oplus J^{\minus}\;\;\;\to\;\;\;  \Gamma(\cD_{\mA(J^\ast)})\]
\begin{enumerate} 
\item $\pi_\lambda(0,0,e_k) = -  \imath \x_k\qquad\qquad\quad\qquad\qquad\qquad\qquad$ for $e_k \in J^{\minus}$
\item $\pi_\lambda(0,L_{ij},0) =  \x_i \partial_{\x^j} -(-1)^{\abs{i}\abs{j}} \x_j \partial_{\x^i} \qquad\qquad\; $ for $L_{ij} \in \mathfrak{osp}(J)$
\item $\pi_\lambda(0,L_e,0)= \frac{\lambda}{2}-\mE$
\item $\pi_\lambda(\bar{e}_k,0,0)= -\imath \cB_\lambda(e_k) \qquad\quad\qquad\qquad \qquad \qquad \text{for  } \bar{e}_k \in J^{\plus}$. 
\end{enumerate}
Here $(e_i)_{i=0}^{m+2n-1}$ is the homogeneous basis of $J^{\minus}=J$ introduced in Section~\ref{Section definition spin factor}. To simplify the expressions, we introduced a basis $(\bar{e}_i)_i$ of $J^{\plus}$ by $\bar{e}_i := e_i$ for $i>0$ and $\bar{e}_0:=-e_0$.  
The Bessel operator $\cB_\lambda$ is an even  global differential operator on $\mA(J^\ast)$ taking values in the super-vector space $(J^{\plus})^\ast$. From \cite[Definition 4.1]{BC}, we get the following expressions for the Bessel operator
\begin{align}\label{Expression Bessel operators}
\cB_\lambda (e_k)=(-\lambda+2\mE)\partial_{k} - \x_k \Delta,
\end{align}
where $\mE$ and $\Delta$ are the Euler operator and Laplacian introduced in equation~\eqref{definition laplace, euler and R squared}. We will also write $\cB_\lambda(\x_k)$ for $\cB_\lambda(e_k)$.

\section{The minimal orbit} \label{Section minimal orbit}
We will use the action of the structure group $Str(J)$ on  $\mA(J^\ast)$ to construct a minimal orbit. For ordinary (i.e.\ not super) Jordan algebras the minimal orbit under the action of the structure group is the one through a primitive idempotent, see~\cite{Kaneyuki}.  We will use this as a definition for the minimal orbit in our case.
\begin{remark}{\rm
If one looks at the action under the identity component of the structure group, as for example is done in \cite{HKM}, then this picture changes a bit. For non-Euclidean Jordan algebras there is still only one minimal orbit, but for Euclidean Jordan algebras we then have two minimal orbits, one through a primitive idempotent $c$ and one through $-c$. 	
}
\end{remark}
Let us first introduce the natural generalisations of primitive idempotents to Jordan superalgebras. 
An even element of a Jordan superalgebra is called an \emph {idempotent} if it satisfies 
$x^2=x$. An idempotent is \emph{primitive} if it can not be written as the sum of two other (non-zero) idempotents. Two idempotents are called \emph{orthogonal} if their product is zero. A \emph{Jordan frame} is a collection of pairwise orthogonal primitive idempotents which sum to the unit \cite[Chapter IV]{FK}.
\begin{proposition}
\label{Prop: idempotents in spin factor}
For the spin factor Jordan superalgebra $J^\ast$ it holds that
an element $c= \lambda e + x \in J^\ast_{\oa}$ is a non-zero idempotent iff $\lambda= \frac{1}{2}$ and $x$ satisfies $\langle x, x \rangle = \frac{1}{4}$ or  $c=e$ and $x=0$.  Here $e$ denotes the unit of $J^\ast$.

All idempotents different from the unit are primitive and if $\frac{1}{2}e +x$ is an idempotent then $(\frac{1}{2}e +x,\frac{1}{2}e -x)$ is a Jordan frame.
\end{proposition}
\begin{proof}
Straightforward verification.
\end{proof}

Observe that the reduced action of the structure group $Str(J)$ on $\abs{\mA(J^\ast)}$ is equivalent with the action of $\mR^+ O(p-1,q-1)$ on $ \mR^{p+q-2}$ given by
\[
(g, x) \to g^{-1} x,
\]
since $Sp(2n, \mR)$ acts trivially on $\mR^{p+q-2}$.
Hence, for a primitive idempotent, the topological space underlying the orbit manifold is the same as in the classical case. This topological space is independent of the chosen idempotent and given by, see for example \cite[Section 1.2]{HKM},  
\[
  \{ x\in \mR^{p+q-2}_{(0)}\mid R^2(x)=0\},
\]
where $R^2= \sum_i \x^i\x_i$ is the superfunction defined in \eqref{definition laplace, euler and R squared}. We interpret $R^2$ not as an operator but as a function. Thus $R^2(x)$ denotes evaluating the function $R^2$ in the even point $x$. This also means that the odd component in $R^2$ does not play any role. 

Let $\mA(J^\ast)_{(0)}$ be the open submanifold of $\mA(J^\ast)$ we get by excluding zero \[ 
\mA(J^\ast)_{(0)} =(\mR^{p+q-2}_{(0)}, \cC^\infty_{ \mR^{p+q-2}_{(0)}} \otimes \Lambda^{2n}).
\]
Denote by $\langle R^2 \rangle $ the ideal in  $\Gamma(\cO_{\mA(J^\ast)_{(0)}})$ generated by $R^2$.
 Set \[
 \abs{C} :=\{ x\in \mR^{p+q-2}_{(0)}\mid R^2(x)=0\}.
 \]
  We will show that there is supermanifold  $C$ which has $\abs{C}$ as its underlying topological space and  $\Gamma(\cO_{\mA(J^\ast)_{(0)}})/ \langle R^2 \rangle$ as its global sections. By \cite[Corollary 4.5.10]{CCF}, the global sections will determine the sheaf $\cO_C$. 
The main theorem of this section establishes that $C=(\abs{C}, \cO_C)$ is the orbit through a primitive idempotent under the action of the structure group.
\begin{theorem} \label{Theorem minimal orbit}
The space $C=(\abs{C},\cO_C)$ is a well-defined supermanifold.
Furthermore it is the orbit through a primitive idempotent of $J^\ast$ under the action of the structure group $Str(J)$ on $\mA(J^\ast)$  defined in Section~\ref{section structure group}. We will call $C$ the minimal orbit.
\end{theorem}
In the two following subsections we will prove this theorem.

\subsection{The space $C$ is a supermanifold}
We first introduce the notion of a regular ideal, which we then use to show that $C$ is a well-defined supermanifold. 

\begin{definition}[{\cite[Definition 5.3.6]{CCF}}]
Let $M$ be a supermanifold with underlying topological space $\abs{M}$. Let $I$ be an ideal in $\Gamma(\cO_M)$. For $m \in \abs{M}$ denote by $\mathcal{J}_m$ the maximal ideal in $\Gamma(\cO_M)$ given by the kernel of the morphism $\ev_m\colon  \Gamma(\cO_M) \to \mR$ and by $I_m$ the image of $I$ in the stalk $\cO_{M,m}$.
Then $I$ is called a \emph{regular ideal} if 
\begin{itemize}
\item  For every $m \in \abs{M}$ such that $I \subset \mathcal{J}_m$ there exist homogeneous $f_1, \ldots, f_n$ in $I$ such that  $[f_1]_m, \ldots, [f_n]_m$ generate $I_m$ and $(d f_1)_m, \ldots, (d f_n)_m$ are linearly independent at $m$, where $[f_i]_m$ is the image of $f_i$ in $\cO_{M,m}$.
\item If $\{ f_i \}_{i\in \mN}$ is a family in $I$ such that any compact subset of $M$ intersects only a finite number of supp $f_i$, then $\sum_i f_i$ is an element of $I$.
\end{itemize}
\end{definition}
Regular ideals can be used to define supermanifolds in the following manner.
\begin{proposition}[{\cite[Proposition 5.3.8]{CCF}}] \label{Prop: embedded submanifold for regular ideal} 
Let $M$ be a supermanifold and $I$ a regular ideal in $\Gamma(\cO_M)$. Then there exists a unique closed embedded supermanifold $(N, j)$, where $j\colon N \to M$ is an embedding, such that \[ \Gamma(\cO_N) = \Gamma(\cO_M )/I. \]
\end{proposition}
From the proof of the proposition it also follows that the underlying topological space $\abs{N}$ of $N$ is given by
\[
\abs{N} = \{ m \in \abs{M} \mid I \subset \mathcal{J}_m \}=  \{ m \in \abs{M} \mid \ev_m (f) =0 \text{ for all } f \in I\}.
\]
\begin{lemma} \label{Lemma ideal is regular}
The ideal $I$ in $\Gamma(\cO_{\mA(J^\ast)_0})$ generated by $R^2$ is regular.
\end{lemma}
\begin{proof}
For any $m$ in $\mR^{p+q-2}_{(0)}$  we have that $I_m$ is generated by  $[R^2]_m$. Furthermore $(dR^2)_m$ is different from zero if $m\not=0$ and thus linearly independent. Since every $f_i$ in $I$ can be written as $R^2g_i$, we have
\[
\sum_i f_i = R^2 \sum_i g_i \in I.
\]
So we conclude that $I$ is a regular ideal.
\end{proof}

We have that \[
\abs{C}= \{ m\in \mR^{p+q-2}_{(0)} \mid \ev_m(R^2)=0\}=\{m\in \mR^{p+q-2}_{(0)} \mid I \subset \mathcal{J}_m \}. \]
is the topological space corresponding to the regular ideal $I=\langle R^2 \rangle $. 
\begin{corollary}
The space $C=(\abs{C},\cO_C)$ is the unique closed embedded submanifold of $\mA(J^\ast)_0$ corresponding to the regular ideal $\langle R^2 \rangle$. 
\end{corollary}
\begin{proof}
This follows immediately from combining Proposition~\ref{Prop: embedded submanifold for regular ideal} and Lemma~\ref{Lemma ideal is regular}.
\end{proof}
We denote the embedding of $C$ in  $\mA(J^\ast)_0$ by $j_C$.

\subsection{The space $C$ is an orbit} \label{Section C is an orbit}
We will show that $C$ is the orbit through a primitive idempotent in the sense of Definition~\ref{Def: orbit}.
We introduce the following morphisms, 
\begin{align*}
a&\colon Str(J) \times \mA(J^\ast) \to \mA(J^\ast) \\
j &\colon C  \hookrightarrow \mA(J^\ast).
\end{align*}
The morphism $a$ is the action of $Str(J)$ on $\mA(J^\ast)$  defined in Section~\ref{section structure group}. For the morphism $j$ we combine the embedding $j_C$ of $C$ in $\mA(J^\ast)_0$ with the embedding of  $\mA(J^\ast)_0$ in $\mA(J^\ast)$.
Define 
\begin{align*}
b \colon Str(J) \times C \to \mA(J^\ast); \qquad b= a \circ (\id_{Str(J)} \times j).
\end{align*}
Then $b=(\abs{b}, b^\sharp)$  with 
 \begin{align*}
\abs{b}= \abs{a} \circ ( \id_{\abs{Str(J)}} \times \abs{j}) \text{ and } b^\sharp = (\id_{\cO_{Str(J)}} \otimes j^\sharp ) a^\sharp.
\end{align*}
\begin{lemma}
The morphism $b$ takes values in $C$.
\end{lemma}
\begin{proof}
We have to show that $b$ factors as $j \circ \gamma$, with $\gamma\colon Str(J) \times C \to C$. This will be the case if  $\im \abs{b} \subset \abs{C}$ and $b^{\sharp} (R^2) = 0$. On the topological level it is immediately clear that $\abs{b}$ takes values in $\abs{C}$. For the sheaf morphism, we will use the fact that 
for a Lie supergroup $G=(G_0,\mg)$  we have \cite[Remark 7.4.6]{CCF}
\begin{align*}
\cO_G(U) \cong \underline{\Hom}_{U(\mg_{\oa})}(U(\mg), \cC^\infty_{G_{0}}(U)).
\end{align*}
Note that by $\underline{\Hom}(V,W)$ we mean all linear maps from $V$ to $W$ including the odd ones. 
Using this isomorphism, an action $a= (\underline{a},\rho_a)$ on $M$ can be expressed in $\underline{a}$ and $\rho_a$ as
\begin{align*}
a^\sharp&\colon \Gamma(\cO_M) \to \underline{\Hom}_{U(\mg_{\oa})} (U(\mg), \cC^\infty_{G_{0}} (G_{0}) \hat{\otimes} \Gamma(\cO_M)) \\
&f \mapsto [X \mapsto (-1)^{\abs{X}} (\id_{\cC^\infty(G_{0})}\otimes \rho_a(X)) \underline{a}^\sharp f], \text { with } X \in U(\mg)
\end{align*}
The Lie group $Str(J)_0$ preserves the orthosymplectic metric on $J$, so \[
 \begin{pmatrix}
(\nu  k)^{-1} &  \\ 
 & (\nu  h)^{-1}
\end{pmatrix}^t \beta^{-1}  \begin{pmatrix}
(\nu  k)^{-1} &  \\ 
 & (\nu  h)^{-1}
\end{pmatrix}= \beta^{-1} 
\] with $\beta$ the matrix corresponding to the metric. 
Hence \[\underline{a}^\sharp (R^2)= \underline{a}^\sharp (\x_i) \beta^{ij} \underline{a}^\sharp(\x_j ) = \id_{\cO_{Str(J)_0}} \otimes \x_i \beta^{ij} \x_j = \id_{\cO_{Str(J)_0}} \otimes R^2.\] 
Therefore \[
a^\sharp(R^2) = [X \mapsto (-1)^{\abs{X}} (1 \otimes \rho_a(X) R^2) ].
\]
So we get
\begin{align*}
b^\sharp (R^2) = (\id_{\cO_{Str(J)}} \otimes j^\sharp ) a^\sharp (R^2)=[X \mapsto (-1)^{\abs{X}} (1 \otimes j^\sharp (\rho_a(X) R^2)) ] =0,
\end{align*}
since  $\rho_a(X) R^2=0$ for $X \in \mathfrak{osp}(J)$ and for $X=L_e$ we use that $R^2$ evaluates to zero on $\abs{C}$.
\end{proof}
For a primitive idempotent  $c$ of $J^\ast$, we define 
\begin{align*}
\pi \colon Str(J) \to C \text{ by }\quad \abs{\pi}g= \abs{b}(g,c), \quad \pi^\sharp = (\id_{\cO_{Str(J)}} \otimes \ev_c )b^\sharp.
\end{align*}
\begin{proposition}\label{Prop: C is an orbit}
The manifold $C$ and the morphisms $\pi$ and $b$ satisfy the conditions of Proposition~\ref{Prop: definition orbit}. In particular $C$ is the manifold corresponding to the orbit through a primitive idempotent of $J^\ast$ under the action of the structure group.
\end{proposition}
\begin{proof}
Almost by definition, the map $\abs{\pi}$ is the natural map from $\abs{Str(J)}$ to $\abs{C}.$
To show that $\pi$ is a submersion, we need
\[
\diff \pi_g \colon T_g Str(J) \to T_{\abs{\pi}(g) }C 
\]
to be surjective for all $g \in \abs{Str(J)}$. 
Consider $ f \in \cO_C(U)$ and let $\tilde{f} \in \cC^\infty_{\mR^{p+q-2}_{(0)}}(U)\otimes \Lambda^{2n}$ be a representative of $f$ i.e.\ $f=\tilde{f} \mod R^2$. Let $X$ be a vector field in $\istr(J)$ and $X_e$ the corresponding vector in $T_eStr(J)$.  From \cite[Proposition 7.2.3]{CCF}, we have $X_g:= \ev_g X = \ev_g (1\otimes X_e) \mu^\sharp$. Combining this with $\ev_c \circ j = \ev_c$, we compute
\begin{align*}
\diff \pi_g(X_g)f&= X_g(\pi^\sharp f) \\
&= \ev_g (1\otimes X_e) \mu^\sharp (\id_{\cO_{Str(J)}} \otimes \ev_c) a^\sharp \tilde{f} \\
&= \ev_c (\ev_g \otimes \id_{\cO_C}) (\id_{\cO_{Str(J)}}\otimes X_e \otimes \id_{\cO_C})(\mu^\sharp \otimes \id_{\cO_C}) a^\sharp \tilde{f}.
\end{align*}
For an action $a$ of $G$ on $M$ it holds that  $(\mu^\sharp \otimes \id_{\cO_M}) a^\sharp=(\id_{\cO_{G}}\otimes a^\sharp) a^\sharp $ and $\rho_a (X)=( X_e \otimes \id_{\cO_M})a^\sharp$. Thus we obtain
\begin{align} \label{Eq expression diff pi}
\diff \pi_g(X_g)f &= \ev_c (\ev_g \otimes X_e \otimes \id_{\cO_C})(\id_{\cO_{Str(J)}}\otimes a^\sharp) a^\sharp \tilde{f} \nonumber \\
&= \ev_c ( X_e \otimes \id_{\cO_C})a^\sharp(\ev_g \otimes  \id_{\cO_C}) a^\sharp \tilde{f} \nonumber \\
&= \ev_c (\rho_a (X) a_g^\sharp \tilde{f}).
\end{align} 
The map from $\mathfrak{osp}(p-1,q-1|2n)$ to $T_x\mR^{p+q-2|2n}$ for $x\in \mR^{p+q-2|2n}$ given by $L_{ij} \mapsto \ev_x \circ L_{ij}$ has codimension one. This follows for example from the fact that for $i$ such that $\ev_x\circ \x_i\not=0$
\[
\{ \ev_x \circ \partial_{i}\} \cup \{ \ev_x \circ L_{kl} \mid 0<k,l\leq p+q-2+2n \}
\]
span $T_x \mR^{p+q-2|2n}$, so the codimension is less than or equal to one, while $L_{ij}(R^2)=0$ implies that the codimension is not zero.
Since $a_g^\sharp$ is surjective, we then conclude from equation~\eqref{Eq expression diff pi}
that \[ \dim \; \im \;(\diff \pi_g \mid_{ \mathfrak{osp}(J)}) = p+q-3+2n.\]
 Since the dimension of 
$T_{\abs{\pi}(g) c}C $ is equal to $p+q-3+2n$ we conclude that also  $\dim \; \im \;(\diff \pi_g)=p+q-3+2n$ and $\diff \pi_g$ is surjective.

Finally we have to show that $b$ is an action that reduces to the natural action $\abs{Str(J)} \times \abs{C} \to \abs{C}$ and $\pi \circ \mu = b\circ (\id_{Str(J)} \times \pi)$. We have $\abs{\pi}\circ \abs{\mu} (g_1,g_2) = (g_1g_2)c$ and $\abs{b}(\id_{\abs{Str(J)}} \times \abs{\pi})(g_1,g_2)= g_1(g_2c)$. We also compute
\begin{align*}
\mu^\sharp \circ \pi^\sharp f &= \mu^\sharp (\id_{\cO_{Str(J)}} \otimes \ev_c) b^\sharp f \\
&= (\id_{\cO_{Str(J)}} \otimes\id_{\cO_{Str(J)}} \otimes \ev_c ) (\mu^\sharp \otimes \id_{\cO_M})a^\sharp f \\
&=(\id_{\cO_{Str(J)}} \otimes\id_{\cO_{Str(J)}} \otimes \ev_c )(\id_{\cO_{Str(J)}} \otimes a^\sharp) a^\sharp f,
\end{align*}
and
\begin{align*}
 (\id_{\cO_{Str(J)}}\otimes \pi^\sharp ) b^\sharp f 
&= (\id_{\cO_{Str(J)}} \otimes\id_{\cO_{Str(J)}} \otimes \ev_c )(\id_{\cO_{Str(J)}} \otimes b^\sharp )b^\sharp f \\
&=(\id_{\cO_{Str(J)}} \otimes\id_{\cO_{Str(J)}} \otimes \ev_c )(\id_{\cO_{Str(J)}} \otimes a^\sharp) a^\sharp f.
\end{align*} 
Since $b$ is almost by definition an action and reduces to the natural action on $\abs{Str(J)} \times \abs{C} \to \abs{C}$, the proposition follows.
\end{proof}

\subsection{Restriction to the minimal orbit} \label{Section restriction to orbit}
Recall that for a simple Jordan algebra the representation constructed in Section \ref{Section: definition representation} corresponds to the representation of the conformal algebra considered in \cite{HKM}. In the latter paper it is also shown that, for certain characters, this representation can be restricted to an orbit.  We will show that for a specific character, also in our case the representation can be restricted to the minimal orbit defined in Section~\ref{Section minimal orbit}.

Consider the representation $\pi_\lambda$ constructed in Section \ref{Section: definition representation}. For $\lambda= 2-M$, with $M=p+q-2-2n$ the superdimension of $J$, we can restrict the representation $\pi_\lambda$ to the minimal orbit, as we will now show. 
We first prove that for this value of $\lambda$ the Bessel operators are tangential to the minimal orbit.
\begin{proposition}\label{Prop:BesselOperatorsTangential}
The Bessel operators are tangential to the minimal orbit, i.e.\ they  map $\langle R^2 \rangle$ into $\langle R^2 \rangle$, if and only if $\lambda = 2-M$, with $M$ the superdimension of $J$. 
\end{proposition}
\begin{proof}
Using the relations of Lemma~\ref{Lemma relations sl(2)} and equation~\eqref{Expression Bessel operators} we obtain
\begin{align*}
[\cB_{\lambda} (e_k) ,R^2] = \x_k (-2\lambda+ 4-2(p+q-2-2n)) +4R^2 \partial_{k}. 
\end{align*}
We conclude that $\langle R^2 \rangle$ gets mapped into $\langle R^2 \rangle$ if and only if $\lambda= 2-M$.
\end{proof}
The operators $\x_i$, $L_{ij}$ and $\mE$ are also tangential to the orbit. Hence $\langle R^2 \rangle $ gives a subrepresentation of $\pi_{2-M}$. Using the embedding $j$ defined in Section~\ref{Section C is an orbit}, we set 
 \[
 \pi_C(X) f = j^\sharp (\pi_{2-M}(X) \tilde{f})
 \]
 for $f$ in $\Gamma(\cO_C)$ and $\tilde{f}$ a representative from $f$ in $\Gamma(\cO_{\mA(J^\ast)_0})$, i.e.\ $j^\sharp(\tilde{f})=f$. Since all the operators occurring in $\pi_{2-M}$ are tangential to $C$, this gives a well defined  quotient representation
 \[
  \pi_C\colon \TKK(J) \to \Gamma (\cD_C) 
 \]
 on the orbit $C$. Here $\Gamma (\cD_C)$ acts by differential operators on $\Gamma(\cO_C)$, hence we found a representation of $\TKK(J)$ on functions on the minimal orbit.

\section{Integration to the conformal group}
\label{Integration to group level}
We introduce the notations
  \begin{align*}
\mg &:= \TKK(J) = \mathfrak{osp}(p,q|2n), & \mk &:= \mathfrak{osp}(p|2n)\oplus \mathfrak{so}(q),\\
 \ck &:= \mathfrak{so}(p)\oplus \mathfrak{so}(q) \oplus \mathfrak{u}(n),& \mk_0 &:= \mk  \cap \istr(J) = \mathfrak{osp}(p-1|2n) \oplus \mathfrak{so}(q-1). 
\end{align*}
Then $\ck$ is a maximal compact subalgebra of the even part of $\mk$ and also  a maximal compact subalgebra of the even part of $\mg$. 

In this section we will integrate a subrepresentation $\pi_C$ of $\mg$ on $\Gamma(\cO_C)$ constructed in Section \ref{Section restriction to orbit} to the conformal group $OSp(p,q|2n)$ using the concept of Harish-Chandra supermodules. To be able to do this we need a $(\mg,\ck)$-module $W$ of $\ck$-finite vectors. As an intermediate step, we will first look for a $(\mg,\mk)$-module of $\mk$-finite vectors.
\begin{remark}{\rm
Our choice of $\mathfrak{k}'$ seems arbitrary, and one might as well work with $\mathfrak{osp}(q|2n)\oplus\mathfrak{so}(p)$. However, the same techniques can be used in this case which leads to similar results. Since $\mathfrak{osp}(p,q|2n)\cong\mathfrak{osp}(q,p|2n)$ it is enough to consider one of the two possible choices.
}
\end{remark}

We start this section by introducing $\mk_0$-invariant radial superfunctions. 
\subsection{Radial superfunctions}\label{Section radial superfunctions}

On $\mR^{p-1}\oplus \mR^{q-1}\oplus \mR^{2n}$ we consider the supersymmetric, non-degenerate, even bilinear form $\beta$  of signature $(p-1,q-1|2n)$ associated to $\mathfrak{osp}(p-1,q-1|2n)$. Choose a basis  $(e_i)_i$, $(f_i)_i$, $(\theta_i)_i$ of $\mR^{p-1}\oplus \mR^{q-1}\oplus \mR^{2n}$  such that \[ \langle e_i , e_j \rangle_{\beta} = \delta_{ij}, \qquad \langle f_i , f_j \rangle_{\beta} =-\delta_{ij}, \qquad \langle e_i, f_j \rangle_{\beta}=0.\]
Let $e^i, f^i$ and $\theta^i$ be the right duals of $e_i, f_i$ and $\theta_i$ with respect to our form.  Then $e^i=e_{i}$ and $f^i=-f_i$.
We will use $x_i$, $y_i$ and $\theta_i$ as the coordinates on $\mA((\mR^{p+q-2|2n})^\ast) \cong\mA(\mR^{p+q-2|2n})$ corresponding to this basis.
Set \[ s^2 = \sum_{i=1}^{p-1}x_i^2,\quad t^2= \sum_{j=1}^{q-1}y_j^2, \quad \theta^2 = \sum_{i=1}^{2n}\theta^i \theta_i.\]

 For a function $h\colon \mR \to \mR$, $h\in \cC^{2n}(\mR_{(0)})$ and a superfunction $f= f_0+ \sum_{I\not=0} f_I \theta^I$, with $f_0$ and $f_I$ in $\cC^\infty(\mR^m_{(0)})$, a new superfunction $h(f)$ in $\cC^\infty(\mR^m_{(0)})\otimes \Lambda^{2n}$  is defined in \cite[Definition 3]{Invariant functions}
\[
h(f) := \sum_{j=0}^{2n}  \frac{(\sum_{I\not=0} f_I \theta^I)^{j}}{j!} h^{(j)}(f_0).
\]

We will use this definition to define radial functions depending on the superfunction $\abs{X}$
\[
\abs{X} = \sqrt{\frac{s^2+t^2+\theta^2}{2}}.
\]
 Note that such a function $h(\abs{X})$ is  $\mk_0$ invariant since $\abs{X}$ is  $\mk_0$ invariant.

\begin{lemma}\label{Lemma properties Bessel functions}
Consider $h\colon \mR \to \mR$, $h\in \cC^{2n+2}(\mR_{(0)})$. The radial function $h(\abs{X})$  satisfies
\begin{align*}
\partial_{x^i} h(\abs{X})&= \frac{x_i}{2\abs{X}}\partial_{\abs{X}} h(\abs{X}), \qquad  \partial_{\theta^i} h(\abs{X}) = \frac{\theta_i}{2\abs{X}}\partial_{\abs{X}}h(\abs{X}),  
\\
  \partial_{y^i} h(\abs{X}) &= -\frac{y_i}{2\abs{X}}\partial_{\abs{X}}h(\abs{X}),  \qquad \mE h(\abs{X})= \abs{X} \partial_{\abs{X}}h(\abs{X}),
\\ 
  \Delta h(\abs{X}) & =  \tfrac{p-q-2n}{2\abs{X}}\partial_{\abs{X}}h(\abs{X}) + \tfrac{R^2}{4\abs{X}^2} \left(\partial_{\abs{X}}^2h(\abs{X})-\frac{\partial_{\abs{X}} h(\abs{X})}{\abs{X}}\right), 
\end{align*}
and
\begin{align*}
(\cB_\lambda(x_i)-x_i) h(\abs{X}) 
&= x_i \left( \partial^2_{\abs{X}}h(\abs{X}) -(p-q-2n+\lambda) \frac{\partial_{\abs{X}}h(\abs{X})}{2 \abs{X}}- h(\abs{X})\right) ,
\\
(\cB_\lambda(\theta_i)-\theta_i) h (\abs{X})
&= \theta_i \left( \partial^2_{\abs{X}}h(\abs{X}) -(p-q-2n+\lambda) \frac{\partial_{\abs{X}}h(\abs{X})}{2 \abs{X}}- h(\abs{X})\right) ,
\\
(\cB_\lambda(y_i)+y_i) h (\abs{X})
&= -y_i \left( \partial^2_{\abs{X}}h (\abs{X})+(p-q-2n-\lambda) \frac{\partial_{\abs{X}}h(\abs{X})}{2 \abs{X}}- h(\abs{X})\right),
\end{align*}
where the three last equalities are modulo $R^2$.
\end{lemma}
\begin{proof}
If $f$ is an even superfunction and $h \in \cC^{2n+1}(\mR_{(0)})$, then we have the chain rule \[ \partial_{\x^i} h(f) = \partial_{\x^i}(f) h'(f).\] Since $\abs{X}$ is an even superfunction, the first three equations follow from this chain rule and
\[
\partial_{x^i} \abs{X}= \frac{x_i}{2\abs{X}} \qquad \partial_{y^i} \abs{X}= -\frac{y_i}{2\abs{X}} \qquad \partial_{\theta^i} \abs{X}= \frac{\theta_i}{2\abs{X}}.
\]
The other equations are then a straightforward corollary from these three equations.
\end{proof}

\subsection{The $(\mg,\mk)$-module W} \label{Section The k module W}
For our definition of $W$ we start from a general $\mk_0$-invariant function  on $\Gamma(\cO_C)$. Acting on this function with basis elements of $\mk$ not in $\mk_0$ leads to the differential equation~\eqref{Bessel function dif rel} that the modified K-Bessel functions satisfy  (see equation~\eqref{equation action bessel operators}  in the proof of Lemma~\ref{Lemma: action of Bessel operators}).  So a natural ansatz for $W$ is the $U(\mg)$-module generated by $\widetilde{K}_\alpha$, the renormalised modified Bessel function of the third kind introduced in \ref{Section Bessel functions}.. 
Set \begin{align}\label{definition mu and nu}
 \mu= \max (p-2n-3,q-3), \qquad \nu = \min (p-2n-3, q-3).
 \end{align}
We also set 
 $ \mR^{\mu+2}= \mR^{p-1|2n}$ and $\mR^{\nu+2}= \mR^{q-1}$ if $p-2n \geq q$ and $ \mR^{\mu+2}= \mR^{q-1}$ and $\mR^{\nu+2}= \mR^{p-1|2n}$ if $p-2n < q$.
 
Let $\Lambda^{\mu,\nu}_{2,j}(\abs{X})$ be the radial superfunction defined using the generalised Laguerre function $\Lambda^{\mu,\nu}_{2,j}(z)$ introduced in \ref{definition Lambda mu nu}. Note that for $j=0$, we find  $\Lambda^{\mu,\nu}_{2,0}(\abs{X})= \frac{1}{\Gamma(\frac{\mu+2}{2})} \widetilde{K}_{\frac{\nu}{2}}(\abs{X}).$

Define \begin{align}\label{Definition W} 
 W := U(\mg) ( \widetilde{K}_{\frac{\nu}{2}} (\abs{X})+\langle R^2 \rangle) \subset \Gamma(\cO_C),
\end{align}

where  the $\mg$-module structure is given by the representation $\pi_C$.  In the following, we will always work modulo $\langle R^2 \rangle$ and  drop $\langle R^2 \rangle$ in our notation. So we write for example $\widetilde{K}_{\frac{\nu}{2}} (\abs{X})$  for $\widetilde{K}_{\frac{\nu}{2}} (\abs{X})+\langle R^2 \rangle$.
\begin{theorem} \label{Theorem structure W}
Assume $\nu \not\in -2\mN$. 
\begin{enumerate}
\item The decomposition of $W$ as $\mk$-module is given by
\begin{align*}
W&= \bigoplus_{j=0}^\infty W_j, & \text{ where }\quad
  W_j =  U(\mk) \Lambda^{\mu,\nu}_{2,j} (\abs{X}). 
 \end{align*}
 \item Assume $q\not=3$ and $p\not= 3$. 
Then $W$ is always indecomposable. It is furthermore a simple $\mg$-module if   $p+q$ is odd or $\mu+\nu\geq 0$ or $q=p-2n=2$ or $p=2$. 
\item If $p+q$ is even, then $W_j$ and thus also $W$ is $\mk$-finite.  
An explicit decomposition of $W_j$ into irreducible $\mk_0$-modules is given by 
  \begin{align} \label{eq structure W_j even}
  W_j = \bigoplus_{k=0}^j \bigoplus_{l=0}^{\tfrac{\mu-\nu}{2} +j} \Lambda^{\mu+2k,\nu + 2l}_{2,j-k} (\abs{X}) \cH_k (\mR^{\mu+2}) \otimes \cH_l(\mR^{\nu+2}).
 \end{align}
 Here $\cH_k(\mR^{\mu+2})$ and $\cH_l(\mR^{\nu+2})$ are spaces of spherical harmonics introduced in Section~\ref{Section spherical harmonics}.
 Furthermore, we also have the following $\mk$-isomorphism
\begin{align*}
W_j &\cong \cH_{j} (\mR^{\mu+3}) \otimes \cH_{\frac{\mu-\nu}{2}+j} (\mR^{\nu+3}).
\end{align*}
	 
 \item If $p+q$ is odd, the decomposition of $W_j$ into irreducible $\mk_0$-modules is given by
  \begin{align} \label{eq structure W_j odd}
  W_j = \bigoplus_{k=0}^j \bigoplus_{l=0}^{\infty} \Lambda^{\mu+2k,\nu + 2l}_{2,j-k} (\abs{X}) \cH_k (\mR^{\mu+2}) \otimes \cH_l(\mR^{\nu+2}).
 \end{align}
If $p,q\not=2$ then $W_j$ is not $\mk$-finite, while for $p=2$ or $q=2$, $W_j$ is still $\mk$-finite.
\end{enumerate}
\end{theorem}
\begin{proof}
This is a combination of Proposition~\ref{Prop: structure Wj}, Proposition~\ref{Structure W}, Corollary~\ref{Corollary W simple} and Proposition~\ref{Prop: structure Wj II}.
\end{proof}
\begin{remark}{\rm
This theorem gives new information  even  in the classical case (i.e.\ $n=0$). Namely, for $\mg = \mathfrak{so(p,q)}$ and   $p+q$ even, it is well-known that the minimal representation $W$ of $\mathfrak{g}$ decomposes as 
\[
W \simeq \bigoplus_{j=0}^\infty \cH_{j} (\mR^{\mu+3}) \otimes \cH_{\frac{\mu-\nu}{2}+j} (\mR^{\nu+3}),
\]
and that for $p+q$ odd, $W_0$ is infinite-dimensional, but in the Schr\"{o}\-ding\-er model only few $\mathfrak{k}$-finite vectors have been made explicit: The $\mk_0$-invariant vectors $\Lambda^{\mu,\nu}_{2,j}$ were given in \cite[Corollary 8.2]{HKMM}, which is the case $l=k=0$. Further, the case $k=j$ and $l$ arbitrary is contained in \cite[Theorem 3.1.1]{KM2}. However, the decomposition of $W_j$ given in \eqref{eq structure W_j even} and \eqref{eq structure W_j odd}, which makes explicit all $\mk$-finite vectors in the Schr\"{o}dinger model if $p+q$ is even, is to the best of our knowledge new.
}
\end{remark}
\begin{remark}{\rm
We have $\widetilde{K}_{\frac{\nu}{2}}(\abs{X})= \Gamma(\frac{\mu+2}{2}) \Lambda^{\mu,\nu}_{2,0} (\abs{X}).$ Hence 
\begin{align*}
W_0 = U(\mk) \Lambda^{\mu,\nu}_{2,0} (\abs{X}) = U(\mk) \widetilde{K}_{\frac{\nu}{2}}(\abs{X}).
\end{align*}
}
\end{remark}

We start with a lemma which gives the action of some elements of $\mk$ on a combination of Laguerre superfunctions with spherical harmonics. 
First we combine  $\phi_k \in \cH_k(\mR^{p-1|2n})$ and  $z_i \in \cP_1(\mR^{p-1|2n})$ to obtain spherical harmonics of degree $k+1$ and $k-1$.  Namely we set
\begin{align} \label{def phi+}
\begin{aligned}
\phi_{k+1,i}^+ &=   \x_i \phi_k - \frac{s^2+\theta^2}{p-3-2n+2k} \partial_{\x^i} \phi_k,  \qquad \text{and} \qquad
\phi^-_{k-1,i} &= \frac{1}{p-3-2n+2k} \partial_{\x^i} \phi_k, 
\end{aligned}
\end{align} 
for  $p-3-2n+2k\not=0.$ 
One checks that  $\phi_{k+1,i}^+$ and $\phi^-_{k-1,i}$ are contained  in $\cH_{k+1}(\mR^{p-1|2n})$ and in $\cH_{k-1}(\mR^{p-1|2n})$ respectively.
Similarly for $\phi_k  \in \cH_k(\mR^{q-1})$, $y_i \in \cP_1(\mR^{q-1})$ and $q-3+2k\not=0$, set
\[ \phi_{k+1,i}^+=  -y_i \phi_k - \frac{t^2}{q-3+2k} \partial_{y^i} \phi_k, \qquad \phi^-_{k-1,i} = \frac{1}{q-3+2k} \partial_{y^i} \phi_k. \]
In this case  $\phi_{k+1,i}^+$ and $\phi^-_{k-1,i}$ are contained  in $\cH_{k+1}(\mR^{q-1})$ and in $\cH_{k-1}(\mR^{q-1})$ respectively.  
\begin{lemma} \label{Lemma: action of Bessel operators}
Let $\phi_k \in \cH_k(\mR^{\mu+2})$ and $\psi_l \in \cH_l (\mR^{\nu+2})$. 
Set 
\begin{align*}
B^+_i &:= \cB_\lambda(\x_i)-\x_i  \qquad \text{ and }\qquad  B^-_i:=\cB_\lambda(y_i)+y_i &\text{ if } p-2n \geq q,
 \\
  B^+_i&: = \cB_\lambda(y_i)+y_i  \qquad \text{ and }\qquad  B^-_i:=\cB_\lambda(\x_i)-\x_i &\text{  if } p-2n <q,
\end{align*} 
where $z_i=x_i$ and $z_{i+p-1}=\theta_i$.
Then for $\nu+2l\not=0$ we have
\begin{align*}
B^+_i(\phi_k \psi_l\Lambda^{\mu+2k,\nu+2l}_{2,j-k}(\abs{X})  )&
=  (j+\mu+k+1) \phi_{k+1,i}^+ \psi_l \Lambda^{\mu+2(k+1),\nu+2l}_{2,j-(k+1)}(\abs{X}) \\
&+ 4(j-k+1) \phi_{k-1,i}^- \psi_l \Lambda^{\mu+2(k-1),\nu+2l}_{2,j-(k-1)}(\abs{X}) 
\\
B^-_i (\phi_k \psi_l\Lambda^{\mu+2k,\nu+2l}_{2,j-k}(\abs{X})) &
 = -(j+\tfrac{\mu-\nu}{2}-l) \phi_k \psi_{l+1,i}^+ \Lambda^{\mu+2k,\nu+2(l+1)}_{2,j-k}(\abs{X})
 \\
 &- 4(j+\tfrac{\mu+\nu}{2}+l) \phi_k \psi_{l-1,i}^- \Lambda^{\mu+2k,\nu+2(l-1)}_{2,j-k}(\abs{X}).
\end{align*}

\end{lemma}
\begin{proof}
We will only prove the case $B_i^+$ for  $p-2n \geq q$. The other cases are similar. 

We first observe that the action of the Bessel operator on a product is given by
\begin{align}\label{Eq Product rule for Bessel operators}
\cB_\lambda(\x_k) &(fg)=  (-\lambda + 2 \mE) \partial_k(fg) - \x_k \Delta (fg) &
\\ \nonumber
&=(\cB_\lambda(\x_k) f) g + (-1)^{\abs{f}\abs{k}} f (\cB_\lambda(\x_k)g) 
+ 2 (-1)^{\abs{f}\abs{k}}\mE(f) \partial_k (g) 
\\
&\quad + 2 \partial_k (f) \mE(g) - 2 \x_k (-1)^{\abs{f}\abs{j}} \beta^{ij}  \partial_i(f) \partial_j(g). \nonumber
\end{align}
Note that this is a special case of Proposition \cite[Proposition 4.2]{BC}.

From this product rule and Lemma~\ref{Lemma properties Bessel functions} we obtain
\begin{align} \label{equation action bessel operators}
& (\cB_\lambda(\x_i) -\x_i) (\Lambda^{\mu+2k,\nu+2l}_{2,j-k} \phi_k \psi_l) \nonumber 
\\
&  = \left((\cB_\lambda(\x_i)-\x_i) \Lambda^{\mu+2k,\nu+2l}_{2,j-k} \right)\phi_k \psi_l+ \Lambda^{\mu+2k,\nu+2l}_{2,j-k} \cB_\lambda(\x_i)( \phi_k \psi_l) \nonumber
\\
 &\quad + 2 \mE(\Lambda^{\mu+2k,\nu+2l}_{2,j-k})  \partial_i (\phi_k \psi_l) + 2 \partial_i (\Lambda^{\mu+2k,\nu+2l}_{2,j-k}) \mE(\phi_k\psi_l) \nonumber
 \\
 & \quad -2 \x_i \beta^{ab} \partial_a(\Lambda^{\mu+2k,\nu+2l}_{2,j-k}) \partial_b (\phi_k\psi_l) \nonumber
 \\
&= \x_i \left((\Lambda^{\mu+2k,\nu+2l}_{2,j-k})''+ \frac{q-2}{\abs{X}} (\Lambda^{\mu+2k,\nu+2l}_{2,j-k})' - \Lambda^{\mu+2k,\nu+2l}_{2,j-k}\right) \phi_k \psi_l 
\\
&  \quad+ \Lambda^{\mu+2k,\nu+2l}_{2,j-k} (\mu+\nu+2k+2l) \partial_i(\phi_k) \psi_l  +2\mE (\Lambda^{\mu+2k,\nu+2l}_{2,j-k}) \partial_i(\phi_k) \psi_l \nonumber
\\
& \quad+ \frac{\x_i}{\abs{X}}(k+l) \Lambda^{\mu+2k,\nu+2l}_{2,j-k})' \phi_k \psi_l - \frac{\x_i}{\abs{X}} (k-l) (\Lambda^{\mu+2k,\nu+2l}_{2,j-k})' \psi_k \phi_l \nonumber
\\
&= (\phi_{k+1,i}^+ +\abs{X}^2 \phi_{k-1,i}^-) \psi_l (j-k+\mu+2k+1) \Lambda^{\mu+2k+2,\nu+2l}_{2,j-k-1} \nonumber
\\
& \quad + (\mu+2k)\phi_{k-1,i}^- \psi_l \left( \mu +2k +\nu+2l +2\mE \right) \Lambda^{\mu+2k,\nu+2l}_{2,j-k} \nonumber
\\ \nonumber
&= (j+\mu+k+1) \phi_{k+1,i}^+ \psi_l \Lambda^{\mu+2(k+1),\nu+2l}_{2,j-(k+1)} + 4(j-k+1) \phi_{k-1,i}^- \psi_l \Lambda^{\mu+2(k-1),\nu+2l}_{2,j-(k-1)}.\nonumber
\end{align}
In the last two steps we used Corollary~\ref{Properties generalised Laguerre polynomials} and $s^2+\theta^2=\abs{X}^2 \mod R^2.$
\end{proof}

\begin{proposition} \label{Prop: structure Wj}
 Assume   $\nu\not\in -2\mN$. 
The decomposition of  $W_j=U(\mk) \Lambda^{\mu,\nu}_{2,j} (\abs{X})$  as a $\mk_0$-module is given by
\begin{align*}
W_j &= \bigoplus_{k=0}^j \bigoplus_{l=0}^{\tfrac{\mu-\nu}{2} +j} \Lambda^{\mu+2k,\nu + 2l}_{2,j-k} (\abs{X}) \cH_k (\mR^{\mu+2}) \otimes \cH_l(\mR^{\nu+2}) \text{ if } p + q \text{ is even}, \\
W_j &= \bigoplus_{k=0}^j \bigoplus_{l=0}^{\infty} \Lambda^{\mu+2k,\nu + 2l}_{2,j-k} (\abs{X}) \cH_k (\mR^{\mu+2}) \otimes \cH_l(\mR^{\nu+2})   \text{ if } p + q \text{ is odd}.
\end{align*} 
If $p\not=3$ and $q\not=3$, then $W_j$ is an indecomposable $\mk$-module. If we also have $p+q$ odd or $ j+\frac{\mu+\nu}{2}\geq 0$, then $W_j$ is a simple $\mk$-module. If $p=2$ or $q=2$, then $W_j$ is always finite-dimensional. 
\end{proposition}
\begin{proof}
Denote by $(k,l)$ the space $\Lambda^{\mu+2k,\nu + 2l}_{2,j-k} (\abs{X}) \cH_k (\mR^{\mu+2}) \otimes \cH_l(\mR^{\nu+2})$.
We can extend a basis of $\mk_0$ with the elements $\imath(\cB_\lambda(x_i)-x_i)$, $i=1,\ldots, p-1$, $\imath(\cB_\lambda(y_i)+y_i)$, $i=1,\ldots, q-1$, $\imath(\cB_\lambda(\theta_i)-\theta_i)$, $i=1,\ldots, 2n$ to get a basis of $\mk$.

First assume $q>3$ and $p>3$. In that case $\cH_k (\mR^{\mu+2}) \otimes \cH_l(\mR^{\nu+2})$ is an irreducible $\mk_0$-module. By Lemma~\ref{Lemma: action of Bessel operators}, we can use $B_i^+$ to go from $(k,l)$ to $(k+1,l)$ and $(k-1,l)$.  Note that $(k,l)=0$ if $k>j$, because then $\Lambda^{\mu+2k,\nu+2l}_{2,j-k}(\abs{X})=0$. Similar we can use $B_i^-$ to go from $(k,l)$ to $(k,l+1)$ and $(k,l-1)$ if $l\not= j+\frac{\mu-\nu}{2}$ and $l \not= -j-\frac{\mu+\nu}{2}$. If $l=j+\frac{\mu-\nu}{2}$, then $B_i^-$ maps $(k,l)$ only to $(k,l-1)$ since the coefficient of the part in $(k,l+1)$ is zero. If $l=-j-\frac{\mu+\nu}{2}$ then $B_i^-$ maps $(k,l)$ to $(k,l+1)$ since now the coefficient of the part in $(k,l-1)$ is zero. Not that these two exceptional case can not occur if $p+q$ is odd, because then $\mu+\nu$ is also odd. The last case can also not occur if $j+\frac{\mu+\nu}{2}\geq 0$. Observe that $W_j$ is the $\mk$-module generated by $(0,0)$. Combining all this, we conclude that 
\[
W_j = \bigoplus_{k=0}^j \bigoplus_{l=0}^{\infty} (k,l),
\]
if $p+q$ is odd. This module is $\mk$-simple since the $(k,l)$ are simple $\mk_0$-modules and we can use $B_i^+$ and $B_i^-$ to go from $(k,l)$ to $(k',l')$ for all $0\leq k, k'\leq j$ and $0\leq l,l'\leq \infty$ . For $p+q$ even, we obtain
\[
W_j = \bigoplus_{k=0}^j \bigoplus_{l=0}^{\tfrac{\mu-\nu}{2} +j} (k,l),
\]
which is simple if $j+\frac{\mu+\nu}{2}\geq 0$. Otherwise  $W_j$ is still indecomposable but it has \[ \bigoplus_{k=0}^j \bigoplus_{l=-j-\tfrac{\mu+\nu}{2}}^{\tfrac{\mu-\nu}{2} +j} (k,l) \] as a simple submodule.

If $q=3$, then $\cH_k(\mR^{q-1})$ is no longer irreducible but decomposes in two submodules. However a real polynomial of degree $k$ has components in both the subspaces of $\cH_{k}(\mR^{q-1})$ and if $\phi_{k-1}$ is real, also $\phi_{k,i}^+$ will be real. Therefore we conclude that the whole $\cH_{k}(\mR^{q-1})$ is contained in $W_j$ and we still obtain
\begin{align*}
W_j &= \bigoplus_{k=0}^j \bigoplus_{l=0}^{\infty} (k,l)
\quad  \text{ if $p+q$ is odd and }\quad 
W_j = \bigoplus_{k=0}^j \bigoplus_{l=0}^{\tfrac{\mu-\nu}{2} +j} (k,l)\quad   \text{ if $p+q$ is even.}
\end{align*}
However, these modules are no longer indecomposable, since $\phi_{k} \in  \mC (x\pm \imath y)^k   $ implies $\phi_{k-1,i}^- \in  \mC (x\pm \imath y)^{k-1}  $ and $\phi_{k+1,i}^+ \in  \mC (x\pm \imath y)^{k+1} $. So $W_j$ decomposes in two submodules.  

For $q=2$ we have $\cH_k(\mR^{q-1})=0$ for $k\geq 2$. For $p-2n\geq 2$ one still obtains, 
\begin{align*}
W_j = \bigoplus_{k=0}^j \bigoplus_{l=0}^{\infty} (k,l)
\quad  \text{ if $p$ is odd and } \quad W_j = \bigoplus_{k=0}^j \bigoplus_{l=0}^{\tfrac{\mu-\nu}{2} +j} (k,l) \quad  \text{ if $p$ is even}
\end{align*}
but with $(k,l)=0$ if $l\geq2$. The module $W_j$ is finite-dimensional and simple. For $p-2n< 2$, the assumption $\nu \not\in -2\mN$ implies that $p$ is even. We then get
\begin{align*}
W_j = \bigoplus_{k=0}^{\min(j,1)} \bigoplus_{l=0}^{\tfrac{\mu-\nu}{2} +j} (k,l),
\end{align*}
which is simple if $j+\frac{\mu+\nu}{2}\geq0$ or $j=0$. For $j+\frac{\mu+\nu}{2}<0$ and $j\not=0$ it is indecomposable with $\bigoplus_{k=0}^1 \bigoplus_{l=-j-\frac{\mu+\nu}{2}}^{\tfrac{\mu-\nu}{2} +j} (k,l)$ as simple submodule.
For $p=3$, the assumption $\nu\not\in -2\mN$ implies $2n=0$ and $q=2$. Then this case is considered in the case $q=2$. However now $W_j$ is not simple since $\cH(\mR^{p-1})$ is not simple. For $p=2$ we have $\cH_k(\mR^{1|2n})=0$ if $k>2n+1.$ Therefore we get 
\begin{align*}
W_j = \begin{cases} \bigoplus_{k=0}^j \bigoplus_{l=0}^{2n+1}(k,l) \qquad\qquad\qquad\;\; \text{ if } q \text{ is odd} \\
\bigoplus_{k=0}^j \bigoplus_{l=0}^{\min(j+\frac{\mu-\nu}{2},2n+1)}(k,l) \qquad \text{ if } q \text{ is even}.\end{cases}
\end{align*}
Then $W_j$ is simple if $q\not=3$. For $q=3$, it decomposes in two simple submodules. 
\end{proof}

\begin{proposition} \label{Structure W}
If  $\nu \not\in -2\mN$, then we have
\[
W= \bigoplus_{j=0}^\infty W_j.
\]
\end{proposition}
\begin{proof}
Using $\pi_\lambda(L_e)= \frac{\lambda}{2}-\mE$ and Proposition~\ref{action Le on Laguerre functions}, we obtain for $\lambda=-( \mu+\nu+2)$ 
\begin{align*}
&\pi_\lambda(-L_e)  \left(\phi_k\psi_l \Lambda^{\mu+2k,\nu+2l}_{2,j-k}(\abs{X})\right) \\
&= \phi_k\phi_l\left(\left(\mE+\frac{\mu+\nu+2}{2}\right) \Lambda^{\mu+2k,\nu+2l}_{2,j-k}(\abs{X})\right)
+(k+l)\phi_k\psi_l \Lambda^{\mu+2k,\nu+2l}_{2,j-k}(\abs{X})
\\ 
&= \tfrac{(j-k+1)(j+k+\mu+1)}{2j+\mu+1}\phi_k\psi_l\Lambda^{\mu+2k,\nu+2l}_{2,j+1-k}(\abs{X})
 - \tfrac{(j+l+\frac{\mu+\nu}{2})(j-l+\frac{\mu-\nu}{2})}{2j+\mu+1}\phi_k\psi_l\Lambda^{\mu+2k,\nu+2l}_{2,j-1-k}(\abs{X}),
\end{align*}
if  $j\not=0$. If $j=0$, then 
\begin{align*}
\pi_\lambda(-L_e)  &\left(\psi_l \Lambda^{\mu,\nu+2l}_{2,0}(\abs{X})\right) = \psi_l\Lambda^{\mu,\nu+2l}_{2,1}.
\end{align*}
So the result of the action of $L_e$  on $W_j$ has a non-zero component in $W_{j+1}$. By repeatedly acting by $L_e$ on $W_0$ we obtain \[ \bigoplus_{j=0}^\infty W_j\subseteq W.\]

We will now show that the action of an element $X$ in $\mg$ on $W_j$ is contained in $W_{j-1}\oplus W_j \oplus W_{j+1}$. Here we set $W_{-1}=0$.
We have $\mg=\mk+ \mathfrak{p}$ with $\mathfrak{p}=[\mk,[\mk,L_e]]$. Hence, we can write every $X \in \mg$ as
\[
X= Y_1  + [Y_2, [ Y_3,L_e]], 
\]
where $Y_1,Y_2,Y_3$ are in $\mk$. 
Because $\mk$ leaves $W_j$ invariant and $L_e$ maps $W_j$ into $W_{j-1}\oplus W_{j+1}$, also $X$ maps $W_j$ into $W_{j-1}\oplus W_j\oplus W_{j+1}$. Therefore 
 \[ W\subseteq \bigoplus_{j=0}^\infty W_j,\]
 which proves the proposition.
\end{proof}
\begin{corollary}\label{Corollary W simple}
Assume $\nu \not\in -2\mN$ and $q\not=3$ and $p\not= 3$. 
Then $W$ is a simple $\mg$-module if   $p+q$ is odd or $\mu+\nu\geq 0$ or $q=p-2n=2$ or $p=2$. Otherwise it is still indecomposable and has
\begin{align*}
\bigoplus_{j=0}^\infty \bigoplus_{k=0}^j \bigoplus_{l=\min(0, -j-\frac{\mu+\nu}{2})}^{\tfrac{\mu-\nu}{2} +j} \Lambda^{\mu+2k,\nu + 2l}_{2,j-k} (\abs{X}) \cH_k (\mR^{\mu+2}) \otimes \cH_l(\mR^{\nu+2}),
\end{align*}
as simple submodule.
\end{corollary}
\begin{proof}
Remark that $L_e$ maps the simple submodule of $W_j$ into the simple submodules of $W_{j-1}$ and $W_{j+1}$ for $\frac{\mu+\nu}{2}+j <0$. Then the corollary follows simply from Proposition~\ref{Prop: structure Wj} and (the proof of) Proposition~\ref{Structure W}.
\end{proof}

\begin{proposition} \label{Prop: structure Wj II}
For $p+q$ even and $\nu\not\in -2\mN$ we have
\begin{align*}
W_j &\cong \cH_{j} (\mR^{\mu+3}) \otimes \cH_{\frac{\mu-\nu}{2}+j} (\mR^{\nu+3}),
\end{align*}
as $\mk$-module. 
\end{proposition}
Let $s_0$ and $t_0$ be the extra even coordinates from extending $\mR^{\mu+2}$ and $\mR^{\nu+2}$ to $\mR^{\mu+3}$ and $\mR^{\nu+2}$. Define also
$S= \sqrt{s^2+\theta^2+s_0^2}$, $T= \sqrt{t^2+t_0^2}$ if $p-2n\geq q$ or $S= \sqrt{t^2+s_0^2}$, $T= \sqrt{s^2+\theta^2+t_0^2}$ if $p-2n< q$.

The $\mk$-intertwining map is explicitly given by $\Phi \colon$
\begin{align} \label{definition isomorphism}
\begin{aligned}
 \bigoplus_{k=0}^j \bigoplus_{l=0}^{\frac{\mu-\nu}{2}+j} \Lambda^{\mu+2k,\nu + 2l}_{2,j-k}(\abs{X}) \cH_k (\mR^{\mu+2}) \otimes \cH_l(\mR^{\nu+2}) 
  \to \cH_{j} (\mR^{\mu+3}) \otimes \cH_{\frac{\mu-\nu}{2}+j} (\mR^{\nu+3})  
\\
\phi_k \psi_l \Lambda^{\mu+2k,\nu + 2l}_{2,j-k}(\abs{X}) 
\mapsto c_kd_l \phi_k \psi_l  S^{j-k} T^{j-l+\frac{\mu-\nu}{2}} \widetilde{C}^{k+\frac{\mu+1}{2}}_{j-k}\left(\frac{s_0}{S} \right) \widetilde{C}^{l+\frac{\nu+1}{2}}_{j-l + \frac{\mu-\nu}{2}}\left(\frac{t_0}{T} \right),
\end{aligned}
\end{align}
where $\widetilde{C}_{n}^{\lambda}(z)$ are the normalised Gegenbauer polynomial introduced in \ref{Sect Gegenbauer polynomials}.
The constants $c_k$ and $d_l$ are given by
\[
c_k = \frac{(-4\imath )^k }{(\mu+j+1)_k }, \qquad d_l = \frac{(4\imath )^l }{(-j-\frac{\mu-\nu}{2})_l }
\]
where we used the Pochhammer symbol $(a)_k = a (a+1) (a+2) \cdots (a+k-1). $

\begin{proof}
Assume $p-2n\geq q$; the case $p-2n<q$ is again similar.
A straightforward calculation shows that the right-hand side of \eqref{definition isomorphism} is indeed contained in $\cH_{j} (\mR^{p|2n}) \otimes \cH_{\frac{\mu-\nu}{2}+j} (\mR^{q})$.
We need to prove
\begin{enumerate} 
\item  $ L_{ij} \left(\Phi\left(\phi_k \psi_l \Lambda^{\mu+2k,\nu + 2l}_{2,j-k}\right)\right) = \Phi \left( L_{ij}\left(\phi_k \psi_l \Lambda^{\mu+2k,\nu + 2l}_{2,j-k}\right)  \right) $ for $L_{ij} \in \mk_0$,\label{enum 1}
\item $2L_{k0} \left(\Phi\left(\phi_k \psi_l \Lambda^{\mu+2k,\nu + 2l}_{2,j-k}\right)\right) = \Phi \left( \imath (\cB_\lambda(\x_k)-\x_k) \left(\phi_k \psi_l \Lambda^{\mu+2k,\nu + 2l}_{2,j-k}\right)  \right) $ \label{enum 2}
\item 
$2L_{kq} \left(\Phi\left(\phi_k \psi_l \Lambda^{\mu+2k,\nu + 2l}_{2,j-k}\right)\right) = \Phi \left( -\imath (\cB_\lambda(y_k)+y_k) \left(\phi_k \psi_l \Lambda^{\mu+2k,\nu + 2l}_{2,j-k}\right)  \right) $. \label{enum 3}
\end{enumerate}
Because $S$ and $T$ are $\mk_0$-invariant \eqref{enum 1} follows immediately. The cases  \eqref{enum 2} and \eqref{enum 3} can be shown by a straightforward calculation using Lemma~\ref{Lemma: action of Bessel operators} and the properties of the the Gegenbauer polynomial mentioned in \ref{Sect Gegenbauer polynomials}.

 We conclude that $\Phi$ is a $\mk$-intertwining map. One can check that there exists an element in the simple submodule of $W_j$ with non-zero image. Hence $\Phi$ is injective. Since $\dim W_j = \dim  \cH_{j} (\mR^{\mu+3}) \otimes \cH_{\frac{\mu-\nu}{2}+j} (\mR^{\nu+3})$, we conclude that $\Phi$ is an isomorphism.
\end{proof}
\begin{remark}{\rm
If $\nu \in -2\mN-1$, then the $\mathfrak{osp}(p|2n)$-module $\cH_k(\mR^{\nu+3})$ is irreducible if $k> -1-\nu$ or $k<-\frac{1-\nu}{2}$.  It is always indecomposable. This is \cite[Theorem 5.2]{Coulembier}. This is in correspondence with Proposition~\ref{Prop: structure Wj}, since $\frac{\mu-\nu}{2}+j>-1-\nu$ is equivalent with $\frac{\mu+\nu}{2}+j\geq 0,$ which was the condition for irreducibility of $W_j$.
}
\end{remark}

\subsection{Harish-Chandra supermodules} \label{Section Harish-Chandra supermodules}
Let $G=(G_0, \mg, \sigma)$ be a Lie supergroup such that $G_0$ is almost connected and real reductive. Let $K$ be a maximal compact subgroup of $G_0$. 
\begin{definition}[{\cite[Definition 4.1]{Alldridge}}]
Let $V$ be a complex super-vector space. Then $V$ is a $(\mg,K)$-module if it is a locally finite $K$-representation which has also a compatible $\mg$-module structure, i.e.\  the derived action of $K$ agrees with the $\ck$-module structure and \[ k\cdot (X \cdot v) = (\sigma(k)X) \cdot (k\cdot v)  \text{ for all } k\in K,\; X \in \mg,\; v \in V. \]
A $(\mg,K)$-module is a Harish-Chandra supermodule if it is finitely generated over $U(\mg)$ and is $K$-multiplicity finite. 
\end{definition}
A $K$-module $W$ is $K$-multiplicity finite if every simple $K$-module occurs only a finite number of times in the decomposition of $W$.

Let $G=(O(p,q) \times Sp(2n,\mR),\TKK(J), \sigma)$ be the conformal Lie supergroup defined in Section~\ref{Section conformal group} and $K= O(p)\times O(q) \times U(n)$ be the maximal compact subgroup of $O(p,q) \times Sp(2n,\mR)$. The Lie algebra of $K$ is given by $\ck =\mathfrak{so}(p)\oplus\mathfrak{so}(q) \oplus \mathfrak{u}(n)$.
If $p+q$ is even and $\nu \not\in -2\mN$, we can define a $K$ representation on the $\mg$-module $W$ using the natural action of $O(p)\times Sp(2n,\mR) $ on $\cH_{l} (\mR^{p|2n})$ and the action of $O(q)$ on $\cH_l(\mR^q)$. 
\begin{proposition} \label{Prop: W is a Harish-Chandra module}
For $p+q$ even and $\nu\not \in -2\mN$ the module $W$ is a Harish-Chandra supermodule.
\end{proposition}
\begin{proof}
From the decomposition, Theorem~\ref{Theorem structure W}, 
\begin{align*}
W &\cong \sum_{l=0}^\infty  \cH_{\frac{\mu-\nu}{2}+l} (\mR^{p|2n}) \otimes \cH_l (\mR^q) \qquad \text{ if } p-2n\leq q, \text{ or } \\
W &\cong \sum_{l=0}^\infty \cH_{l} (\mR^{p|2n}) \otimes \cH_{\frac{\mu-\nu}{2}+l} (\mR^q) \qquad \text{ if } p-2n\geq q,
\end{align*}
it immediately follows that $W$ is locally $K$-finite. This decomposition also implies that $W$ is $O(q)$-multiplicity finite since $\cH_l (\mR^q)$ is an irreducible $O(q)$-module with $\cH_l (\mR^q)\not\cong \cH_k (\mR^q) $ if $l\not=k$, while $ \cH_{\frac{\mu-\nu}{2}+l} (\mR^{p|2n})$ is finite-dimensional with trivial $O(q)$ action. If there would be a simple $K$-module which has infinite multiplicity in the decomposition of $W$, this would imply that all the simple $O(q)$-modules contained in this $K$-module also have infinite multiplicity. Therefore we conclude that $W$ is $K$-multiplicity finite.

By definition the derived action of $K$ agrees with the action of $\ck \subset \mg$ and even with the action of $\mk$. To show $k\cdot (X \cdot v) = (\sigma(k)X) \cdot (k\cdot v)$ for all $k\in K$, $X \in \mg,$ $v \in V$, it thus suffices to show this for $X=L_e$, since $\mg= \mk +[\mk, [\mk,L_e]].$ We also only need to verify it for one element in each connected component of $K$. We choose $k$ as the element which act as the identity on the coordinates except on $s_0$ and $t_0$, the coordinates extending $\mR^{p-1}$ to $\mR^{p}$ and $\mR^{q-1}$ to $\mR^{q}$.  On $s_0$ and $t_0$ it  acts by $\pm 1$. This gives us four different $k$, corresponding to the four connected components of
$O(p)\times O(q)\times U(n)$. In the proof of Proposition \ref{Structure W} we calculated the action of $L_e$ on $W_j$. Combining this with the isomorphism given in Proposition~\ref{Prop: structure Wj II}, allows us to verify $k\cdot (L_e \cdot v) = (\sigma(k)L_e) \cdot (k\cdot v)$ for $v \in \cH_{\frac{\mu-\nu}{2}+l} (\mR^{p|2n}) \otimes \cH_l (\mR^q)$ by a direct calculation.
This finishes the proof
\end{proof}
The importance of the previous proposition lies in the fact that it allows us to integrate our representation to group level. 
In \cite{Alldridge}, it is proven that a Harish-Chandra supermodule has a (unique) smooth Fr\'{e}chet globalisation. This means that we have a Fr\'{e}chet space and a representation of the Lie supergroup on this Fr\'{e}chet space for which the space of $K$-finite vectors is the Harish-Chandra supermodule. 
\begin{corollary}\label{corollary integration to group level}
The $(\mg,K)$-module $W$ integrates to a unique smooth Fr\'{e}chet representation of moderate growth for the Lie supergroup $G$.
\end{corollary}
\begin{proof}
This follows from combining Proposition~\ref{Prop: W is a Harish-Chandra module} and Theorem A in \cite{Alldridge}.
\end{proof}

\section{Joseph ideal} \label{Section Joseph ideal}
We will now investigate some properties that the minimal representation we constructed  satisfies. 
We will show that the annihilator ideal is equal to a Joseph-like ideal if the superdimension satisfies $p+q-2n>2$. In this sense our representation is indeed a super version of a minimal representation since minimal representations for Lie groups are characterised by the property that their annihilator ideal is the Joseph ideal.
The classical Joseph ideals and the Joseph ideal for $\mathfrak{osp}(m|2n)$ with $m-2n>2$ have the property that any ideal which contains the Joseph ideal and which has still infinite codimension is equal to the Joseph ideal, as follows from the characterisation  by Garfinkle \cite{Garfinkle}. So the Joseph ideal is in this sense the biggest ideal with infinite codimension. If $p+q-2n\leq 2$, we still have that the annihilator ideal contains the Joseph ideal, but due to the lack of the characterisation by Garfinkle in this case, we no longer know whether the annihilator ideal is equal to the Joseph ideal. 

 We will also calculate the Gelfand--Kirillov dimension of our representation. We find that it is equal to $p-q-3$ which is also the Gelfand--Kirillov dimension of the minimal representation of $O(p,q)$. 
As mentioned in the introduction, we know that there are no unitary representations of $\mathfrak{osp}(p,q|2n)$. However, we  can still construct a non-degenerate, superhermitian, sesquilinear bilinear form.  We also show that our representation is skew-symmetric with respect to this form if $p+q-2n\geq 6$. 


\subsection{The super Fourier transform}\label{Sect: Fourier transform}
In the classical case a minimal representation for  a simple real Lie group $G$ is a unitary representation such that the annihilator ideal  of the derived representation in the universal enveloping algebra of Lie$(G)_{\mC}$ is  the Joseph ideal. We will show that the representation $\pi_C$ has as annihilator ideal the generalisation of the Joseph ideal for the $\mathfrak{osp}$-case. To do this, we will use the Fourier transformed representation. So we start this section by introducing the super Fourier transform. 

Consider $\cS(\mR^m)$ the Schwartz space of rapidly decreasing functions and the dual space $\cS'(\mR^m)$ of tempered distributions.
The (even) Fourier transform $\mF^{\pm}_{even}$ on $\cS(\mR^m)$ is given by
\[
\mF^{\pm}_{even} f(y) = \frac{1}{(2\pi)^{\frac{m}{2}}} \int_{\mR^m} \exp\left( \pm \imath \sum_{i,j=1}^{m} \x^i(x) \beta^s_{ij} \x^j(y) \right) f(x) \diff x,
\]
where $\beta^s$ is the symmetric part of the orthosymplectic metric.
The Fourier transform can easily be extended to the dual space $\cS'(\mR^m)$, by duality.
It satisfies $\mF^{\pm}_{even}\mF^{\mp}_{even}=\id$.

Let $\Lambda^{4n}:=\Lambda(\mR^{2n}\oplus \mR^{2n})$ be generated by $\theta_i, \eta_i$, $i=1,\ldots,2n$, with the relations $\theta_i\theta_j =-\theta_j\theta_i$, $\eta_i\eta_j=-\eta_j\eta_i$, $\theta_i\eta_j=-\eta_j\theta_i$. It contains two copies of $\Lambda^{2n}:= \Lambda(\mR^{2n})$, one generated by $\eta_i$ and one generated by $\theta_i$. 
 Then we set $K^{\pm}(\theta,\eta):= \exp (\mp \imath \sum_{i,j} \theta^i \beta_{ij}^a \eta^j)$ where $\beta^a$ is the antisymmetric part of the orthosymplectic metric.
Define the odd Fourier transform by
\[
\mF_{odd}^{\pm}\colon \Lambda^{2n} \to \Lambda^{2n}; \quad \mF_{odd}^{\pm}(f) = \int_{B,\theta} K^{\pm}(\theta,\eta) f(\theta),
\]
where $\int_{B,\theta}= \partial_{\theta_{2n}} \ldots \partial_{\theta_1}$ is the Berezin integral,  see e.g.\ \cite{Le}. The odd Fourier transform satisfies $\mF_{odd}^{\pm} \mF_{odd}^{\mp} =\id$.

Then the super Fourier transform $\mF^{\pm}\colon \cS'(\mR^m)\otimes \Lambda^{2n} \to\cS'(\mR^m)\otimes \Lambda^{2n}  $ is given by
\begin{align*}
\mF^{\pm} (f):= \sum_{I\in \mZ_2^n} \mF^{\pm}_{even} (f_I) \mF^{\pm}_{odd}(\theta^I) & & \text{ for } f=\sum_{I\in \mZ_2^n} f_I \theta^I.
\end{align*}
The super Fourier transform has the following properties.
\begin{proposition}\label{Prop: Fourier transform}
Let $(e_k)_k$ be a basis of $\mR^{m|2n}$ and $(e^k)_k$ its right dual basis with respect to the orthosymplectic metric and $z_k$, $z^k$ the corresponding coordinate functions. Then we have
\begin{align*}
\mF^{\pm}(\partial_{k} f) &= \mp \imath \x_k \mF^{\pm}(f),\quad 
\mF^{\pm}(\x_k f) = \mp \imath \partial_{k} \mF^{\pm}(f) \quad \text{and} \quad
\mF^{\pm}\mF^{\mp}=\id.
\end{align*}
\end{proposition}
\begin{proof}
See~\cite[Theorem 7 and Lemma 3]{Hendrik}. Remark that the metric used in op. cit. is not orthosymplectic. However, the same results and proofs still hold, mutatis mutandis.
\end{proof}

\subsection{Fourier-transformed and adjoint representation}
Using an isomorphism between $\mA(J^\ast)$ and $\mA^{p+q-2|2n}$, we can restrict the representation $\pi_\lambda$ defined in Section~\ref{Section: definition representation} to $\mathcal{S}(\mR^{p+q-2})\otimes \Lambda^{2n}$.
Then we can also immediately extend $\pi_\lambda$ to $\mathcal{S}'(\mR^{p+q-2})\otimes \Lambda^{2n}$ since $\pi_\lambda$ is given by differential operators. We will denote this extension also by $\pi_\lambda$. We will use the notations $\pi_{\lambda,\cS}$ and $\pi_{\lambda,\cS'}$ if we want to specify on which space the representations acts.

We define $\hat{\pi}_\lambda$ on $\cS(\mR^{p+q-2}) \otimes \Lambda^{2n}$ or $\cS'(\mR^{p+q-2}) \otimes \Lambda^{2n}$ using the super Fourier transform introduced in Section~\ref{Sect: Fourier transform}:
\[
\hat{\pi}_\lambda(X):= \mF^- \circ \pi_{-\lambda-2M} (X) \circ \mF^+.
\]
From Proposition~\ref{Prop: Fourier transform} it follows that $\hat{\pi}_\lambda(X)$ is given by
 \begin{enumerate} 
\item $\hat{\pi}_\lambda(0,0,e_k) = \partial_{k} \qquad\qquad\quad\qquad\qquad\qquad\quad$ for $e_k \in J^{\minus}$
\item $\hat{\pi}_\lambda(0,L_{kl},0) =  \x_k \partial_{l} -(-1)^{\abs{k}\abs{l}} \x_l \partial_{k} \qquad\qquad\; $ for $L_{kl} \in \mathfrak{osp}_(J)$
\item $\hat{\pi}_\lambda(0,L_e,0)= -\frac{\lambda}{2}+\mE$
\item $\hat{\pi}_\lambda(\bar{e}_k,0,0)= -\x_k( 2\mE -\lambda) +
R^2 \partial_{k} \;\qquad \qquad \text{ for } \bar{e}_k \in J^{\plus}$. 
\end{enumerate}
\begin{proposition}
The kernel of the Laplace operator $\Delta$ is a subrepresentation of $\hat{\pi}_\lambda$  if and only if $\lambda=2-M$, with $M=p+q-2-2n$ the superdimension of $J$. 
\end{proposition}
\begin{proof}
Using Lemma~\ref{Lemma relations sl(2)}, we find 
\begin{align*}
[\Delta, \hat{\pi}_\lambda (\bar{e}_k,0,0)] 
& = [\Delta, -\x_k (2\mE-\lambda) ] + [\Delta, R^2 \partial_{k} ] 
= 2(\lambda -2+M) \partial_{k} -4 \x_k \Delta.
\end{align*}
Hence, $\hat{\pi}_\lambda (\bar{e}_k,0,0)$ preserves the kernel of $\Delta$ if and only if $\lambda=2-M$. One verifies similarly that $\hat{\pi}_\lambda(0,0,e_k)$, 
$\hat{\pi}_\lambda(0,L_{kl},0)$ and $\hat{\pi}_\lambda(0,L_e,0)$ also preserve the kernel of $\Delta$.
\end{proof}
Denote by $\langle \phi, f \rangle_{\cS}$ the value of the action of $\phi \in \cS'(\mR^{p+q-2})\otimes \Lambda^{2n}$ on 
$f \in \cS(\mR^{p+q-2})\otimes \Lambda^{2n}$.
If $\phi \in \cS'(\mR^{p+q-2})\otimes \Lambda^{2n}$ is an element of $\ker \Delta$, then for all $f \in \cS(\mR^{p+q-2})\otimes \Lambda^{2n}$,
\begin{align*}
0=\langle \phi, \Delta \mF^{\plus} f \rangle_{\cS} = -\langle \phi, \mF^{\plus} R^2 f \rangle_{\cS} = -\langle \mF^{\plus} \phi, R^2 f \rangle_{\cS}.
\end{align*}
Hence the Fourier transform of $\ker \Delta$ consist of elements contained in $\cS(\mR^{p+q-2})'\otimes \Lambda^{2n}$ with support contained in the closure of $\abs{C}$.

 For $A \in \End(\mathcal{S}(\mR^{p+q-2})\otimes \Lambda^{2n})$ define the adjoint operator $A^\ast \in \End(\mathcal{S}'(\mR^{p+q-2})\otimes \Lambda^{2n})$  by
\[
\langle\overline{ A^\ast  \phi}, f \rangle_\cS =(-1)^{\abs{A}\abs{\phi}}\langle \phi, A f \rangle_\cS, \qquad 
\text{ for } \phi \text{ in } \mathcal{S}'(\mR^{p+q-2})\otimes \Lambda^{2n} \text{ and }f \text{ in } \mathcal{S}(\mR^{p+q-2})\otimes \Lambda^{2n}.\]
Here $\overline{ A^\ast  \phi}$ is the complex conjugate of $ A^\ast  \phi.$
\begin{proposition}\label{Prop: adjoint rep}
The adjoint operator $\pi_\lambda(X)^\ast$ is equal to the operator $-\pi_{-\lambda-2M}(X)$. Similar the adjoint operator $\hat{\pi}_\lambda(X)^\ast$ is equal to $-\hat{\pi}_{-\lambda-2M}(X)$.
\end{proposition}
\begin{proof}
On $\cS(\mR^{p+q-2}) \otimes \Lambda^{2n}$ we have
\begin{align*}
\langle  \partial_{k} \phi , f \rangle_{\cS}   := -(-1)^{\abs{k}\abs{\phi}}  \langle \phi ,  \partial_{k} f \rangle_{\cS} 
\qquad \text{and}\qquad
 \langle  \x_k \phi ,f \rangle_{\cS}  := (-1)^{\abs{k}\abs{\phi}}\langle \phi,  \x_kf \rangle_{\cS}.
\end{align*}
Using this we obtain
\begin{align*}
\langle -i \x_k \phi, f \rangle_{\cS}
 &= -(-1)^{\abs{k}\abs{\phi}} \langle \phi, i \x_k f \rangle_{\cS}, 
\\
\langle (-\mE-\frac{\lambda+2M}{2} )\phi, f\rangle_{\cS} 
&= -\langle \phi, (-\mE+\frac{\lambda}{2}) f \rangle_{\cS}, 
\\ 
\langle L_{ij} \phi, f\rangle_{\cS}
&= -(-1)^{(\abs{i}+\abs{j})\abs{\phi}}\langle \phi, L_{ij} f \rangle_{\cS}, 
\\
\langle -i\cB_{-\lambda-2M}(\x_k) \phi, f\rangle_{\cS}
&= -(-1)^{\abs{k}\abs{\phi}}\langle \phi, i\cB_{\lambda}(\x_k) f \rangle_{\cS}.
\end{align*}
From this the proposition follows.
\end{proof}

\subsection{Connection with a Joseph-like ideal}
We will now quickly introduce the Joseph ideal. A more detailed account is given in \cite{CSS}.
Set $\mg_\mC= \mathfrak{osp}_{\mC}(p+q|2n)$.
We choose a non-standard root system with the following simple positive roots
\begin{align*}
\epsilon_1-\epsilon_2, \dots, \epsilon_{\frac{p+q-3}{2}}-\epsilon_{\frac{p+q-1}{2}}, \epsilon_{\frac{p+q-1}{2}} -\delta_1, \delta_1-\delta_2, \dots, \delta_{n-1}-\delta_n, \delta_n, 
\end{align*}
for $p+q$ odd,
\begin{align*}
\epsilon_1-\epsilon_2, \dots, \epsilon_{\frac{p+q-2}{2}}-\epsilon_{\frac{p+q}{2}}, \epsilon_{\frac{p+q}{2}} -\delta_1, \delta_1-\delta_2, \dots, \delta_{n-1}-\delta_n,2\delta_n,
\end{align*}
for $p+q$ even. If $p+q-2n \not\in \{1,2\}$, then the tensor product $\mg_\mC \otimes \mg_\mC$ contains a decomposition factor isomorphic to the simple $\mg_\mC$-module of highest weight $2\epsilon_1+2\epsilon_2$, \cite[Theorem 3.1]{CSS}. We will denote this decomposition factor by $\mg_\mC \circledcirc \mg_\mC$.  

Define a one-parameter family $\{\cJ_\mu \,|\,\mu\in \mC\}$ of quadratic two-sided ideals in the tensor algebra $T(\fg_\mC)=\oplus_{j\ge 0}\otimes^j\fg_\mC$, where $\cJ_\mu$ is generated by
\begin{align*}
\label{definition jlambda}
\{X\otimes Y- X\circledcirc Y -\frac{1}{2} [X,Y]-\mu\langle X,Y \rangle \mid  X,Y\in\mg_\mC \}
\subset\, \fg_\mC\otimes\fg_\mC\,\,\oplus\,\,\fg_\mC\,\,\oplus\,\,\mC\,\subset\, T(\fg_\mC).
\end{align*}
Here $X \circledcirc Y$ is the projection of $X\otimes Y$ on $\mg_\mC \circledcirc \mg_\mC$, and $\langle
 X,Y \rangle$ is the Killing form. 

By construction there is a unique ideal $J_\mu$ in the universal enveloping algebra $U(\fg_\mC)$, which satisfies $T(\fg_\mC)/\cJ_\mu \cong U(\fg_\mC)/J_\mu$. Define \[ \mu^{c}:=-\frac{p+q-4-2n}{4(p+q-1-2n)}.\]
Then $J_\mu$ has finite codimension for $\mu\not= \mu^c$ and infinite codimension for $\mu= \mu^c$, \cite[Theorem 5.3]{CSS}.
We call $J_{\mu^c}$ the Joseph ideal of $\mg_{\mC}$.

The annihilator ideal of a representation $(\pi, V)$ of $\mg_{\mC}$ is by definition the ideal in $U(\mg_{\mC})$ given by
\[
\Ann (\pi) := \{X \in  U(\mg_{\mC}) \mid \pi(X) v = 0 \text{ for all } v \in V \}.
\]

\begin{theorem}\label{Theorem Joseph ideal}
If $p+q-2n>2$, then  \[
\Ann (\pi_C) =J_{\mu^c}.
\]
We also have $ J_{\mu^c} \subseteq \Ann (\pi_C)$ if $p+q-2n \not\in \{1,2\}$.
\end{theorem}
\begin{proof}
 From \cite[Corollary 5.8]{CSS}, it follows that
 \[
 \hat{\pi}_{\lambda,\cS}(X)\hat{\pi}_{\lambda,\cS}(Y) = \hat{\pi}_{\lambda,\cS}(X\circledcirc Y) + \frac{1}{2} \hat{\pi}_{\lambda,\cS}([X,Y]) + \mu^c \langle X, Y \rangle 
 \]
on $\cS \cap \ker \Delta$ for $\lambda = 2-M$. Therefore 
$
J_{\mu^c} \subseteq \Ann( \hat{\pi}_{\lambda,\cS} {}_{\mid \ker \Delta})$  for $\lambda=2-M.$

 Proposition~\ref{Prop: adjoint rep}  implies
\begin{align*}
\Ann( {\hat{\pi}_{\lambda,\cS}} {}_{\mid \ker \Delta})  &= \Ann( {\hat{\pi}_{\lambda,\cS'}^\ast} {}_{\mid \ker \Delta}) 
 = \Ann( \hat{\pi}_{-\lambda-2M, \cS'} {}_{\mid \ker \Delta}).
\end{align*}
We have \[ \pi_{\lambda,S'}(X) v = (\mF^+ \circ {\hat{\pi}}_{-\lambda-2M, \cS'}(X) \circ \mF^-) v =0   \] for all  $v$ in $ \cS' \cap \ker \Delta$ if and only if  ${\hat{\pi}}_{-\lambda-2M, \cS'}(X) v= 0 $ for all  $v$ in $\cS'$ with support contained in the closure of $\abs{C}$. Therefore
\begin{align*}
\Ann( \hat{\pi}_{-\lambda-2M, \cS'} {}_{\mid \ker \Delta})
&= \Ann( \pi_{\lambda,\cS'} {}_{\mid \text{ supp contained in } \abs{C}}).
\end{align*}
Furthermore
\begin{align*}
\Ann( \pi_{\lambda,\cS'} {}_{\mid \text{ supp contained in } \abs{C}})& \subseteq \Ann( \pi_{\lambda,\cS} {}_{\mid \text{ supp contained in } \abs{C}})
 = \Ann( \pi_C).
\end{align*}
We conclude 
\[
J_{\mu^c} \subseteq \Ann( \pi_C).
\]
From \cite[Theorem 5.4]{CSS}, it follows that every ideal with infinite codimension that contains the Joseph-like ideal $J_{\mu^c}$ is equal to $J_{\mu^c}$ if $p+q-2n>2.$
Since $J_{\mu^c}\subseteq  \Ann( \pi_C)$ and $\Ann (\pi_C)$ has infinite codimension,  the theorem follows.
\end{proof}

\section{The Gelfand--Kirillov dimension}\label{Section Gelfand-Kirillov dimension}
The Gelfand--Kirillov dimension is a measure of the size of a representation. 
Suppose that $R$ is a finitely generated algebra and $M$ is a finitely generated $R$-module. 
Then the Gelfand--Kirillov dimension  (GK-dimension) of $M$ is defined by
\[
GK(M)=  \lim\sup_{k\to\infty}  \left( \log_k \dim (V^k F)\right),
\]
where $V$ is a finite-dimensional subspace of $R$ containing $1$ and generators of $R$, and $F$ is a finite-dimensional subspace of $M$ which generates $M$ as an $R$-module. 
This definition is independent of our choice of $V$ and $F$, \cite[Section 7.3]{Musson}.
\begin{proposition} \label{Prop Gelfand-Kirillov}
The Gelfand--Kirillov dimension of the $U(\mg)$-module $W$ defined in \eqref{Definition W} is given by
\[
GK(W) = p+q-3.
\]
\end{proposition}
\begin{proof}
We choose $W_0$ for $F$ and $\mg\oplus 1\subset U(\mg)$ for $V$. Then $V^k = U_k(\mg)$, with $U_k(\mg)$ the canonical filtration on the universal enveloping algebra. From the proof of Proposition~\ref{Prop: structure Wj}, we obtain
\[
U_k(\mg) W_0 = \bigoplus_{j=0}^k W_j. 
\]

Using Proposition~\ref{Prop: structure Wj II} and the dimension of the space of spherical harmonics given in Proposition~\ref{Prop: dimension of spherical harmonics}, we compute the dimension of $U_k(\mg) W_0 $ for the case $p-2n\geq q$
\begin{align*}
 \dim U_k(\mg) W_0 &=\sum_{j=0}^k  \left(  \sum_{i=0}^{\min (j,2n)} \left( \begin{smallmatrix}\
 2n \\ i \end{smallmatrix}\right) \begin{pmatrix}\begin{smallmatrix}
 j-i+p-1 \\ p-1 
\end{smallmatrix} \end{pmatrix} -  \sum_{i=0}^{\min (j-2,2n)} \begin{pmatrix}\begin{smallmatrix}
 2n \\ i \end{smallmatrix}\end{pmatrix} \begin{pmatrix}
 j-i+p-3 \\ p-1 
 \end{pmatrix} \right)  \cdot  
 \\
& \quad \left( \begin{pmatrix}\begin{smallmatrix}
 \frac{\mu-\nu}{2}+j+q-1 \\q- 1
\end{smallmatrix}\end{pmatrix}  - \begin{pmatrix}\begin{smallmatrix}
\frac{\mu-\nu}{2}+j+q-3 \\ q-1
\end{smallmatrix}\end{pmatrix} \right) 
\\
&= \sum_{i=0}^{2n} \begin{pmatrix}\begin{smallmatrix}
2n \\ i
\end{smallmatrix}\end{pmatrix} \left( \begin{pmatrix}\begin{smallmatrix}
k-i+p-1 \\p-1
\end{smallmatrix}\end{pmatrix} -\begin{pmatrix}\begin{smallmatrix}
k-i+p-3 \\p-1
\end{smallmatrix}\end{pmatrix} \right)\begin{pmatrix}\begin{smallmatrix}
\frac{\mu-\nu}{2} + k+q-1 \\q-1
\end{smallmatrix}\end{pmatrix} 
\\
&\quad +\sum_{i=0}^{2n} \begin{pmatrix}\begin{smallmatrix}
2n \\ i
\end{smallmatrix}\end{pmatrix} 
\left( \begin{pmatrix}\begin{smallmatrix}
k-i+p-2 \\p-1
\end{smallmatrix}\end{pmatrix} -\begin{pmatrix}\begin{smallmatrix}
k-i+p-4 \\p-1
\end{smallmatrix}\end{pmatrix} \right)\begin{pmatrix}\begin{smallmatrix}
\frac{\mu-\nu}{2} + k+q-2 \\q-1
\end{smallmatrix}\end{pmatrix} 
\\ 
& \quad -\begin{pmatrix}\begin{smallmatrix}
\frac{\mu-\nu}{2} +q-3 \\ q-1 
\end{smallmatrix}\end{pmatrix}-(p+2n) \begin{pmatrix}\begin{smallmatrix}
\frac{\mu-\nu}{2}+q-2 \\ q-1,
\end{smallmatrix}\end{pmatrix}
\end{align*}
where we assumed $k>>2n$. By \cite[Lemma 7.3.1]{Musson}, it is sufficient to know the highest exponent of $k$ in the expression for $\dim U_k(\mg) W_0$ to calculate  $\lim\sup_{k\to\infty} \log_k \dim U_k(\mg) W_0$. The highest exponent of $k$ in $
\left( \begin{pmatrix}\begin{smallmatrix}
k-i+p-1 \\p-1
\end{smallmatrix}\end{pmatrix} -\begin{pmatrix}\begin{smallmatrix}
k-i+p-3 \\p-1
\end{smallmatrix}\end{pmatrix} \right)
$  is given by $p-2$, while in $\begin{pmatrix}\begin{smallmatrix}
\frac{\mu-\nu}{2} + k+q-1 \\q-1
\end{smallmatrix}\end{pmatrix} $ it is given by $q-1$. Therefore, we conclude 
$GK(W) = p+q-3$.
\end{proof}
\section{Non-degenerate sesquilinear form}\label{Section integration}\label{Non-degenerate sesquilinear form}
In this section we will define an `integral' on the minimal orbit. More specifically, we will define a functional on a subspace of $\Gamma(\cO_{\mA^{p+q-2|2n}_{(0)}})$. This functional also leads to a functional on a subspace of $\Gamma(\cO_{C})$, which then can be used to define a sesquilinear form on $W$, where $W$ is the submodule defined in Equation \ref{Definition W}. Then we show that the representation $\pi_C$ on $W$ is skew-symmetric with respect to this sesquilinear form if $p+q-2n-6\geq 0$. 

We will use the same conventions as in Section~\ref{Section radial superfunctions} for $s^2$, $t^2,$ and $\theta^2$.
Further, we also set
\begin{align*}
1+\eta &:= \sqrt{1- \frac{\theta^2}{2s^2}} = \sum_{j=0}^n  \frac{1}{j!2^j}\left(\frac{-1}{2}\right)_j \frac{\theta^{2j}}{s^{2j}}, \\
1+\xi &:= \sqrt{1+ \frac{\theta^2}{2t^2}} = \sum_{j=0}^n  \frac{(-1)^j}{j!2^j}\left(\frac{-1}{2}\right)_j \frac{\theta^{2j}}{t^{2j}},
\end{align*}
where $\left(\frac{-1}{2}\right)_j$ is the Pochhammer symbol $\frac{-1}{2}\frac{1}{2}\frac{3}{2}\cdots \frac{-1+2(j-1)}{2}. $
Note that $\eta$ and $\xi$ are nilpotent since $\eta^{n+1}=0=\xi^{n+1}$.

\subsection{Bipolar coordinates} We use bipolar coordinates to define a morphism between certain algebras of superfunctions. More precisely, for $(x,y)\in\mR^{p-1}\times\mR^{q-1}=\mR^{p+q-2}$ consider spherical coordinates by setting 
$x_i= s\omega_i^p$, and  $y_j= t\omega_j^q$ with $\omega^p \in \mathbb{S}^{p-2}$ and $\omega^q  \in \mathbb{S}^{q-2}$. We then define
\[
\partial_u := \frac{1}{2}\partial_{t^2} -\frac{1}{2}\partial_{s^2} =\frac{1}{4t}\partial_t -\frac{1}{4s}\partial_s.
\]

\begin{lemma} \label{Lemma multiplication and partial derivative}
The morphism 
\begin{align}
\label{Def phisharp}
\begin{aligned}
\phi^\sharp \colon \cC^\infty(\mR^{p+q-2}_{(0)})\otimes \Lambda^{2n} \to \cC^\infty\left(\mR^+ \times \mathbb{S}^{p-2} \times \mR^+ \times \mathbb{S}^{q-2} \right)\otimes \Lambda^{2n} 
\\ 
f \mapsto  \phi^\sharp(f) = \exp \theta^2 \partial_{u} (f)= \sum_{j=0}^n  \frac{\theta^{2j}}{j!} \left(\frac{1}{4t} \partial_{t}-\frac{1}{4s} \partial_{s}\right)^j ( f), 
\end{aligned}
\end{align}
is a well-defined (algebra) morphism.
\end{lemma}
\begin{proof}
One can easily check that $\phi^\sharp$ is a linear map which satisfies $\phi^\sharp(fg) = \phi^\sharp(f)\phi^\sharp(g)$. Note that there are points of $\mR^{p+q-2}_{(0)}$, for which $s=0$ or $t=0$ and in those points $1/s^k$ and $1/t^k$ are not well-defined. Therefore, we restricted the domain of the image to $s>0$ and $t>0$, where $1/s^k$ and $1/t^k$ are smooth. The product of a smooth function with a smooth function gives again a smooth function. For the partial derivatives we remark that  $\partial_s =\sum_i \frac{x_i}{s}\partial_{x_i}$. Multiplication with  $x_i$ and derivations with respect to  $x_i$ are smooth operators, so $\partial_s$ is a smooth operator. Similar we also have that $\partial_t$ is smooth, which proves the lemma.
\end{proof}
The superalgebra morphism~$\phi^\sharp$ satisfies the following properties:
\begin{lemma} \label{Lemma: properties of phi sharp}
We have
\begin{align*}
\phi^\sharp x_i &= (1+\eta)x_i \phi^\sharp, &
\phi^\sharp \partial_{x^i} &=  \frac{1}{1+\eta} (\partial_{x^i}-x_i\frac{\theta^2}{s^2} \partial_{s^2}) \phi^\sharp , \\
\phi^\sharp y_i &= (1+\xi) y_i \phi^\sharp, &
\phi^\sharp \partial_{y^i} &=  \frac{1}{1+ \xi } (\partial_{y^i}  - 2y_i\frac{ \theta^2  }{ t^2} \partial_{t^2}) \phi^\sharp, \\
  \phi^\sharp \theta_k &= \theta_k  \phi^\sharp, & \phi^\sharp \partial_{\theta^k} &= (\partial_{\theta^k} -2\theta_k\partial_{u}) \phi^\sharp, \\
\phi^\sharp s &= (1+\eta)s \phi^\sharp, & \phi^\sharp \partial_s &= (1+\eta)\partial_s \phi^\sharp , \\
\phi^\sharp t &= (1+\xi)t \phi^\sharp, & \phi^\sharp \partial_t &= (1+\xi)\partial_t \phi^\sharp . 
\end{align*}

\end{lemma}
\begin{proof}
We have $\partial_{t^2}=\frac{1}{2t} \partial_{t}= \frac{1}{2t^2}\sum_i y_i\partial_{y_i}$. Using this we obtain, for $l \in \mN$.
\[
\partial_{t^2} \left( \frac{y_k}{t^{2l}} \right)= (\tfrac{1}{2}-l) \frac{y_k}{t^{2l+2}} \text{ and } \partial^l_{t^2} (y_k) = (-1)^l\left(\frac{-1}{2}\right)_l  \frac{y_k}{t^{2l}}.
\]
Therefore
\[
\phi^\sharp(y_k)= y_k \sum_{j=0}^n \frac{(-1)^j}{j!2^j}\left(\frac{-1}{2}\right)_j  \frac{\theta^{2j}}{ t^{2j} } = y_k (\xi+1).
\]
Since $\phi^\sharp$ is an algebra morphism, we have $
\phi^\sharp(y_i f) = \phi^\sharp(y_i) \phi^\sharp(f) = (1+\xi)y_i\phi^\sharp(f).$
In the same way, we get $\phi^\sharp(x_i) =  (1+\eta)x_i$, $\phi^\sharp(s) =(1+\eta)s$ and $\phi^\sharp(t) = (1+\xi)t $. Since $\partial_t = 2t \partial_{t^2}$, we obtain
\[
\phi^\sharp( \partial_t) = 2t (\xi+1) \partial_{t^2} \phi^\sharp= (\xi+1) \partial_t \phi^\sharp.
\]
Rewrite $\partial_{y^i}$ as $\frac{-y_i}{t}\partial_t -\sum_{k} \frac{y^k}{t^2}L^q_{ki}$, with $L^q_{ki}= y_k \partial_{y^i}-y_i \partial_{y^k}$ and use the fact that $[\partial_{t^2}, L^q_{ki}]=0$. Then we compute
\begin{align*}
\phi^\sharp \partial_{y^i} 
& = -(\xi+1) \frac{y_i}{t}\partial_t \phi^\sharp - \sum_k\frac{y^k}{(\xi+1)t^2} L^q_{ki} \phi^\sharp 
= \frac{1}{1+ \xi } \partial_{y^i} \phi^\sharp - \frac{ (\xi^2+2\xi)y_i }{ (\xi+1)t} \partial_{t} \phi^\sharp.
\end{align*}
Because $\xi^2 + 2\xi = \frac{\theta^2}{2t^2}$, this proves it for $\partial_{y^i}$. 
Using $[\partial_{\theta^i}, \theta^{2k}]= 2k \theta_i \theta^{2k-2}$ we obtain 
\[
\phi^\sharp \partial_{\theta^i} =(  \partial_{\theta^i} -2\theta_i\partial_{u} )\phi^\sharp,
\]
while the cases $\partial_s$ and $\partial_{x^i}$ are similar to $\partial_t$ and $\partial_{y^i}$.
\end{proof}
\subsection{The functional}

In \cite[Equation (2.2.3)]{KM2}  the following distribution on  $\mR^{p+q-2}_{(0)}$ is defined
\begin{align} \label{Integration in bipolar coordinates}
\begin{aligned}
\langle \delta(r^2), f \rangle
& =\int_{\abs{C}} f
:= \frac{1}{2}\int_{0}^\infty \int_{\mathbb{S}^{p-2}}\int_{\mathbb{S}^{q-2}} f_{\mid s=t=\rho} \rho^{p+q-5} \diff\rho \diff\omega^p \diff\omega^q,
\end{aligned}
\end{align}
where $f$ is a smooth function with compact support.
The Berezin integral on $\Lambda^{2n}$ is defined as
\[
\int_B := \partial_{\theta_{2n}}\partial_{\theta_{2n-1}} \cdots \partial_{\theta_1}. 
\]

In the spirit of \cite{Integration}, where integration over the supersphere was studied, we then define the following functional on $\cC^\infty_c (\mR^{p+q-2}_{(0)}) \otimes \Lambda^{2n}$, where $\cC^\infty_c(\mR^{p+q-2}_{(0)})$ stands for smooth functions on $\mR^{p+q-2}_{(0)}$ with compact support, 
\begin{align}\label{definition integral}
\int_{\kerorbit} f &:=\int_{\abs{C}} \int_B \left(1+\eta\right)^{p-3}\left(1+\xi \right)^{q-3} \phi^\sharp (f)  
\\
&= \frac{1}{2}\int_{0}^\infty \int_{\mathbb{S}^{p-2}}\int_{\mathbb{S}^{q-2}}  \int_B  \rho^{p+q-5} \left(1+\eta\right)^{p-3}\left(1+\xi \right)^{q-3} \phi^\sharp (f)_{\mid s=t=\rho}  \diff\rho \diff\omega^p \diff\omega^q,  \nonumber
\end{align}
with $\phi^\sharp(f)$ the morphism defined in equation~\eqref{Def phisharp}.

 Since the integral $\int_{0}^\infty \int_{\mathbb{S}^{p-2}}\int_{\mathbb{S}^{q-2}} \diff\rho \diff\omega^p \diff\omega^q$ is convergent for smooth functions with compact support, Lemma~\ref{Lemma multiplication and partial derivative} implies that the functional defined in  \eqref{definition integral} is well-defined. 
 
 
 We can extend the domain of our functional $\int_C$ from smooth functions with compact support to Bessel functions with polynomials of high enough degree. 
 \begin{lemma} \label{integral well-defined}
 Let $P_k$ be a homogeneous polynomial of degree $k$ in $\cP(\mR^{p+q-2|2n})$ and $\widetilde{K}_{\alpha} (\abs{X}),$ $\widetilde{K}_{\beta} (\abs{X})$  Bessel functions with $\alpha$, $\beta$ in $\mR$. If $p+q-2n-4+k > 2 \max(\alpha,0) + 2 \max(\beta,0)$, then we can extend the domain of our functional $\int_C$ such that
 \[
 \int_C P_k \widetilde{K}_{\alpha} (\abs{X})\widetilde{K}_{\beta} (\abs{X})
 \]
 is also defined.
 \end{lemma}
 \begin{proof}

The morphism $\phi^\sharp$ leaves the degree of a polynomial unchanged. Hence, we can expand 
\[
(\phi^\sharp (P_k))_{s=t=\rho} = \sum_{j=0}^k \rho^{k-j} a_j(\theta) b_j(\omega^p,\omega^q),
\]
where $a_j(\theta)$ is a polynomial in  $\cP(\mR^{0|2n})$ of degree $j$ and $b_j(\omega^p,\omega^q)$ is a function depending on the spherical coordinates $\omega^p$ and $\omega^q$.
Since $\partial_u(\abs{X}) =0$, we have
\[
\phi^\sharp(\widetilde{K}_\alpha(\abs{X}) )= \widetilde{K}_\alpha(\abs{X}).
\]
We also have
\begin{align}\label{Eq expansion1}
(1+\eta)^{p-3}&=\left(1-\frac{\theta^2}{2s^2}\right)^{\frac{p-3}{2}} = \sum_{j=0}^{n} \frac{1}{j!}\left(\frac{3-p}{2}\right)_j \frac{\theta^{2j}}{2^js^{2j}}, 
\\
(1+\xi)^{q-3}&=\left(1+\frac{\theta^2}{2t^2}\right)^{\frac{q-3}{2}} = \sum_{j=0}^{n} \frac{(-1)^j}{j!}\left(\frac{3-q}{2}\right)_j \frac{\theta^{2j}}{2^jt^{2j}}. \label{Eq expansion2}
\end{align}
and 
\[
\widetilde{K}_{\alpha}(\abs{X})=\sum_{j=0}^n\frac{(-1)^j\theta^{2j}}{j!8^j}\widetilde{K}_{\alpha+j}(\rho),
\]
where we used  the differential recurrence relation \eqref{Bessel function diff rec rel} of the Bessel function.

Combining all this, we see that \[ \int_{\abs{C}} \int_B  \left(1+\eta \right)^{p-3} \left(1+\xi \right)^{q-3} \phi^\sharp (P_k \widetilde{K}_{\alpha} (\abs{X})\widetilde{K}_{\beta} (\abs{X}))  \]
converges if 
\[
\int_{0}^\infty\int_B \rho^{p+q-5+k-j_1-2j_2} a_{j_1}(\theta) \theta^{2j_2+2j_3+2j_4} \widetilde{K}_{\alpha+j_3}(\rho) \widetilde{K}_{\beta+j_4}(\rho) \diff \rho
\]
converges for all $0\leq j_1 \leq k$, $0\leq j_2,j_3,j_4 \leq n$.
The Berezin integral is zero unless $j_1+2(j_2+j_3+j_4)=2n$.
The integral
\[
\int_0^\infty \widetilde{K}_\alpha(\rho) \widetilde{K}_\beta(\rho) \rho^{\sigma-1} \diff \rho
\]
converges if $\sigma > 2 \max(\alpha,0) + 2 \max(\beta,0)$. This follows from the asymptotic behaviour of the Bessel functions, see Section~\ref{Section Bessel functions}.
Therefore we get the following condition
\[
p+q-4+k -(j_1+2j_2) > 2 \max(\alpha+j_3,0) + 2\max(\alpha+j_4,0),
\]
with $j_1+2(j_2+j_3+j_4)=2n$.
This is equivalent with 
\[
p+q-2n-4+k > 2 \max(\alpha,-j_3) + 2\max(\alpha,-j_4),
\]
which proves the lemma.
 \end{proof}
 For future reference, we also need the following lemma.
 \begin{lemma} \label{Lemma lim rho}
 Let $P_k$ be a homogeneous polynomial in $\cP(\mR^{p+q-2|2n})$ of degree $k$ and $\widetilde{K}_{\alpha} (\abs{X}),$ $\widetilde{K}_{\beta} (\abs{X})$  Bessel functions with $\alpha$, $\beta$ in $\mR$. Then for $z_i=x_i$ or $z_i=y_i$
 \[
\lim_{\rho\to 0} \int_B \rho^{p+q-5} (1+\xi)^{q-3} \frac{z_i}{\rho} \phi^\sharp (P_k \widetilde{K}_{\alpha} (\abs{X})\widetilde{K}_{\beta} (\abs{X})) =0
 \]
if $p+q-2n-5+k > 2 \max(\alpha,0) + 2 \max(\beta,0)$. The limit of $\rho$ to infinity is always zero.
 \end{lemma}
 \begin{proof}
 Similar as in the proof of Lemma~\ref{integral well-defined}, we have to calculate
 \[
 \lim_{\rho\to 0} \rho^{p+q-5+k-2n+2j_1+2j_2}  \widetilde{K}_{\alpha+j_1}(\rho)\widetilde{K}_{\beta+j_2} (\rho),
 \]
 for $0\leq j_1,j_2 \leq n$.
 Using the asymptotic behaviour at zero of the K-Bessel function, we obtain that this limit is zero if 
 \[
 p+q-5+k-2n +2j_1 +2 j_2 > 2\max(\alpha+j_1,0) + 2\max(\beta+j_2,0).
 \]
This is equivalent with $M-3+k > 2 \max(\alpha,0) + 2 \max(\beta,0)$. The Bessel function goes exponentially to zero at infinity. Hence the limit for $\rho$ to infinity is also zero.
 \end{proof}
 As an example and to show that our functional is non-zero, we will now calculate the functional for the generating function of $W$.
\begin{lemma} \label{Integral of Knu}
For $\mu$ and $\nu$ as defined in \eqref{definition mu and nu}, we have
\begin{align*}
\int_C &K_{\frac{\nu}{2}} (|X|) K_{\frac{\nu}{2}} (|X|)= \tfrac{2^{\mu+\nu} }{n!}  \left(\tfrac{3-p}{2}\right)_n \tfrac{\pi^{\frac{p+q-2}{2}}}{\Gamma(\frac{p-1}{2})\Gamma(\frac{q-1}{2})}   \tfrac{\Gamma(\frac{\mu-\nu}{2}+1) \Gamma(\frac{\mu+\nu}{2}+1) \Gamma(\frac{\mu}{2}+1) \Gamma(\frac{\mu}{2}+1) }{\Gamma(\mu+2)} .
 \end{align*}
\end{lemma} 
Note that $\left(\frac{3-p}{2}\right)_n=0$ implies $\nu \in -2\mN$. Thus for $\nu \not\in -2\mN$, $\int_C K_{\frac{\nu}{2}} (|X|) K_{\frac{\nu}{2}} (|X|)$ is non-zero. 
 \begin{proof}
Using $\phi^\sharp (\widetilde{K}_{\alpha}(\abs{X}))_{\mid s=t=\rho} = \widetilde{K}_{\alpha}(\abs{X})_{\mid s=t=\rho} 
=\sum_{j=0}^n\frac{(-1)^j\theta^{2j}}{j!8^j}\widetilde{K}_{\alpha+j}(\rho)$, the expansion of $(1+\eta)^{p-3}$ and $(1+\xi)^{q-3}$ given in \eqref{Eq expansion1} and \eqref{Eq expansion2}, and the following property \cite[10.3 (49)]{EMOT},
 \begin{align*}
  \int_{0}^\infty \rho^{\sigma-1} \widetilde{K}_{\alpha}(\rho) \widetilde{K}_{\beta}(\rho)\diff \rho
 = 2^{\sigma-3}\frac{\Gamma(\frac{\sigma}{2})}{\Gamma(\sigma-\alpha-\beta)} \Gamma(\frac{\sigma-2\alpha}{2})\Gamma(\frac{\sigma-2\beta}{2})\Gamma(\frac{\sigma-2\alpha-2\beta}{2}),
\end{align*}
we obtain after a long but straightforward calculation
\begin{align*}
\int_C &K_{\frac{\nu}{2}} (|X|) K_{\frac{\nu}{2}} (|X|)=  \frac{\pi^{\frac{p+q-2}{2}}}{\Gamma(\frac{p-1}{2})\Gamma(\frac{q-1}{2})} 2^{\mu+\nu-n}   \frac{\Gamma(\frac{\mu-\nu}{2}+1) \Gamma(\frac{\mu+\nu}{2}+1) \Gamma(\frac{\mu}{2}+1) \Gamma(\frac{\mu}{2}+1) }{\Gamma(\mu+2)}  \Sigma(p,q,n),
\end{align*}
where 
\begin{align*}
\Sigma(p,q,n) 
=\sum\limits_{\substack{i,j,k,l=0,\\ i+j+k+l=n}}^n\tfrac{(-1)^{i+j+k}}{i!j!k!l! } \left(\tfrac{3-q}{2}\right)_k \left(\tfrac{3-p}{2}\right)_l  \left(\tfrac{\mu}{2}+1\right)_i \left(\tfrac{\mu}{2}+1\right)_j\frac{\left(\frac{\mu+\nu}{2}+1\right)_{i+j}}{\left(\mu+2\right)_{i+j}}.
 \end{align*}
 We have 
 \begin{align*}
 \sum_{i=0}^a \begin{pmatrix}
 a \\ i 
 \end{pmatrix} (x)_i (y)_{a-i} = (x+y)_a.
 \end{align*}
Using this we can compute $\Sigma(p,q,n)$, with $a=i+j$,
\begin{align*}
\Sigma(p,q,n)
&= \sum_{l=0}^n \sum_{a=0}^{n-l} \sum_{i=0}^{a} \frac{(-1)^{n-l}}{(n-l-a)!l! } \left( \left(\tfrac{3-q}{2}\right)_{n-l-a}\left(\tfrac{3-p}{2}\right)_l  \frac{\left(\frac{\mu+\nu}{2}+1\right)_{a}}{\left(\mu+2\right)_{a}} \frac{\left(\frac{\mu}{2}+1\right)_i \left(\frac{\mu}{2}+1\right)_{a-i} }{i!(a-i)!} \right),
\end{align*}
which can be simplified to $\Sigma(p,q,n)=\frac{2^n }{n!}  \left(\frac{3-p}{2}\right)_n$.
This finishes the proof.
 \end{proof}
The main proposition of this section is the following.
\begin{proposition} Let  $f=P_k \widetilde{K}_{\alpha} (\abs{X})\widetilde{K}_{\beta} (\abs{X})$, with $P_k$ a homogeneous polynomial of degree $k$ with $p+q-2n-5+k > 2 \max(\alpha,0) + 2 \max(\beta,0)$ or let $f$ be in $\cC^\infty_c (\mR^{p+q-2}_{(0)}) \otimes \Lambda^{2n}$ . 

The integral $\int_{\kerorbit}$ \label{Proposition properties integral} has the following properties.
\begin{enumerate}
\item Only depends on the restriction of $f$ to the minimal orbit $C$:  \[\int_{\kerorbit} R^2 f =0.\]

\item It is $\mathfrak{osp}(p-1,q-1|2n)$ invariant: \[ \int_{\kerorbit} X(f) = 0 \qquad \text{ for all }  X \text{ in } \mathfrak{osp}(p-1,q-1|2n).\]
\item It satisfies 
\[
\int_{\kerorbit} (\mE+M-2) (f) = 0,
\]
where $M=m-2n$ is the superdimension of $\mR^{p+q-2|2n}$.
\item The integral is symmetric with respect to the Bessel operators
\[
\int_{\kerorbit} (\cB_\lambda(e_k)f) g=(-1)^{\abs{f}\abs{k}} \int_{\kerorbit} f (\cB_\lambda(e_k)g),
\]
for the critical value $\lambda= -M +2$.
\end{enumerate} 
\end{proposition}

The integration of a derivative is as follows. 
\begin{lemma} \label{Lemma: integration of a derivative}
For $f$  as in Proposition~\ref{Proposition properties integral} it holds
\[
\int_{\kerorbit} \partial_{\x^i} f = \int_{\kerorbit} (s\partial_s + p-1) \frac{\x_i}{2s^2} f- (t\partial_t + q-1) \frac{\x_i}{2t^2} f.
\]
\end{lemma}

\begin{proof}
Assume first that $z_i$ is equal to $x_i$.

For $x_i$ we have $\partial_{x^i} = \frac{x_i}{s}\partial_s +\sum_k  \frac{x^k }{s^2} L^p_{ki}$, with $L^p_{ki} = x_k \partial_{x^i}-x_i \partial_{x^k}$.
Then, using  $\sum_k [x^k, L_{ki}^p]= (p-2)x_i$, the fact that  $\int_{\mathbb{S}^{p-2}} L^p_{ki} f = 0$ and Lemma~\ref{Lemma: properties of phi sharp}, we  obtain
\begin{align*}
\int_\kerorbit  \partial_{x^i} f &= \int_{\abs{C}}\int_B \frac{(1+\eta)^{p-3}}{(1+\xi)^{3-q}} \frac{1}{1+\eta}\left(\partial_{x^i} -x_i \frac{\theta^2}{2s^2}\partial_{s^2}\right) \phi^\sharp f
\\
&=\int_{\abs{C}}\int_B \frac{(1+\eta)^{p-4}}{(1+\xi)^{3-q}} \left( \frac{x_i}{s}\partial_s + \sum_k \frac{1}{s^2}[x^k,L^p_{ki}] - x_i \frac{\theta^2}{2s^2}\partial_{s^2} \right) \phi^\sharp f
\\
&=\int_{\abs{C}}\int_B \frac{(1+\eta)^{p-4}}{(1+\xi)^{3-q}} \left( \frac{x_i}{s}(1+\eta^2)\partial_s +  \frac{p-2}{s^2}x_i \right) \phi^\sharp f.
\end{align*}
On the other hand, using Lemma~\ref{Lemma: properties of phi sharp}, we obtain
\begin{align*}
&\int_C \left((s\partial_s+p-1)\frac{x_i}{2s^2}-(t\partial_t +q-1)\frac{x_i}{2t^2}\right) f
\\
&= \int_C \left(\frac{x_i}{2s}\partial_s+(p-2)\frac{x_i}{2s^2}-\frac{x_i}{2t}\partial_t -(q-3)\frac{x_i}{2t^2}\right)f 
\\
&= \int_{\abs{C}} \int_B \frac{(1+\eta)^{p-3}}{(1+\xi)^{3-q}}(1+\eta) \left(\frac{x_i}{2s}\partial_s+(p-2)\frac{x_i}{2(1+\eta)^2s^2} -\frac{x_i}{2t}\partial_t -(q-3)\frac{x_i}{2(1+\xi)^2t^2} \right)\phi^\sharp f
\end{align*}
We claim 
\begin{align*}
&\int_{\abs{C}}  \int_B  \frac{(1+\eta)^{p-2}}{(1+\xi)^{3-q}} \frac{x_i}{2s}\partial_s \phi^\sharp f \\
&=- \int_{\abs{C}}\int_B \frac{(1+\eta)^{p-2}}{(1+\xi)^{3-q}} \left( \frac{p-2}{(1+\eta)^2 }+ \frac{q-3}{(1+\xi)^2 } \right)\frac{x_i}{2\rho^2} \phi^\sharp f
-\int_{\abs{C}}  \int_B \frac{(1+\eta)^{p-2}}{(1+\xi)^{3-q}} \frac{x_i}{2\rho} \partial_t\phi^\sharp f .
\end{align*}
Applying this claim, we obtain, 
\begin{align*}
\int_C \left((s\partial_s+p-1)\frac{x_i}{2s^2}-(t\partial_t +q-1)\frac{x_i}{2t^2}\right) f
&= \int_{\abs{C}} \int_B \frac{(1+\eta)^{p-2}}{(1+\xi)^{3-q}} \left(\frac{x_i}{s}\partial_s+(p-2)\frac{x_i}{(1+\eta)^2s^2} \right)\phi^\sharp f,
\end{align*}
and we already shown that the right-hand side is equal to $\int_\kerorbit  \partial_{x^i} f.$

We will now prove the claim.   
We have $ \partial_\rho (g_{\mid s=t=\rho})=  ((\partial_s+\partial_t)g)_{\mid \rho=s=t}$. Therefore 
\begin{align} \label{equation claim}
\begin{aligned}
&\int_{\abs{C}}  \int_B \left((1+\eta)^{p-2}(1+\xi)^{q-3} \frac{x_i}{2s}\partial_s \phi^\sharp f\right)_{\mid s=t=\rho} 
\\
&=\int_{\abs{C}}  \int_B \left((1+\eta)^{p-2}(1+\xi)^{q-3} \frac{x_i}{2s}\right)_{\mid s=t=\rho}  \left(\partial_\rho( \phi^\sharp f_{\mid s=t=\rho})-(\partial_t\phi^\sharp f)_{\mid s=t=\rho} \right).
\end{aligned}
\end{align}
We will integrate by parts with respect to $\partial_\rho$ and use the fact that by  Lemma~\ref{Lemma lim rho} the boundary terms are zero. Then, using \begin{align*}
\partial_\rho (1+\eta)^{p-2}& = (p-2)(1+\eta)^{p-4} \frac{\theta^2}{2\rho^3}, 
\quad \partial_\rho (1+\xi)^{q-3} = -(q-3)(1+\xi)^{q-5} \frac{\theta^2}{2\rho^3} ,  \quad \partial_\rho (\frac{x_i}{\rho}) = 0,
\end{align*} we obtain 
\begin{align*}
&\int_{\abs{C}}  \int_B \left((1+\eta)^{p-2}(1+\xi)^{q-3} \frac{x_i}{2s}\right)_{\mid s=t=\rho} \partial_\rho( \phi^\sharp f_{\mid s=t=\rho})
\\
&= \int_{\abs{C}}\int_B \frac{(1+\eta)^{p-2}}{(1+\xi)^{3-q}} \left(-(p+q-5) - (p-2)\tfrac{\theta^2}{2(1+\eta)^2 \rho^2} + (q-3)\tfrac{\theta^2}{2(1+\xi)^2 \rho^2} \right)\frac{x_i}{2\rho^2} (\phi^\sharp f_{\mid s=t=\rho}).
\end{align*}
Combining this with \eqref{equation claim} proves the claim. 

For $z_i=y_i$, the proof is completely similar.
Finally for $\theta_i$, the proposition can be shown in the same spirit using Lemma~\ref{Lemma: properties of phi sharp} and integration by parts with respect to $\theta_i$.
\end{proof}

\begin{proof}[Proof of Proposition~\ref{Proposition properties integral}, part (1)-(3) ]
From Lemma~\ref{Lemma: properties of phi sharp} we obtain that \[ \phi^\sharp R^2 = ((1+\eta)^2s^2 - (1+ \xi)^2  t^2 + \theta^2 )\phi^\sharp= (s^2-t^2) \phi^\sharp.\]
So $\phi^\sharp R^2{}_{\mid s=t} =0$ and
we conclude $\int_\kerorbit R^2 f =0$. This proves part $(1)$ of the proposition.

The operators $L_{i,j} := \x_i \partial_j -(-1)^{\abs{i}\abs{j}} \x_j \partial_i$ for $i \leq j $  span $\mathfrak{osp}(p-1,q-1|2n)$. We can rewrite the operator $L_{i,j}$ as follows
\begin{align*}
L_{i,j}f
& = (-1)^{\abs{i}\abs{j}} \partial_j (\x_i f) - (-1)^{\abs{i}\abs{j}}  \beta_{ji} f - \partial_i (\x_j f) +  \beta_{ij} f 
= (-1)^{\abs{i}\abs{j}} \partial_j (\x_i f)  - \partial_i (\x_j f).
\end{align*}
Using Lemma~\ref{Lemma: integration of a derivative}, we can then compute that
\begin{align*}
\int_\kerorbit L_{i,j} f &= \int_\kerorbit (-1)^{\abs{i}\abs{j}} \partial_j(\x_i f) - \partial_i (\x_j f) 
=0.
\end{align*}
 This finishes the proof of part (2).
For part (3), we use \[ \mE f= \sum_k \x^k \partial_k = \sum_k (-1)^{\abs{k}} \partial_k  (\x^k f) - M f.\]
Hence, again using Lemma~\ref{Lemma: integration of a derivative}, we find  
\begin{align*}
\int_\kerorbit  \mE f &= \sum_k (-1)^{\abs{k}}\int_\kerorbit(s\partial_s +p-1)\frac{\x_k \x^k}{2s^2} f- (t\partial_t + q-1) \frac{\x_k \x^k}{2t^2} f - \int_\kerorbit M f 
\\
& =  \int_\kerorbit R^2 \left((s\partial_s +p-1)\frac{1}{2s^2}f -(t\partial_t + q-1) \frac{1}{2t^2}f\right) 
 + \int_\kerorbit  \frac{s^2}{s^2}f+\frac{t^2}{t^2} f   - \int_\kerorbit M f  \\
&= (2-M) \int_\kerorbit f,
\end{align*}
where we used part $(1)$ to eliminate the term with $R^2$. 
\end{proof}

For the Laplacian and the Bessel operators we have the following.
\begin{lemma}  \label{Lemma: int Laplacian and Bessel operator}
For $\lambda = -M+2$ we have 
\begin{align*}
\int_\kerorbit \Delta f = -(M-4) \int_\kerorbit (s\partial_s+p-1)\frac{f}{2s^2}-(t\partial_t +q -1 )\frac{f}{2t^2}
\quad\text{and}\quad
\int_\kerorbit \cB_\lambda(\x_k) f =0.
\end{align*}

\end{lemma}
\begin{proof}
For the Laplacian this follows from applying Lemma~\ref{Lemma: integration of a derivative} and Proposition~\ref{Proposition properties integral}(3) to $\int_\kerorbit \sum_{i,j} \beta^{ij} \partial_i \partial_j f$, while for the Bessel operator $\cB_\lambda(\x_k)= (-\lambda+2\mE)\partial_{k} - \x_k \Delta$, we first use Proposition~\ref{Proposition properties integral}(3) and the fact that $[\Delta, e_k ]= 2 \partial_k $ to obtain \[
\int_\kerorbit \cB_\lambda (\x_k) f = (-M+2) \int_\kerorbit \partial_k f - \int_\kerorbit \Delta (\x_k f) + 2 \int_\kerorbit \partial_k f
\]
which can be shown to be zero by Lemma~\ref{Lemma: integration of a derivative} and the expression we found for the Laplacian.
\end{proof}
Now we can prove the final part of Proposition~\ref{Proposition properties integral}.
\begin{proof}[Proof of Proposition~\ref{Proposition properties integral}, part (4)]
We first remark that, using Lemma~\ref{Lemma: integration of a derivative}, we obtain 
\begin{align*}
\int_\kerorbit & \partial_k (\mE(f) g ) 
- \sum_{i,j}(-1)^{\abs{j}(\abs{i}+\abs{k})}\partial_j \left(\x_k g^{ij} \partial_i (f)  g\right) = 0.
\end{align*}
Combining this with part (3) of Proposition~\ref{Proposition properties integral}, we then find 
\begin{align*}
\begin{aligned}
 \int_\kerorbit   \left( 2 (-1)^{\abs{f}\abs{k}}\mE(f) \partial_k (g) 
 + 2 \partial_k (f) \mE(g)  -\sum_{i,j} 2 \x_k (-1)^{\abs{f}\abs{j}} g^{ij}  \partial_i(f) \partial_j(g) \right) &=  - 2 \int_\kerorbit (\cB_\lambda(\x_k)f) g. 
\end{aligned}
\end{align*}
Together with Lemma~\ref{Lemma: int Laplacian and Bessel operator} and the product rule given in \eqref{Eq Product rule for Bessel operators}, we can then conclude
\[
\int_\kerorbit (\cB_\lambda(\x_k) f) g - (-1)^{\abs{f}\abs{k}} f (\cB_\lambda(\x_k)g)=0, 
\]
which proves part (4) of the proposition.
\end{proof}
\subsection{The sesquilinear form}

Define a sesquilinear form on the minimal orbit $C$ using the functional $\int_\kerorbit $
\[
\langle f,g\rangle := \int_\kerorbit \overline{f} g .
\]

\begin{theorem} \label{Theorem skew-symmetric}
 Suppose $\nu\not\in-2\mN,$ $\mu+\nu$ even and $\mu+\nu= p+q-2n-6\geq 0$.
The representation $\pi_C$ on $W$ is skew-symmetric for the  form $\langle \cdot,\cdot \rangle$, i.e.\ for $X \in \TKK(J)$, and $f,g$ in $W$
\[
\langle \pi_C(X)f , g \rangle +(-1)^{\abs{X}\abs{f}} \langle f, \pi_C(X) g \rangle=0.
\]
\end{theorem}
\begin{proof}
The theorem follows easily from Proposition~\ref{Proposition properties integral}. So we have to show that we can apply this proposition.
We will prove
\begin{align}\label{Eq: inclusion W}
W \subset \sum_{a=0}^{\tfrac{\mu-\nu}{2}}\sum_{b=0}^\infty  \widetilde{K}_{\frac{\nu}{2}+a+b}(\abs{X}) \otimes \cP_{\geq a+2b} (\mR^{p+q-2|2n}).
\end{align}
If $f$ and $g$ are in the right-hand side, then $\overline{f}g$ is a linear combination of elements of the form $P_k \widetilde{K}_{\frac{\nu}{2}+a+b}(\abs{X}) \widetilde{K}_{\frac{\nu}{2}+a'+b'}(\abs{X})$, where $P_k$ is a homogeneous polynomial of degree $k$ with $k\geq a+a'+2b+2b'$ and $a,a' \leq \frac{\mu-\nu}{2}.$
We have \begin{align*}
\mu+\nu+1 &> \max(\mu+\nu,0) \geq \max(\nu+a,0) + \max(\nu+a',0) 
\\
&\geq \max(\nu+a,-a-2b)+\max(\nu+a,-a'-2b').
\end{align*}
Hence \[\mu+\nu+1 +k > \max (\nu+2a+2b,0) + \max(\nu+2a' +2b',0), \] and $\overline{f}g$ satisfies Proposition~\ref{Proposition properties integral}.

To see \eqref{Eq: inclusion W}, note that $W_0 = U(\mk) \widetilde{K}_{\frac{\nu}{2}}(\abs{X})$ is  contained in the right-hand side. Using the differential relation of equation~\eqref{Bessel function diff rec rel}, 
 we obtain  $\mE (\widetilde{K}_\alpha(\abs{X})) = -\frac{\abs{X}^2}{2} \widetilde{K}_{\alpha+1}(\abs{X})$. Therefore the right-hand side is invariant for the action of $U(L_e)$, the associative algebra generated by powers of $L_e$. It is also clearly invariant for $U({J}^{\minus})$ which acts by multiplication with polynomials. By the Poincar\'{e}--Birkhoff--Witt theorem $W=U(\mg)\widetilde{K}_{\frac{\nu}{2}}=U({J}^{\minus})U(L_e)U(\mk)\widetilde{K}_{\frac{\nu}{2}}$, hence equation~\eqref{Eq: inclusion W} follows.
 \end{proof}
We can use this skew-symmetry to show non-degeneracy of our form.
\begin{lemma} Assume $\nu \not\in -2\mN$, $p\not=3$, $q\not=3$, $p+q$ even and $p+q-2n-6 \geq 0$.
The form $\langle, \rangle$ defines a sesquilinear, non-degenerate form on $W$, which is superhermitian, i.e.\ 
\[ \langle f,g\rangle = (-1)^{\abs{f}\abs{g}} \overline{\langle g,f\rangle}.
\]
\end{lemma}
\begin{proof}
We see immediately that our form is sesquilinear and superhermitian. From Theorem~\ref{Theorem skew-symmetric}, it follows that the radical of the form gives a subrepresentation. Namely if $\langle f, g\rangle =0 $ for all $g$ in $W$, then also 
\[
\langle \pi_C(X) f , g\rangle =- (-1)^{\abs{f}\abs{X}} \langle f, \pi_C(X) g \rangle =0, \qquad \text{ for all } g \in W.
\]
So $\pi_C(X) f$ is also contained in the radical. By Corollary~\ref{Corollary W simple} $W$ is simple for $\mu+\nu\geq0$, hence the radical is zero or the whole $W$. Since $\int_C \widetilde{K}_{\frac{\nu}{2}}(\abs{X})\widetilde{K}_{\frac{\nu}{2}}(\abs{X})\not=0$ by Lemma~\ref{Integral of Knu} we conclude that the radical is zero and the form is non-degenerate.
\end{proof}

\section{Open questions}
We end this paper by mentioning some open questions and possibilities for future research.

\subsection{Density of $W$ in $\Gamma(\cO_C)$.} From \cite[Theorem 6.2.1]{NS},  we know a priori that our representation is not unitary. However, we can still define a Hilbert superspace on the minimal orbit in the sense of the new definition introduced in \cite{Michel}.  Namely,  we can `pullback' the Hilbert superspace  $L^2(\mR^+,\rho^{p+q-5}\diff\rho) \hat{\otimes} L^2(\mathbb{S}^{p-2}) \hat{\otimes} L^2(\mathbb{S}^{q-2})$ using the isomorphism $\phi^\sharp$ defined in Section~\ref{Section integration} to obtain a Hilbert superspace $H$ on the minimal orbit. This defines a topology on $H$ and $W$ is contained in $H$. So a natural question to ask is if $W$ is dense in $H$ with respect to this topology. We note that we were not able to show continuity of the operators in our representation with respect to this topology. This has to do with the fact that in the supercase isometrical operators are not necessarily continuous and that our operators do not respect the fundamental decomposition of $H$. So in particular we cannot use the standard techniques for integrating a $(\mg,\mathfrak{k})$-representation to a representation of the corresponding group on $H$. 
However, one could still investigate if there is any connection between the Fr\'{e}chet space on which we defined the representation of $OSp(p,q|2n)$ and the Hilbert superspace $H$. 

\subsection{Characterisation of $W$.} At the moment our definition of $W$ looks a bit arbitrary. In particular it depends on our choice of intermediate algebra $\mk= \mathfrak{osp}(p|2n)\oplus \mathfrak{so}(q)$. Therefore an intrinsic characterisation of $W$, which could be generalized to other Lie superalgebras, would be interesting. In the classical case, $W$ is simply the space of $\mathfrak{k}$-finite vectors in the representation $\pi_C$ of $\mathfrak{g}$ on $\mathcal{C}^\infty(C)$, but a statement like this is not immediate in the supercase.

\subsection{Other minimal representations} Another natural direction of study is to use the approach developed in this paper for other Lie superalgebras corresponding to simple Jordan superalgebras. In particular for (real forms of) the Jordan superalgebra $osp_\mC(m|2n)_+$, one would expect to find the metaplectic representation of (real forms of) $SpO(4n|2m)$ as constructed in \cite{Michel}.


\appendix
\section{The affine superspace and supermanifolds}\label{supermanifolds}
A general introduction to supermanifolds can be found in \cite{DM} and \cite{CCF}. Here we quickly introduce definitions and notations.

Consider a topological space $\abs{M}$. We associate a category $\mathcal{C}_{\abs{M}}$ with it as follows. The objects of $\mathcal{C}_{\abs{M}}$ are the open sets of $\abs{M}$ and its morphisms are the inclusions. So if $U \subset V$ for $U$ and $V$ open sets in $\abs{M}$, then there exists a unique morphism from $U$ to $V$.

A {\it presheaf} (of superrings) on $\abs{M}$ is a contravariant functor $\mathcal{O}$ from the category $\mathcal{C}_{\abs{M}}$ to the category of superrings. This means that there corresponds a superring $\mathcal{O}(U)$ to each open set $U$ in $\abs{M}$ and that there exists a morphism $r_{U,V} \colon \mathcal{O}(V) \to \mathcal{O}(U) $ if $U \subset V$. These morphisms satisfy 
$r_{U,U}= \id$ and $r_{U,V} \circ r_{V,W} = r_{U,W}$
for $U \subset V \subset W$. We will often write $r_{U,V}(f)$ as $f_{\mid U}$ for a section $f$ in $\mathcal{O}(V)$.

A presheaf $\mathcal{O}$ on $\abs{M}$ is a {\it sheaf} if it has the following gluing property. 
Consider an open set $U$ in $\abs{M}$ and an open covering $\{ U_i \}_{i \in I}$. Assume we have a family $\{f_i\}_{i\in I} $ of sections $f_i \in \mathcal{O}(U_i)$ for which ${f_i} {\mid_{U_i \cap U_j}} = {f_j} {\mid_{U_i \cap U_j}}$ for all $i, j \in I$. Then the gluing property asserts that there exists a unique $f$ in $\mathcal{O}(U)$ such that $f {\mid_{U \cap U_i}} = {f_i} {\mid_{U \cap U_i}}$ for all $i$.

For every $x$ in $\abs{M}$ we can define the stalk $\mathcal{O}_{x}$ as the direct limit \[ \lim_{\longrightarrow}  \mathcal{O}(U),\]
where we take the limit over all open neighbourhoods $U$ of $x$.
The idea is that the stalk captures the behaviour of the sheaf locally around the point $x$. It consists of sections defined on some neighbourhood of $x$ and sections are considered equivalent if their restrictions on a smaller neighbourhood agree.

\begin{definition}
A superringed space $(\abs{S}, \cO_S)$ is a topological space $\abs{S}$ and a sheaf $\cO_S$ of superrings.  \\
A superspace $(\abs{S}, \cO_S)$ is a superringed space for which the stalk $\cO_{S,x}$ is a local superring for all points $x \in \abs{S}$.
\end{definition}
A superring is {\it local} if it has a unique maximal ideal. 

Let $V$ be a real finite-dimensional super-vector space. Then the affine superspace is the superringed space
\[
\mA(V)= (V_{\oa}, \cC^{\infty}_{V_{\oa}} \otimes_{\mR} \Lambda V_{\ob}^\ast),
\]
where $\cC^\infty_{V_{\oa}}$ is the sheaf of smooth, complex-valued functions on $V_{\oa}$ and $ \Lambda V_{\ob}^\ast$ is the Grassmann algebra  of $V_{\ob}^\ast$.
 In case $V=\mR^{m|n}$ we also use the notation $\mA^{m|n}$ for $\mA(\mR^{m|n})$.
 
 A morphism $\phi=(\abs{\phi}, \phi^\sharp)$ between two superspaces $M$ and $N$ is a continuous map $\abs{\phi}\colon \abs{M} \to \abs{N}$ and a sheaf morphism $\phi^\sharp\colon \cO_N \to  \abs{\phi}_\ast \cO_M$. Here $\abs{\phi}_\ast \cO_M $ is the sheaf on $\abs{N}$ given by $\abs{\phi}_\ast \cO_M (U) =  \cO_M( \abs{\phi}^{-1} (U))$. 
 
A (real smooth) supermanifold $M$ is a superspace  that is locally isomorphic to $\mA^{m|n}$. We denote the underlying topological space by $\abs{M}$ and the structure sheaf of commutative superrings by  $\cO_M$. The global sections are denoted by $\Gamma(\cO_M)$. If $M$ is an ordinary manifold, then we will also use the notation $\cC^\infty(M)$ for $\Gamma(\cO_M)$.  Note that for supermanifolds the global sections $\Gamma(\cO_M)$ determine the sheaf $\cO_M$, \cite[Corollary 4.5.10]{CCF}.

The elements in $\cO_M(U)$ act by multiplication on $\cO_M(U)$ and they form the differential operators of degree zero. The differential operators of degree $k$ are defined inductively:
\[
\cD_M^k(U) := \{ D \in \End(\cO_M(U)) \mid [D,f] \in \cD_M^{k-1}(U)\quad   \forall f \in \cO_M(U) \}.
\]
Here $[D,f]= D f - (-1)^{\abs{D}\abs{f}} f D$  is the supercommutator. 
The sheaf of differential operators  $\cD_M$ is then defined by
\[
\cD_M(U) = \bigcup^{\infty}_{i=0} \cD^k_M(U).
\]
We again use  the notation $\Gamma(\cD_M)$ for the global sections.

The product of supermanifolds $M$ and $N$ is given by
\[
M \times N = (\abs{M}\times \abs{N}, \cO_{M\times N}),
\]
where $ \cO_{M\times N}(U\times V) := \cO_{M}(U) \hat{\otimes} \cO_N(V),$
for an open set $U\times V \in \abs{M}\times \abs{N}$. Here $\hat{\otimes}$ is the completion of the tensor product with respect to the projective tensor topology. This is the unique topology such that 
\[
C^\infty(U) \hat{\otimes} C^\infty(V) \cong C^\infty (U \times V) 
\]
for $U\subset \mR^m$ and $V \subset \mR^n$, \cite[Section 4.5]{CCF}.

The following proposition tells us that for most practical purposes it is sufficient to only consider the tensor product.
\begin{proposition}[{\cite[Proposition 4.5.4]{CCF}}]\
\begin{enumerate}
\item The space of sections $\Gamma(\cO_M) \otimes \Gamma(\cO_N)$ is dense in $\Gamma(\cO_{M\times N})$.
\item If $\phi_i\colon M_i \to N_i$,$ i=1,2$ are supermanifold morphisms, then the sheaf morphism of the map $\phi_1\times \phi_2 \colon M_1 \times M_2 \to N_1 \times N_2$ is given by $\phi_1^\sharp \hat{\otimes} \phi_2^\sharp$ which is in turn completely determined by $\phi_1^\sharp \otimes \phi_2^\sharp.$
\end{enumerate}
\end{proposition}
Let $\theta_1,\ldots,  \theta_n$ be a basis of $V_{\ob}^\ast$. For a multi-index $I=(i_1,i_2, \ldots, i_n) \in \mZ_2^{n}$, we introduce the notation $\theta^{I}:=\theta_1^{i_1}\theta_2^{i_2}\cdots \theta_n^{i_n}$. Then we can decompose every section $f \in \cO_{\mA(V)}(U)$ for $U$ an open subspace of $V_{\oa}$ as
\[
f= f_0 + \sum_{I \in \mZ_2^n \backslash \{0\} } f_I \theta^I,
\]
where $f_0, f_I$ are in $\cC^{\infty}(U).$
The value of $f$ at a point $x$ in $V_{\oa}$ is defined as \[ f(x):=\ev_x(f):=f_0(x).\]
Note that $\ev_x(fg) = \ev_x(f)\ev_x(g)$ for $f, g$ in $\cO_{\mA(V)}(U)$.
\section{Gegenbauer polynomials, Bessel functions and generalized Laguerre functions} \label{Section Appendix}
We will also need some orthogonal polynomials and special functions which we introduce here. 
\subsection{Gegenbauer polynomials}\label{Sect Gegenbauer polynomials}
For $n \in \mN$ and $\lambda \in \mC$, we define the Gegenbauer polynomial 
\[
C^\lambda_n(z) = \frac{1}{\Gamma(\lambda)} \sum_{k=0}^n \frac{(-1)^k \Gamma(\lambda+k) \Gamma(n+2\lambda+k)}{k! (n-k)!\Gamma(2\lambda+2k)}\left(\frac{1-z}{2}\right)^k.
\]
We will use the normalised version 
\begin{align*}
\widetilde{C}^\lambda_n(z) = \Gamma(\lambda) C^\lambda_n(z), 
\end{align*}
which, in contrast to $C^\lambda_n(z)$, is non-zero for $\lambda=0$.
We need the following two properties of the normalised Gegenbauer polynomial, \cite[3.15(21) and 3.15(30)]{EMOT}:
\begin{align*}
\partial_z \widetilde{C}^\lambda_m (z) = 2 \widetilde{C}^{\lambda+1}_{m-1}(z),
\end{align*}
and
\begin{align*}
4(1-z^2) \widetilde{C}^{\lambda+1}_{m-1}(z)-2z (2\lambda-1)\widetilde{C}^{\lambda}_{m}(z) 
=-(m+1)(2\lambda+m-1)& \widetilde{C}^{\lambda-1}_{m+1}(z).
\end{align*}
\subsection{Bessel functions} \label{Section Bessel functions}
The modified Bessel function of the first kind or $I$-Bessel function is defined, for $z>0$ and $\alpha \in \mC$, by 
\[
I_\alpha(z) := \left(\frac{z}{2}\right)^\alpha \sum_{n=0}^\infty \frac{1}{n!\Gamma 	(n+\alpha+1)} \left(\frac{z}{2}\right)^{2n}
\]
and the modified Bessel function of the third kind or $K$-Bessel function by
\[
K_\alpha(z) := \frac{\pi}{2 \sin (\pi \alpha)} (I_{-\alpha}(z) -I_\alpha (z)),
\]
see \cite[Section 3.7]{Wat}.
We will use the following renormalisations
\begin{align*}
\widetilde{I}_\alpha(z) := \left(\frac{z}{2}\right)^{-\alpha } I_\alpha(z), 
\qquad \widetilde{K}_\alpha(z) :=\left(\frac{z}{2}\right)^{-\alpha } K_\alpha(z).
\end{align*}
The functions $\widetilde{I}_\alpha(z)$ and $ \widetilde{K}_\alpha(z)$  are linearly independent and solve the following second order differential equation 
\begin{align} \label{Bessel function dif rel}
z^2 \frac{\diff^2 u}{\diff z^2} + (2\alpha+1) z \frac{\diff u}{\diff z} -z^2 u =0.
\end{align}
We also have the differential recurrence relations, \cite[III.71 (6)]{Wat}
\begin{align} \label{Bessel function diff rec rel}
\frac{\diff }{\diff z }\widetilde{I}_\alpha(z) = \frac{z}{2}\widetilde{I}_{\alpha+1}(z), \qquad 
\frac{\diff }{\diff z }\widetilde{K}_\alpha(z) = -\frac{z}{2}\widetilde{K}_{\alpha+1}(z).
\end{align}
Using these relations we can rewrite the second order differential equation as a recurrence relation, \cite[III.71 (1)]{Wat},
\begin{align} \label{Bessel function rec rel}
\begin{aligned}
\frac{z^2}{4} \widetilde{I}_{\alpha+1}(z) + \alpha \widetilde{I}_\alpha(z) - \widetilde{I}_{\alpha-1}(z) &=0, 
\qquad 
\frac{z^2}{4} \widetilde{K}_{\alpha+1}(z) - \alpha \widetilde{K}_\alpha (z)- \widetilde{K}_{\alpha-1}(z) &=0.
\end{aligned}
\end{align}

The asymptotic behaviour of the $K$-Bessel function is given by, \cite[Chapter III and VII]{Wat},
\begin{align*}
\text{ for } x\to 0:\quad \widetilde{K}_\alpha(x)&  = 
\begin{cases}
\frac{\Gamma(\alpha)}{2} (\frac{x}{2})^{-2\alpha} + o(x^{-2\alpha})\qquad \quad\text{ if } \alpha >0 \\
-\log(\frac{x}{2}) + o(\log(\frac{x}{2})) \qquad\quad\;\;\text{ if } \alpha=0 \\
\frac{\Gamma(-\alpha)}{2} + o(1) \qquad \qquad\qquad\quad \text{ if } \alpha<0. 
\end{cases} \\
\text{ for } x\to \infty:\quad  \widetilde{K}_\alpha(x) &=  \frac{\sqrt{\pi}}{2}\left(\frac{x}{2}\right)^{-\alpha -\frac{1}{2}} e^{-x} \left(1 + \cO\left(\frac{1}{x}\right)\right).
\end{align*}

\subsection{Generalised Laguerre functions} \label{definition Lambda mu nu}
Consider the generating function 
\[
G^{\mu,\nu}_2(t,x) :=  \frac{1}{(1-t)^{\frac{\mu+\nu+2}{2}}} \widetilde{I}_{\frac{\mu}{2}}\left(\frac{tx}{1-t}\right)\widetilde{K}_{\frac{\nu}{2}}\left(\frac{x}{1-t}\right),\quad  \text{ for parameters } \mu,\nu \in \mC.
\]

This function $G^{\mu,\nu}_2$ is holomorphic near $t=0$. We will define the generalised Laguerre functions $\Lambda^{\mu,\nu}_{2,j}(x)$ as the coefficients in the expansion 
\begin{align} 
G^{\mu,\nu}_2(t,x) = \sum_{j=0}^\infty \Lambda^{\mu,\nu}_{2,j}(x) t^j.
\end{align}
 Note that $\Lambda^{\mu,\nu}_{2,0}(x) =\frac{1}{\Gamma(\frac{\mu+2}{2})} \widetilde{K}_{\frac{\nu}{2}}(x)$. For notational convenience we set $\Lambda^{\mu,\nu}_{2,j} = 0$ for $j<0$.
We have some relations between the generating functions, which in turn lead to corresponding differential recurrence relations for the $\Lambda^{\mu,\nu}_{2,j}$.
\begin{proposition}\label{Prop generating functions}
The generating functions satisfy
\begin{align*}
\partial^2_x G^{\mu,\nu}_2 (x,t) + \frac{(\nu+1)}{x} \partial_x G^{\mu,\nu}_2(x,t) -   G^{\mu,\nu}_2(x,t) 
&= t(\mE_t +\mu+2) G^{\mu+2,\nu}_2 (x,t),
 \\
\partial^2_x G^{\mu,\nu}_2 (x,t) + \frac{(\mu+1)}{x} \partial_x G^{\mu,\nu}_2 (x,t)  -   G^{\mu,\nu}_2(x,t) &= -(\mE_t +\frac{\mu-\nu}{2}) G^{\mu,\nu+2}_2 (x,t),
 \\
 t \left(\mu(\mu +  \nu+ 2  \mE_x) G_2^{\mu,\nu}(x,t) + x^2 (t \right. \left. (\mE_t +\mu+2)) G_2^{\mu+2,\nu}(x,t)\right)  &= 4 \mE_t G_2^{\mu-2,\nu}(x,t), 
 \\
 -\nu(\mu+ \nu +  2  \mE_x ) G_2^{\mu,\nu}(x,t) + x^2  (\mE_t  +\frac{\mu-\nu}{2}) G_2^{\mu,\nu+2}(x,t)  &= (4 \mE_t +2(\mu+\nu) )G_2^{\mu,\nu-2}(x,t) ,
\end{align*}
where $\mE_x = x \partial_x $ and $\mE_t = t \partial_t$.
\end{proposition}
\begin{proof}
First one uses the differential recursion relations for the Bessel functions, equation~\eqref{Bessel function diff rec rel}, to calculate
\begin{align*}
\partial_x  G^{\mu,\nu}_2(x,t) &= \tfrac{x}{2(1-t)}(t^2 G^{\mu+2,\nu}_2(x,t) - G^{\mu,\nu+2}_2(x,t)),
 \\
\partial_x^2   G^{\mu,\nu}_2(x,t) &= \tfrac{x^2t^4}{4(1-t)^2} G^{\mu+4,\nu}_2(x,t)+\tfrac{t^2}{2(1-t)}G^{\mu+2,\nu}_2(x,t)
 -\tfrac{x^2 t^2 }{2(1-t)^2} G^{\mu+2,\nu+2}_2(x,t) 
 \\
&\quad -\tfrac{1}{2(1-t)}G^{\mu,\nu+2}_2(x,t) + \tfrac{x^2}{4(1-t)^2}G^{\mu,\nu+4}_2(x,t), 
\\
\partial_t    G^{\mu,\nu}_2(x,t) &
= \tfrac{\mu+\nu+2}{2(1-t)} G^{\mu,\nu}_2(x,t) + \tfrac{x^2 t}{2(1-t)^2}G^{\mu+2,\nu}_2(x,t) - \tfrac{x^2}{2(1-t)^2} G^{\mu,\nu+2}_2(x,t).
\end{align*}
From the recurrence relations~\eqref{Bessel function rec rel} for the Bessel functions, we get the following recurrence relations for $G^{\mu,\nu}_2(x,t)$
\begin{align*}
\frac{x^2}{4(1-t)} G^{\mu,\nu+2}_2(x,t)-\frac{\nu}{2} G^{\mu,\nu}_2(x,t) - \frac{1}{1-t}G^{\mu,\nu-2}_2(x,t) &= 0, 
\\
\frac{x^2t^2}{4(1-t)} G^{\mu+2,\nu}_2(x,t) + \frac{\mu }{2} G^{\mu,\nu}_2(x,t) - \frac{1}{1-t} G^{\mu-2,\nu}_2(x,t)&=0.
\end{align*}
We can combine these relations with the expressions for the partial derivatives to obtain the proposition. 
\end{proof}

\begin{corollary}\label{Properties generalised Laguerre polynomials}
The generalised Laguerre functions satisfy
\begin{align}\label{Properties gen Lag poly 1}
\partial^2_x \Lambda^{\mu,\nu}_{2,j}(x) + \tfrac{(\nu+1)}{x} \partial_x \Lambda^{\mu,\nu}_{2,j}(x) -   \Lambda^{\mu,\nu}_{2,j}(x)&= (j+\mu+1)\Lambda^{\mu+2,\nu}_{2,j-1}(x) \\
\partial^2_x \Lambda^{\mu,\nu}_{2,j}(x) + \tfrac{(\mu+1)}{x} \partial_x \Lambda^{\mu,\nu}_{2,j}(x) -   \Lambda^{\mu,\nu}_{2,j}(x)&= -(j+\tfrac{\mu-\nu}{2})\Lambda^{\mu,\nu+2}_{2,j}(x) \nonumber \\
\mu(\mu+\nu + 2 \mE_x )\Lambda^{\mu,\nu}_{2,j} + (j+\mu+1) x^2 \Lambda^{\mu+2,\nu}_{2,j-1} &= 4(j+1) \Lambda^{\mu-2,\nu}_{2,j+1} \nonumber\\
\nu(\mu+\nu + 2 \mE_x )\Lambda^{\mu,\nu}_{2,j} + (-j-\tfrac{\mu-\nu}{2}) x^2 \Lambda^{\mu,\nu+2}_{2,j} &=-4(j+\tfrac{\mu+\nu}{2}) \Lambda^{\mu,\nu-2}_{2,j}. \nonumber
\end{align}
\end{corollary}
\begin{proof}
The corollary follows from Proposition~\ref{Prop generating functions} and the definition of $\Lambda^{\mu,\nu}_{2,j}$  in equation~\eqref{definition Lambda mu nu} as coefficients in the expansion of $G^{\mu,\nu}_2(x,t)$.
\end{proof}
In general we do not have of an explicit expression for the functions $\Lambda^{\mu,\nu}_{2,j}(x)$. However for our purposes it is sufficient to know when they are non-zero.
\begin{corollary}\label{Generalised Laguerre functions are non-zero}
Assume $\mu \not \in - \mN$ or $\mu+j\geq 0$. Then on every open interval the function $\Lambda^{\mu,\nu}_{2,j} $ is different from zero for $j \in \mN$. 
\end{corollary}
\begin{proof}
Suppose $\Lambda^{\mu,\nu}_{2,j}(x)=0$ for all $x \in I$, with $I\subset \mR^+$ an open interval. Then from \eqref{Properties gen Lag poly 1} it would follow that also $(j+\mu+1)\Lambda^{\mu+2,\nu}_{2,j-1}(x)=0$. Since $\mu+1+j\not=0$, we obtain $\Lambda^{\mu+2,\nu}_{2,j-1}(x)=0$. This would again lead to $\Lambda^{\mu+4,\nu}_{2,j-2}(x)=0$ and so on. Finally we get $\Lambda^{\mu+2j,\nu}_{2,0}(x)=0$. This is a contradiction since $\Lambda^{\mu+2j,\nu}_{2,0}(x)= \frac{1}{\Gamma(\frac{\mu+2j+2}{2})}\widetilde{K}_{\frac{\nu}{2}}(x)$ and the Bessel function is different from zero on $I$. 
\end{proof}
We also use the following recursion relation. 
\begin{proposition} \label{action Le on Laguerre functions}
For $\mu,\nu \in \mC$, we have for $j\in\mZ$
\begin{align*}
&(2j+\mu+1) \left(\mE_x +\tfrac{\mu+\nu+2}{2}\right) \Lambda_{2,j}^{\mu,\nu}(x)\\& = (j+1)(j+\mu+1) \Lambda^{\mu,\nu}_{2,j+1}(x) - \left(j +\tfrac{\mu+\nu}{2}\right)\left(j + \tfrac{\mu-\nu}{2}\right) \Lambda^{\mu,\nu}_{2,j-1}(x). 
\end{align*}
For $j=0$, we have
\begin{align*}
 \left(\mE_x +\tfrac{\mu+\nu+2}{2}\right) \Lambda_{2,0}^{\mu,\nu}(x) =  \Lambda^{\mu,\nu}_{2,1}(x), 
\end{align*}
even for $\mu=-1$.
\end{proposition}
\begin{proof}
This is \cite[Proposition 3.6.1]{Mollers} and \cite[Example 3.3.1]{Mollers}.
\end{proof}

\noindent
SB: Department of Mathematical Analysis, Ghent University, Krijgslaan 281, 9000 Gent, Belgium;
E-mail: {\tt Sigiswald.Barbier@UGent.be}

JF:  Department Mathematik, FAU Erlangen-N\"{u}rnberg, Cauerstr. 11, 91058 Erlangen, Germany;
E-mail: {\tt frahm@math.fau.de}

\date{}

\end{document}